\documentclass[12pt]{amsart}

\usepackage{amsmath, amssymb} 
\usepackage{graphicx} 
\usepackage{wrapfig}

\newtheorem{theorem}{Theorem}[section]
\newtheorem{proposition}[theorem]{Proposition}

\numberwithin{equation}{section}

\begin{document} 
\title[]{2-knot homology and Roseman move} 

\author[]{Hiroshi MATSUDA} 

\address{
Faculty of Science, 
Yamagata University, Yamagata 990-8560, JAPAN}

\email{matsuda@sci.kj.yamagata-u.ac.jp} 

\thanks{The author is partially supported by 
JSPS KAKENHI Grant number 15K04865.} 

\maketitle

\begin{abstract} 
Ng constructed an invariant of knots in ${\mathbb{R}}^3$, 
a combinatorial knot contact homology. 
Extending his study, 
we construct an invariant of surface-knots in ${\mathbb{R}}^4$ 
using diagrams in ${\mathbb{R}}^3$. 
\end{abstract}

\section{Introduction}

Topological invariants of knots in ${\mathbb{R}}^3$ are 
constructed by Ng \cite{ng1}, \cite{ng2}, \cite{ng-frame}, \cite{ng3}, 
in a combinatorial method. 
These invariants are 
equivalent to 
the {\it knot contact homology}, 
constructed by Ekholm, Etnyre, Ng, Sullivan \cite{kch} 
in symplectic topology, 
and extended by Cieliebak, Ekholm, Latschev, Ng \cite{cieliebak}. 
The knot contact homology 
detects several classes of knots \cite{lidman-gordon}, and 
an enhancement of the knot contact homology 
is a complete invariant of knots \cite{ekholm-ng-shende}. 

A {\it surface}-{\it knot} $F$ is a closed connected oriented surface 
embedded locally flatly in ${\mathbb{R}}^4$. 
For a 
projection $\pi \colon {\mathbb{R}}^4 \to {\mathbb{R}}^3$, 
we may assume 
that the projection $\pi |_F$ is {\it generic}, 
that is, $\pi |_F$ has double points, 
isolated triple points and isolated branch points in the image 
as its singularities. 
A {\it diagram} of $F$ 
is a generic projection $\pi (F)$ 
equipped with a height information with respect to $\pi$. 
Extending Ng's 
construction of knot invariants, 
we define, in $\S$2 and $\S$10, 
differential graded algebras $(CR^{\varepsilon \delta}(D), \partial)$ 
associated with a diagram $D$ of 
$F$, where $\varepsilon, \delta \in \{ +, - \}$.

\begin{theorem} \label{main-theorem} 
Let $D_1, D_2$ denote diagrams of a surface-knot in ${\mathbb{R}}^4$. 
Then the differential graded algebra $(CR^{\varepsilon \delta} (D_1), \partial)$ 
is stably tame isomorphic to 
$(CR^{\varepsilon \delta} (D_2), \partial)$, 
where $\varepsilon, \delta \in \{ +, - \}$. 
\end{theorem}

Let $D$ denote a diagram of a surface-knot $F$ in ${\mathbb{R}}^4$. 
Theorem \ref{main-theorem} shows that 
the stably tame isomorphism class of 
$(CR^{\varepsilon \delta}(D), \partial)$ is an invariant of $F$, 
which we denote by $(CR^{\varepsilon \delta}(F), \partial)$. 
Therefore 
the homology of $(CR^{\varepsilon \delta}(F), \partial)$ is an invariant of $F$, 
which we denote by $HR^{\varepsilon \delta}(F)$ 
and call {\it Roseman homology}. 
In $\S \S$3--9, 
we give a proof of Theorem \ref{main-theorem} 
when $(\varepsilon, \delta) = (-, -)$. 
The cases when $(\varepsilon, \delta) = (-, +), (+, -), (+, +)$ 
are proved similarly. 
In $\S$11, we show that at least three of the four differential graded algebras 
$(CR^{--}(F), \partial)$, $(CR^{-+}(F), \partial)$, 
$(CR^{+-}(F), \partial)$, $(CR^{++}(F), \partial)$ 
are distinct when $F$ is the 2-twist spun-trefoil, and that 
they distinguish the spun-trefoil \cite{artin} 
from the 2-twist spun-trefoil \cite{zeeman}. 

\begin{theorem} \label{2-twist} 
Let $F$ denote the 2-twist spun-trefoil in ${\mathbb{R}}^4$. 
Then 
$(CR^{--} (F), \partial)$ is not stably tame isomorphic to 
$(CR^{\varepsilon \delta} (F), \partial)$, 
where $(\varepsilon, \delta) = (-, +), (+, -), (+, +)$. 
Moreover,  
$(CR^{++} (F), \partial)$ is not stably tame isomorphic to 
$(CR^{-+} (F), \partial)$, 
$(CR^{+-} (F), \partial)$. 
\end{theorem}

\begin{theorem} \label{plus-minus} 
The tuple of four differential graded algebras 
$(CR^{--}(F), \partial)$, $(CR^{-+}(F), \partial)$, 
$(CR^{+-}(F), \partial)$, $(CR^{++}(F), \partial)$ 
distinguishes the spun-trefoil 
from the 2-twist spun-trefoil in ${\mathbb{R}}^4$. 
\end{theorem} 

Ng and Gadgil \cite{ng2} defined a {\it cord ring} 
for codimension-2 submanifolds, and 
they proved the following. 

\begin{theorem} \cite{ng2} 
The cord ring distinguishes between the unknotted $S^2$ in $S^4$ 
and the spun-knot obtained from any knot in $S^3$ 
with non-trivial cord ring. 
\end{theorem}

The cord ring using {\it generic} near homotopy of cords calculates 
the 0-dimensional part of Roseman homology $HR^{\varepsilon \delta} (F)$. 
A direct calculation shows the following. 

\begin{theorem} \label{2-twist-spun-trefoil}
The 0-dimensional Roseman homology 
$HR_0^{\varepsilon \delta} (T^2(2, 3))$ 
of $T^2(2, 3)$ in ${\mathbb{R}}^4$ is isomorphic to 
$HR_0^{\varepsilon \delta} (T^0(2, 3))$ of $T^0(2, 3)$ in ${\mathbb{R}}^4$. 
\end{theorem} 

In particular, the 0-dimensional Roseman homology does not distinguish 
$T^2(2, 3)$ from $T^0(2, 3)$. 
We notice that 
$HR_0^{\varepsilon \delta} (F)$ 
does not see triple points in the diagram of $F$.

\section{Definition}

A {\it surface}-{\it knot} $F$ is a closed connected oriented surface 
embedded locally flatly in ${\mathbb{R}}^4$. 
For a 
projection $\pi \colon {\mathbb{R}}^4 \to {\mathbb{R}}^3$, 
we may assume 
that 
$\pi |_F$ is {\it generic}, 
that is, $\pi |_F$ has double points, 
isolated triple points and isolated branch points in the image 
as its singularities. 
A {\it diagram} of $F$ 
is a generic projection $\pi (F)$ 
equipped with a height information 
with respect to $\pi$ 
that is indicated 
by removing regular neighborhoods 
of double points in the lower component. 
We refer to \cite{carter-saito} and \cite{carter-kamada-saito} for details. 
A diagram is regarded as a disjoint union of 
compact 
oriented surfaces, each of which is called a {\it sheet}. 
We indicate an orientation of a surface-knot on its diagram 
by assigning its normal direction ${\overrightarrow{n}}$, depicted by an arrow, 
to each sheet of the diagram so that 
the ordered triple 
$({\overrightarrow{v_1}}, {\overrightarrow{v_2}}, {\overrightarrow{n}})$ 
agrees with the orientation of ${\mathbb{R}}^3$, 
where the ordered pair 
$({\overrightarrow{v_1}}, {\overrightarrow{v_2}})$ denotes 
the orientation of the surface $F$. 
See Figure \ref{singularity}.

Roseman \cite{roseman} introduced 
seven types of moves on diagrams of surface-knots, 
called {\it Roseman moves}. 
Yashiro \cite{yashiro} showed that 
one of them 
is 
obtained from the others. 
Roseman's theorem, after Yashiro's modification, 
is stated as follows.

\begin{theorem} \label{roseman-yashiro} \cite{roseman}, \cite{yashiro} 
Let $F_1, F_2$ denote surface-knots in ${\mathbb{R}}^4$, and 
let $D_1, D_2$ denote diagrams of $F_1, F_2$ in ${\mathbb{R}}^3$, 
respectively. 
Then the followings are equivalent. \\ 
$(1)$ 
$F_1$ is ambient isotopic to $F_2$ in ${\mathbb{R}}^4$. \\ 
$(2)$ 
$D_2$ is obtained from $D_1$ by 
a finite sequence of Roseman moves of types I, III, IV, V, VI and VII, 
as illustrated in Figures \ref{roseman3}, \ref{roseman14}, 
\ref{roseman5}, \ref{roseman6}, \ref{roseman7-21}. 
\end{theorem}

Let $D$ denote a diagram of a surface-knot in ${\mathbb{R}}^4$. 
We label sheets of $D$ by $1, \cdots, n$, 
connected components of double curves of $D$ 
by ${\bf 1}, \cdots, {\bf m}$, 
triple points of $D$ 
by ${\mathfrak{1}}, \cdots, {\mathfrak{t}}$, 
and 
branch points of positive sign of $D$ 
by ${\textsl{1}}, \cdots, {\textsl{b}}$. 

Let ${\bf x}$ denote a label on a double curve of $D$. 
See Figure \ref{singularity}. 
There are two kinds of sheets, an over-sheet and under-sheets, 
around the double curve with respect to the height. 
The over-sheet at the double curve 
is labeled by $o_{\bf x}$. 
The under-sheet at the double curve 
on the positive side of the over-sheet, 
that is, 
in the direction of the positive normal to the over-sheet, 
is labeled by $u^+_{\bf x}$, and 
the other under-sheet at the double curve 
is labeled by $u^-_{\bf x}$.

\begin{figure}
\begin{center}
\includegraphics{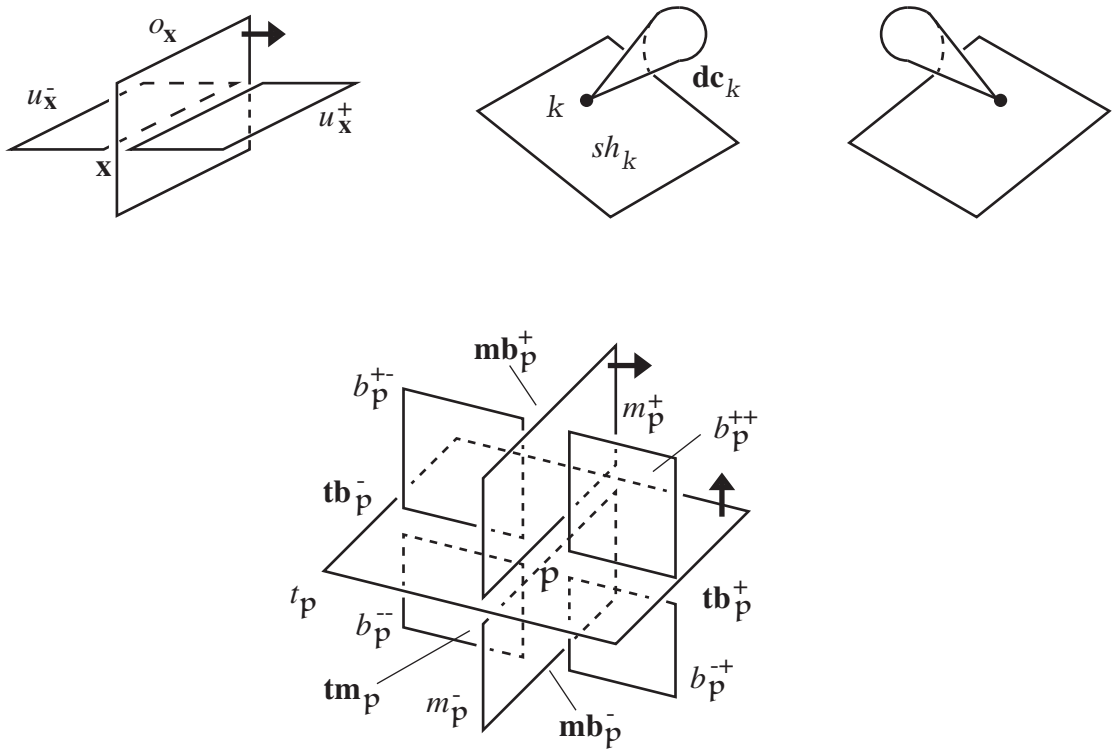}
\caption{diagram near double-curve (upper left), 
diagrams near branch point 
of positive sign (upper center) and of negative sign (upper right), and 
diagram near triple point (lower)}
\label{singularity}
\end{center}
\end{figure}

Let ${\mathfrak{p}}$ denote a label on a triple point of $D$. 
See Figure \ref{singularity}. 
There are three kinds of sheets, top, middle and bottom, 
around the triple point with respect to the height. 
The top sheet is labeled by $t_{\mathfrak{p}}$. 
The middle sheet on the positive (resp. negative) side 
of the top sheet 
is labeled by 
$m_{\mathfrak{p}}^+$ (resp. $m_{\mathfrak{p}}^-$). 
The bottom sheet on the $\alpha$-side of 
the top sheet, 
and 
on the $\beta$-side of the middle sheet 
is labeled by $b_{\mathfrak{p}}^{\alpha \beta}$, 
where $\alpha, \beta \in \{ +, - \}$. 
There are five double curves near the triple point. 
The double curve between the top sheet and the middle sheets is 
labeled by ${\bf tm}_{\mathfrak{p}}$. 
The double curve between the middle sheet and the bottom sheets 
on the $\alpha$-side of the top sheet is 
labeled by ${\bf mb}_{\mathfrak{p}}^\alpha$, 
where $\alpha \in \{ +, - \}$. 
The double curve between the top sheet and the bottom sheets 
on the $\beta$-side of the middle sheet 
is labeled by ${\bf tb}_{\mathfrak{p}}^\beta$, 
where $\beta \in \{ +, - \}$.

Let ${\textsl{k}}$ denote a label on a branch point of positive sign of $D$. 
See Figure \ref{singularity}. 
The double curve emanating from the branch point 
is labeled by ${\bf{dc}}_{\textsl{k}}$, 
and the sheet around the branch point 
is labeled by 
$sh_{\textsl{k}}$.

A unital graded algebra $CR^{\varepsilon \delta}(D)$ over ${\mathbb{Z}}$ 
$(\varepsilon, \delta \in \{ +, - \})$ is 
generated by the group ring ${\mathbb{Z}}[\mu, \mu^{-1}]$ 
in degree 0, 
along with the following generators: \\ 
$\{ a_{11}(i, j) \}$ in degree 0, \\ 
$\{ a_{21}({\bf x}, i) \}$, $\{ a_{12}(i, {\bf x}) \}$ 
in degree 1, \\ 
$\{ a_{22}({\bf x}, {\bf y}) \}$, $\{ a_{2}({\bf x}) \}$, 
$\{ a_{31}^{\varepsilon \delta} ({\mathfrak{p}}, i) \}$, 
$\{ a_{13}^{\varepsilon \delta} (i, {\mathfrak{p}}) \}$, 
$\{ a_{b1} ({\textsl{k}}, i) \}$, 
$\{ a_{1b} (i, {\textsl{k}}) \}$ 
in degree 2, \\ 
$\{ a_{32}^{\varepsilon \delta} ({\mathfrak{p}}, {\bf x}) \}$, 
$\{ a_{23}^{\varepsilon \delta} ({\bf x}, {\mathfrak{p}}) \}$, 
$\{ a_{b2} ({\textsl{k}}, {\bf x}) \}$, 
$\{ a_{2b} ({\bf x}, {\textsl{k}}) \}$, 
$\{ a_{3}^{\varepsilon \delta} ({\mathfrak{p}}) \}$, 
$\{ a_{b} ({\textsl{k}}) \}$ 
in degree 3, \\ 
$\{ a_{33}^{\varepsilon \delta} ({\mathfrak{p}}, {\mathfrak{q}}) \}$, 
$\{ a_{b3}^{\varepsilon \delta} ({\textsl{k}}, {\mathfrak{p}}) \}$, 
$\{ a_{3b}^{\varepsilon \delta} ({\mathfrak{p}}, {\textsl{k}}) \}$, 
$\{ a_{bb} ({\textsl{k}}, {\textsl{l}}) \}$ 
in degree 4, \\ 
where $i \neq j \in \{ 1, \cdots, n \}$, 
${\bf x}, {\bf y} \in \{ {\bf 1}, \cdots, {\bf m} \}$, 
${\mathfrak{p}}, {\mathfrak{q}} \in \{ {\mathfrak{1}}, \cdots, {\mathfrak{t}} \}, 
{\textsl{k}}, {\textsl{l}} \in \{ {\textsl{1}}, \cdots, {\textsl{b}} \}$. 
We set $a_{11}(i, i) = 1+\mu$ for $i \in \{ 1, \cdots, n \}$. 
We suppose that 
the above generators do not commute with each other 
in $CR^{\varepsilon \delta}(D)$, 
and that $\mu$ and $\mu^{-1}$ commute with all generators 
in $CR^{\varepsilon \delta}(D)$.

A differential $\partial$ on generators of $CR^{--}(D)$ is defined as follows: \\ 
$\partial a_{21}({\bf x}, i)$ 
$= \mu a_{11}(u^-_{\bf x}, i) 
+ a_{11}(u^+_{\bf x}, i) 
- a_{11}(u^-_{\bf x}, o_{\bf x}) 
a_{11}(o_{\bf x}, i)$, \\ 
$\partial a_{12}(i, {\bf x})$ 
$= - a_{11}(i, u^-_{\bf x}) 
- \mu a_{11}(i, u^+_{\bf x}) 
+ a_{11}(i, o_{\bf x}) 
a_{11}(o_{\bf x}, u^-_{\bf x})$, \\ 
$\partial a_{22}({\bf x}, {\bf y})$ 
$= \mu a_{12}(u^-_{\bf x}, {\bf y}) 
+ a_{12}(u^+_{\bf x}, {\bf y}) 
- a_{11}(u^-_{\bf x}, o_{\bf x}) 
a_{12}(o_{\bf x}, {\bf y})$ \\ 
$+ a_{21}({\bf x}, u^-_{\bf y}) 
+ \mu a_{21}({\bf x}, u^+_{\bf y}) 
- a_{21}({\bf x}, o_{\bf y}) 
a_{11}(o_{\bf y}, u^-_{\bf y})$, \\ 
$\partial a_{2}({\bf x})$ 
$= \mu a_{21}({\bf x}, u^+_{\bf x}) 
+ \mu a_{12}(u^-_{\bf x}, {\bf x}) 
- a_{11}(u^-_{\bf x}, o_{\bf x}) 
a_{12}(o_{\bf x}, {\bf x})$, \\ 
$\partial a_{31}^{--}({\mathfrak{p}}, i)$ 
$= \mu a_{21}({\bf tb}^-_{\mathfrak{p}}, i)$ 
$+ a_{21}({\bf tb}^+_{\mathfrak{p}}, i)$
$- a_{21}({\bf tb}^-_{\mathfrak{p}}, m^+_{\mathfrak{p}}) 
a_{11}(m^+_{\mathfrak{p}}, i)$ \\ 
$- \mu a_{21}({\bf mb}^-_{\mathfrak{p}}, i)$ 
$- a_{21}({\bf mb}^+_{\mathfrak{p}}, i)$
$+ a_{21}({\bf mb}^-_{\mathfrak{p}}, t_{\mathfrak{p}}) 
a_{11}(t_{\mathfrak{p}}, i)$ \\ 
$- a_{11}(b^{--}_{\mathfrak{p}}, m^-_{\mathfrak{p}}) 
a_{21}({\bf tm}_{\mathfrak{p}}, i)$
$- a_{12}(b^{--}_{\mathfrak{p}}, {\bf tm}_{\mathfrak{p}}) 
a_{11}(m^+_{\mathfrak{p}}, i)$ 
$+ \mu^{-1} a_{11}(b^{--}_{\mathfrak{p}}, t_{\mathfrak{p}}) 
a_{12}(t_{\mathfrak{p}}, {\bf tm}_{\mathfrak{p}}) 
a_{11}(m^+_{\mathfrak{p}}, i)$, \\ 
$\partial a_{13}^{--}(i, {\mathfrak{p}})$ 
$= a_{12}(i, {\bf tb}^-_{\mathfrak{p}})$ 
$+ \mu a_{12}(i, {\bf tb}^+_{\mathfrak{p}})$ 
$- a_{11}(i, m^+_{\mathfrak{p}}) 
a_{12}(m^+_{\mathfrak{p}}, {\bf tb}^-_{\mathfrak{p}})$ \\ 
$- a_{12}(i, {\bf mb}^-_{\mathfrak{p}})$ 
$- \mu a_{12}(i, {\bf mb}^+_{\mathfrak{p}})$ 
$+ a_{11}(i, t_{\mathfrak{p}}) 
a_{12}(t_{\mathfrak{p}}, {\bf mb}^-_{\mathfrak{p}})$ \\ 
$- a_{12}(i, {\bf tm}_{\mathfrak{p}}) 
a_{11}(m^-_{\mathfrak{p}}, b^{--}_{\mathfrak{p}})$ 
$- a_{11}(i, m^+_{\mathfrak{p}}) 
a_{21}({\bf tm}_{\mathfrak{p}}, b^{--}_{\mathfrak{p}})$ 
$+ a_{11}(i, m^+_{\mathfrak{p}}) 
a_{21}({\bf tm}_{\mathfrak{p}}, t_{\mathfrak{p}}) 
a_{11}(t_{\mathfrak{p}}, b^{--}_{\mathfrak{p}})$, \\ 
$\partial a_{32}^{--}({\mathfrak{p}}, {\bf x})$ 
$= \mu a_{22}({\bf tb}^-_{\mathfrak{p}}, {\bf x}) 
+ a_{22}({\bf tb}^+_{\mathfrak{p}}, {\bf x}) 
- a_{21}({\bf tb}^-_{\mathfrak{p}}, m^+_{\mathfrak{p}}) 
a_{12}(m^+_{\mathfrak{p}}, {\bf x})$ \\ 
$- \mu a_{22}({\bf mb}^-_{\mathfrak{p}}, {\bf x}) 
- a_{22}({\bf mb}^+_{\mathfrak{p}}, {\bf x}) 
+ a_{21}({\bf mb}^-_{\mathfrak{p}}, t_{\mathfrak{p}}) 
a_{12}(t_{\mathfrak{p}}, {\bf x})$ \\ 
$- a_{11}(b^{--}_{\mathfrak{p}}, m^-_{\mathfrak{p}}) 
a_{22}({\bf tm}_{\mathfrak{p}}, {\bf x}) 
- a_{12}(b^{--}_{\mathfrak{p}}, {\bf tm}_{\mathfrak{p}}) 
a_{12}(m^+_{\mathfrak{p}}, {\bf x}) 
+ \mu^{-1} a_{11}(b^{--}_{\mathfrak{p}}, t_{\mathfrak{p}}) 
a_{12}(t_{\mathfrak{p}}, {\bf tm}_{\mathfrak{p}}) 
a_{12}(m^+_{\mathfrak{p}}, {\bf x})$ \\ 
$- a_{31}^{--}({\mathfrak{p}}, u^-_{\bf x}) 
- \mu a_{31}^{--}({\mathfrak{p}}, u^+_{\bf x}) 
+ a_{31}^{--}({\mathfrak{p}}, o_{\bf x}) 
a_{11}(o_{\bf x}, u^-_{\bf x})$, \\ 
$\partial a_{23}^{--}({\bf x}, {\mathfrak{p}})$ 
$= \mu a_{13}^{--}(u^-_{\bf x}, {\mathfrak{p}}) 
+ a_{13}^{--}(u^+_{\bf x}, {\mathfrak{p}}) 
- a_{11}(u^-_{\bf x}, o_{\bf x}) 
a_{13}^{--}(o_{\bf x}, {\mathfrak{p}})$ \\ 
$- a_{22}({\bf x}, {\bf tb}^-_{\mathfrak{p}}) 
- \mu a_{22}({\bf x}, {\bf tb}^+_{\mathfrak{p}}) 
+ a_{21}({\bf x}, m^+_{\mathfrak{p}}) 
a_{12}(m^+_{\mathfrak{p}}, {\bf tb}^-_{\mathfrak{p}})$ \\ 
$+ a_{22}({\bf x}, {\bf mb}^-_{\mathfrak{p}}) 
+ \mu a_{22}({\bf x}, {\bf mb}^+_{\mathfrak{p}}) 
- a_{21}({\bf x}, t_{\mathfrak{p}}) 
a_{12}(t_{\mathfrak{p}}, {\bf mb}^-_{\mathfrak{p}})$ \\ 
$+ a_{22}({\bf x}, {\bf tm}_{\mathfrak{p}}) 
a_{11}(m^-_{\mathfrak{p}}, b^{--}_{\mathfrak{p}}) 
+ a_{21}({\bf x}, m^+_{\mathfrak{p}}) 
a_{21}({\bf tm}_{\mathfrak{p}}, b^{--}_{\mathfrak{p}}) 
- a_{21}({\bf x}, m^+_{\mathfrak{p}}) 
a_{21}({\bf tm}_{\mathfrak{p}}, t_{\mathfrak{p}}) 
a_{11}(t_{\mathfrak{p}}, b^{--}_{\mathfrak{p}})$, \\ 
$\partial a_{3}^{--}({\mathfrak{p}})$ 
$= \mu a_{31}^{--}({\mathfrak{p}}, b^{++}_{\mathfrak{p}})$ 
$- \mu a_{13}^{--}(b^{--}_{\mathfrak{p}}, {\mathfrak{p}})$ 
$+ a_{11}(b^{--}_{\mathfrak{p}}, m^-_{\mathfrak{p}}) 
a_{13}^{--}(m^-_{\mathfrak{p}}, {\mathfrak{p}})$ \\ 
$+ (a_{11}(b^{--}_{\mathfrak{p}}, t_{\mathfrak{p}}) 
- \mu^{-1} a_{11}(b^{--}_{\mathfrak{p}}, m^-_{\mathfrak{p}}) 
a_{11}(m^-_{\mathfrak{p}}, t_{\mathfrak{p}})) 
a_{13}^{--}(t_{\mathfrak{p}}, {\mathfrak{p}})$ \\ 
$+ a_{2}({\bf tb}^-_{\mathfrak{p}})$ 
$+ a_{2}({\bf mb}^+_{\mathfrak{p}})$ 
$+ \mu^{-1} a_{11}(b^{--}_{\mathfrak{p}}, m^-_{\mathfrak{p}}) 
a_{2}({\bf tm}_{\mathfrak{p}}) 
a_{11}(m^-_{\mathfrak{p}}, b^{--}_{\mathfrak{p}})$ 
$- a_{2}({\bf tb}^+_{\mathfrak{p}})$ 
$- a_{2}({\bf mb}^-_{\mathfrak{p}})$ \\ 
$- \mu a_{22}({\bf tb}^-_{\mathfrak{p}}, {\bf mb}^+_{\mathfrak{p}})$ 
$+ a_{11}(b^{--}_{\mathfrak{p}}, m^-_{\mathfrak{p}}) 
a_{22}({\bf tm}_{\mathfrak{p}}, {\bf mb}^+_{\mathfrak{p}})$ 
$- \mu^{-1} a_{11}(b^{--}_{\mathfrak{p}}, m^-_{\mathfrak{p}}) 
a_{22}({\bf tm}_{\mathfrak{p}}, {\bf tb}^-_{\mathfrak{p}})$ \\ 
$+ \mu a_{22}({\bf mb}^-_{\mathfrak{p}}, {\bf tb}^+_{\mathfrak{p}})$ 
$- a_{21}({\bf mb}^-_{\mathfrak{p}}, t_{\mathfrak{p}}) 
a_{12}(t_{\mathfrak{p}}, {\bf tb}^+_{\mathfrak{p}})$ \\ 
$+ (\mu^{-1} a_{11}(b^{--}_{\mathfrak{p}}, m^-_{\mathfrak{p}}) 
a_{21}({\bf tm}_{\mathfrak{p}}, m^+_{\mathfrak{p}}) 
+ a_{12}(b^{--}_{\mathfrak{p}}, {\bf tm}_{\mathfrak{p}}) 
- \mu^{-1} a_{11}(b^{--}_{\mathfrak{p}}, t_{\mathfrak{p}}) 
a_{12}(t_{\mathfrak{p}}, {\bf tm}_{\mathfrak{p}}))$ \\ 
$(a_{12}(m^+_{\mathfrak{p}}, {\bf tb}^-_{\mathfrak{p}}) 
+ a_{21}({\bf tm}_{\mathfrak{p}}, b^{--}_{\mathfrak{p}}) 
- a_{21}({\bf tm}_{\mathfrak{p}}, t_{\mathfrak{p}}) 
a_{11}(t_{\mathfrak{p}}, b^{--}_{\mathfrak{p}}))$ \\ 
$+ (a_{21}({\bf tb}^-_{\mathfrak{p}}, m^+_{\mathfrak{p}}) 
+ a_{12}(b^{--}_{\mathfrak{p}}, {\bf tm}_{\mathfrak{p}}) 
- \mu^{-1} a_{11}(b^{--}_{\mathfrak{p}}, t_{\mathfrak{p}}) 
a_{12}(t_{\mathfrak{p}}, {\bf tm}_{\mathfrak{p}})) 
a_{12}(m^+_{\mathfrak{p}}, {\bf mb}^+_{\mathfrak{p}})$, \\ 
$\partial a_{33}^{--}({\mathfrak{p}}, {\mathfrak{q}})$ 
$= \mu a_{23}^{--}({\bf tb}^-_{\mathfrak{p}}, {\mathfrak{q}}) 
+ a_{23}^{--}({\bf tb}^+_{\mathfrak{p}}, {\mathfrak{q}}) 
- a_{21}({\bf tb}^-_{\mathfrak{p}}, m^+_{\mathfrak{p}}) 
a_{13}^{--}(m^+_{\mathfrak{p}}, {\mathfrak{q}})$ \\ 
$- \mu a_{23}^{--}({\bf mb}^-_{\mathfrak{p}}, {\mathfrak{q}}) 
- a_{23}^{--}({\bf mb}^+_{\mathfrak{p}}, {\mathfrak{q}}) 
+ a_{21}({\bf mb}^-_{\mathfrak{p}}, t_{\mathfrak{p}}) 
a_{13}^{--}(t_{\mathfrak{p}}, {\mathfrak{q}})$ \\ 
$- a_{11}(b^{--}_{\mathfrak{p}}, m^-_{\mathfrak{p}}) 
a_{23}^{--}({\bf tm}_{\mathfrak{p}}, {\mathfrak{q}}) 
- a_{12}(b^{--}_{\mathfrak{p}}, {\bf tm}_{\mathfrak{p}}) 
a_{13}^{--}(m^+_{\mathfrak{p}}, {\mathfrak{q}})$ 
$+ \mu^{-1} a_{11}(b^{--}_{\mathfrak{p}}, t_{\mathfrak{p}}) 
a_{12}(t_{\mathfrak{p}}, {\bf tm}_{\mathfrak{p}}) 
a_{13}^{--}(m^+_{\mathfrak{p}}, {\mathfrak{q}})$ \\ 
$+ a_{32}^{--}({\mathfrak{p}}, {\bf tb}^-_{\mathfrak{q}}) 
+ \mu a_{32}^{--}({\mathfrak{p}}, {\bf tb}^+_{\mathfrak{q}}) 
- a_{31}^{--}({\mathfrak{p}}, m^+_{\mathfrak{q}}) 
a_{12}(m^+_{\mathfrak{q}}, {\bf tb}^-_{\mathfrak{q}})$ \\ 
$- a_{32}^{--}({\mathfrak{p}}, {\bf mb}^-_{\mathfrak{q}}) 
- \mu a_{32}^{--}({\mathfrak{p}}, {\bf mb}^+_{\mathfrak{q}}) 
+ a_{31}^{--}({\mathfrak{p}}, t_{\mathfrak{q}}) 
a_{12}(t_{\mathfrak{q}}, {\bf mb}^-_{\mathfrak{q}})$ \\ 
$- a_{32}^{--}({\mathfrak{p}}, {\bf tm}_{\mathfrak{q}}) 
a_{11}(m^-_{\mathfrak{q}}, b^{--}_{\mathfrak{q}}) 
- a_{31}^{--}({\mathfrak{p}}, m^+_{\mathfrak{q}}) 
a_{21}({\bf tm}_{\mathfrak{q}}, b^{--}_{\mathfrak{q}})$ 
$+ a_{31}^{--}({\mathfrak{p}}, m^+_{\mathfrak{q}}) 
a_{21}({\bf tm}_{\mathfrak{q}}, t_{\mathfrak{q}}) 
a_{11}(t_{\mathfrak{q}}, b^{--}_{\mathfrak{q}})$, \\ 
$\partial a_{b1}({\textsl{k}}, i)$ 
$= a_{21}({\bf dc}_{\textsl{k}}, i)$,  \\ 
$\partial a_{1b}(i, {\textsl{k}})$ 
$= a_{12}(i, {\bf dc}_{\textsl{k}})$, \\ 
$\partial a_{b2}({\textsl{k}}, {\bf x})$ 
$= a_{22}({\bf dc}_{\textsl{k}}, {\bf x}) 
- a_{b1}^{--}({\textsl{k}}, u^-_{\bf x}) 
- \mu a_{b1}^{--}({\textsl{k}}, u^+_{\bf x}) 
+ a_{b1}^{--}({\textsl{k}}, o_{\bf x}) 
a_{11}(o_{\bf x}, u^-_{\bf x})$, \\  
$\partial a_{2b}({\bf x}, {\textsl{k}})$ 
$= \mu a_{1b}^{--}(u^-_{\bf x}, {\textsl{k}}) 
+ a_{1b}^{--}(u^+_{\bf x}, {\textsl{k}}) 
- a_{11}(u^-_{\bf x}, o_{\bf x}) 
a_{1b}^{--}(o_{\bf x}, {\textsl{k}}) 
- a_{22}({\bf x}, {\bf dc}_{\textsl{k}})$, \\  
$\partial a_{b2}({\textsl{k}}, {\bf dc}_{\textsl{k}})$ 
$= a_{22}({\bf dc}_{\textsl{k}}, {\bf dc}_{\textsl{k}})$, \\ 
$\partial a_{2b}({\bf dc}_{\textsl{k}}, {\textsl{k}})$ 
$= - a_{22}({\bf dc}_{\textsl{k}}, {\bf dc}_{\textsl{k}})$, \\  
$\partial a_{b3}^{--}({\textsl{k}}, {\mathfrak{p}})$ 
$= a_{23}^{--}({\bf dc}_{\textsl{k}}, {\mathfrak{p}}) 
+ a_{b2}({\textsl{k}}, {\bf tb}^-_{\mathfrak{p}})$ 
$+ \mu a_{b2}({\textsl{k}}, {\bf tb}^+_{\mathfrak{p}})$ 
$- a_{b1}({\textsl{k}}, m^+_{\mathfrak{p}}) 
a_{12}(m^+_{\mathfrak{p}}, {\bf tb}^-_{\mathfrak{p}})$ \\ 
$- a_{b2}({\textsl{k}}, {\bf mb}^-_{\mathfrak{p}})$ 
$- \mu a_{b2}({\textsl{k}}, {\bf mb}^+_{\mathfrak{p}})$ 
$+ a_{b1}({\textsl{k}}, t_{\mathfrak{p}}) 
a_{12}(t_{\mathfrak{p}}, {\bf mb}^-_{\mathfrak{p}})$ \\ 
$- a_{b2}({\textsl{k}}, {\bf tm}_{\mathfrak{p}}) 
a_{11}(m^-_{\mathfrak{p}}, b^{--}_{\mathfrak{p}})$ 
$- a_{b1}({\textsl{k}}, m^+_{\mathfrak{p}}) 
a_{21}({\bf tm}_{\mathfrak{p}}, b^{--}_{\mathfrak{p}})$ 
$+ a_{b1}({\textsl{k}}, m^+_{\mathfrak{p}}) 
a_{21}({\bf tm}_{\mathfrak{p}}, t_{\mathfrak{p}}) 
a_{11}(t_{\mathfrak{p}}, b^{--}_{\mathfrak{p}})$, \\  
$\partial a_{3b}^{--}({\mathfrak{p}}, {\textsl{k}})$ 
$= \mu a_{2b}({\bf tb}^-_{\mathfrak{p}}, {\textsl{k}})$ 
$+ a_{2b}({\bf tb}^+_{\mathfrak{p}}, {\textsl{k}})$
$- a_{21}({\bf tb}^-_{\mathfrak{p}}, m^+_{\mathfrak{p}}) 
a_{1b}(m^+_{\mathfrak{p}}, {\textsl{k}})$ \\ 
$- \mu a_{2b}({\bf mb}^-_{\mathfrak{p}}, {\textsl{k}})$ 
$- a_{2b}({\bf mb}^+_{\mathfrak{p}}, {\textsl{k}})$
$+ a_{21}({\bf mb}^-_{\mathfrak{p}}, t_{\mathfrak{p}}) 
a_{1b}(t_{\mathfrak{p}}, {\textsl{k}})$ \\ 
$- a_{11}(b^{--}_{\mathfrak{p}}, m^-_{\mathfrak{p}}) 
a_{2b}({\bf tm}_{\mathfrak{p}}, {\textsl{k}})$
$- a_{12}(b^{--}_{\mathfrak{p}}, {\bf tm}_{\mathfrak{p}}) 
a_{1b}(m^+_{\mathfrak{p}}, {\textsl{k}})$ 
$+ \mu^{-1} a_{11}(b^{--}_{\mathfrak{p}}, t_{\mathfrak{p}}) 
a_{12}(t_{\mathfrak{p}}, {\bf tm}_{\mathfrak{p}}) 
a_{1b}(m^+_{\mathfrak{p}}, {\textsl{k}})$ \\ 
$+ a_{32}^{--}({\mathfrak{p}}, {\bf dc}_{\textsl{k}})$, \\  
$\partial a_{bb}({\textsl{k}}, {\textsl{l}})$ 
$= a_{2b}({\bf dc}_{\textsl{k}}, {\textsl{l}})$ 
$+ a_{b2}({\textsl{k}}, {\bf dc}_{\textsl{l}})$, \\ 
$\partial a_{b}({\textsl{k}})$ 
$= a_{2}({\bf dc}_{\textsl{k}})$ 
$- \mu a_{b1}({\textsl{k}}, sh_{\textsl{k}})$ 
$+ a_{1b}(sh_{\textsl{k}}, {\textsl{k}})$, \\ 
where $i \in \{ 1, \cdots, n \}, 
{\bf x}, {\bf y} \in \{ {\bf 1}, \cdots, {\bf m} \}, 
{\mathfrak{p}}, {\mathfrak{q}} \in \{ {\mathfrak{1}}, \cdots, {\mathfrak{t}} \}, 
{\textsl{k}}, {\textsl{l}} \in \{ {\textsl{1}}, \cdots, {\textsl{b}} \}$. 
We extend 
the differential $\partial$ 
by linearity over ${\mathbb{Z}}$, 
and the signed Leibniz rule: 
$\partial (v w) = (\partial v) w + (-1)^{\rm{deg} {\it v}} v (\partial w)$. 
It is straightforward to see that the equation $\partial \circ \partial = 0$ 
holds on generators of $CR^{--}(D)$. 
In $\S$10, we define differentials 
on generators 
of $CR^{-+}(D)$, $CR^{+-}(D)$ and $CR^{++}(D)$.

An algebra map between differential graded algebras 
$\phi \colon 
({\mathbb{Z}} \langle a_1^1, \cdots, a_n^1 \rangle, \partial^1) 
\to 
({\mathbb{Z}} \langle a_1^2, \cdots, a_n^2 \rangle, \partial^2)$ 
is an {\it elementary isomorphism} 
if the followings are satisfied: \\ 
$(1)$ $\phi$ is a graded chain map, \\ 
$(2)$ $\phi(a_i^1) = \alpha a_i^2 + v$ 
for some $i \in \{ 1, \cdots, n \}$, 
where $\alpha, v \in {\mathbb{Z}} \langle a_1^2, \cdots, a_n^2 \rangle$, and 
$\alpha$ is a unit, \\ 
$(3)$ 
$\phi(a_j^1) = a_j^2$ 
for $j \neq i$. \\ 
A {\it tame isomorphism} is a composition of elementary isomorphisms. 
Let $(E^i, \partial^i)$ be the tensor algebra 
on two generators $e_1^i, e_2^i$ with 
${\rm deg} (e_1^i) - 1 = {\rm deg} (e_2^i) = i$ 
such that 
the differential is induced by $\partial^i e_1^i = e_2^i$, $\partial^i e_2^i = 0$. 
The degree-$i$ {\it algebraic stabilization} of a differential graded algebra 
$(A, \partial)$ is the coproduct of $A$ with 
$E^i$, 
with the differential induced from $\partial$ and $\partial^i$. 
The inverse operation of the degree-$i$ algebraic stabilization 
is a degree-$i$ algebraic {\it destabilization}. 
Two differential graded algebras $(A_1, \partial_1)$ and $(A_2, \partial_2)$ 
are {\it stably tame isomorphic} 
if they are tame isomorphic 
after some number of algebraic stabilizations and destabilizations of 
$(A_1, \partial_1)$ and $(A_2, \partial_2)$.

Let $(A, \partial)$ be a differential graded algebra, 
where $A = {\mathbb{Z}} \langle a_1, \cdots, a_n \rangle$, 
and let $I$ denote a two-sided ideal in $A$ 
generated by $\{ \partial a_i \ | \  i = 1, \cdots, n \}$. 
A {\it characteristic algebra} ${\mathcal{C}}(A)$ of $(A, \partial)$ 
is defined to be the algebra $A/I$. 
See \cite{ng-computable}. 
Two characteristic algebras $A_1/I_1$ and $A_2/I_2$ are 
{\it tamely isomorphic} 
if we can add some number of generators to $A_1$ and 
the same number of generators to $I_1$, 
and similarly for $A_2$ and $I_2$, 
so that there is a tame isomorphism 
between $(A_1, \partial)$ and $(A_2, \partial)$ 
sending $I_1$ to $I_2$. 
In particular, tamely isomorphic characteristic algebras are 
isomorphic as algebras. 
A stabilization of $(A, \partial)$ adds two generators $e_1, e_2$ to $A$ 
and one generator $e_2$ to $I$; 
therefore $A/I$ changes by adding one generator $e_1$ and no relations. 
Two characteristic algebras $A_1/I_1$ and $A_2/I_2$ are 
{\it equivalent} if 
they are tamely isomorphic, 
after adding a possibly different finite number of generators 
to $A_1$ and $A_2$ 
but no additional relations to $I_1$ and $I_2$. 
It follows that if two differential graded algebras are 
stably tame isomorphic, 
then their characteristic algebras are equivalent.

\section{Roseman move III}

\begin{proposition} \label{roseman-3} 
Suppose that a diagram $D_2$ of a surface-knot $F$ 
is obtained from a diagram $D_1$ of $F$ 
by applying one Roseman move of type III. 
See Figures \ref{roseman3}. 
Then the differential graded algebra $(CR^{--} (D_2), \partial)$ 
is stably tame isomorphic to $(CR^{--} (D_1), \partial)$. 
\end{proposition} 

\begin{proof} 
There are two cases to study: 
the positive normal to the over-sheet 
either points 
against the under-sheet, 
as illustrated in the upper pair of Figure \ref{roseman3}, 
or 
points 
toward the under-sheet, 
as illustrated in lower pair of Figure \ref{roseman3}, 
when the sheets are disjoint. 
In the following, we study the case of upper pair of Figure \ref{roseman3}. 
Similar arguments as below prove the case of lower pair of 
Figure \ref{roseman3}.

\begin{figure}
\begin{center}
\includegraphics{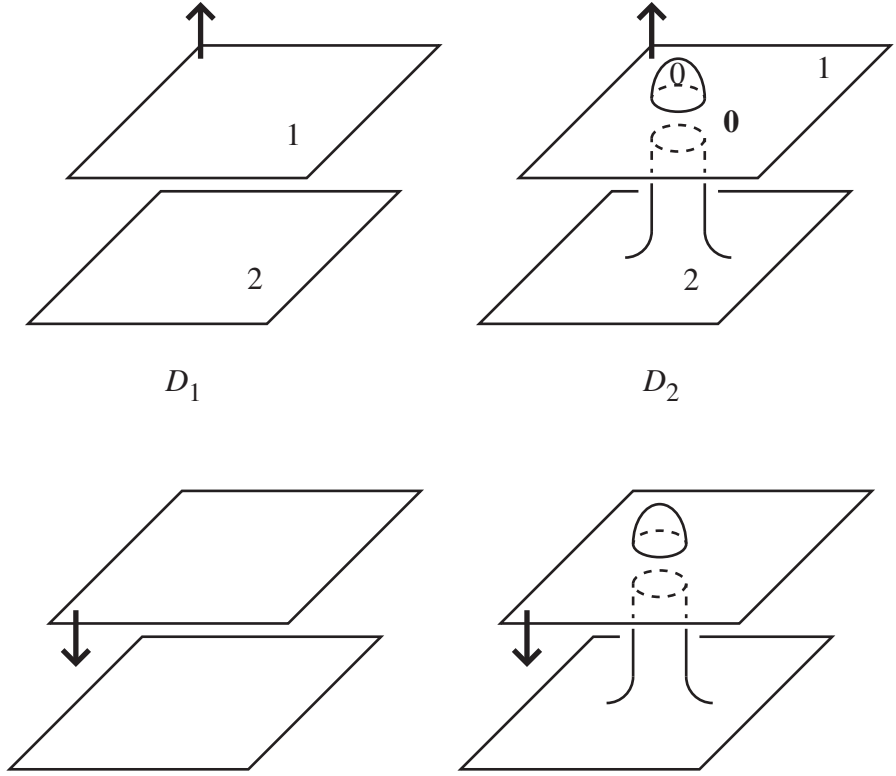}
\caption{Roseman move of type III}
\label{roseman3}
\end{center}
\end{figure}

We label sheets of $D_1$ by $1, \dots, n$, 
double curves of $D_1$ by ${\bf 1}, \cdots, {\bf m}$, 
triple points of $D_1$ by ${\mathfrak{1}}, \cdots, {\mathfrak{t}}$, and 
branch points of positive sign of $D_1$ by ${\textsl{1}}, \cdots, {\textsl{b}}$. 
We suppose that 
$1, 2$ are labels on the sheets of $D_1$ 
that are involved in the Roseman move. 
See Figure \ref{roseman3}. 
Let ${\bf 0}$ denote a label on the double-curve of $D_2$ 
created by the move, and 
let $0$ denote a label on the sheet of $D_2$ created by the move. 
The over-sheet of $D_2$ 
that corresponds to the sheet of $D_1$ with label $1$ 
is labeled by $1$, and 
the 
under-sheet of $D_2$ 
that corresponds to the sheet of $D_1$ with label $2$ 
is labeled by $2$. 
We suppose that the sheet of $D_2$ with label $i$ $(i \in \{ 3, \cdots, n \})$ 
corresponds to the sheet of $D_1$ with label $i$, 
that the double curve of $D_2$ with label ${\bf x}$ 
$({\bf x} \in \{ {\bf 1}, \cdots, {\bf m} \})$ 
corresponds to the double curve of $D_1$ with label ${\bf x}$, 
that the triple point of $D_2$ with label ${\mathfrak{p}}$ 
$({\mathfrak{p}} \in \{ {\mathfrak{1}}, \cdots, {\mathfrak{t}} \})$ 
corresponds to the triple point of $D_1$ with label ${\mathfrak{p}}$, 
and that the branch point of positive sign of $D_2$ with label ${\textsl{k}}$ 
$({\textsl{k}} \in \{ {\textsl{1}}, \cdots, {\textsl{b}} \})$ 
corresponds to the branch point of positive sign of $D_1$ 
with label ${\textsl{k}}$ .

The diagram of $D_2$ in Figure \ref{roseman3} shows that 
the differential of $(CR^{--} (D_2), \partial)$ 
on generators is described as follows. \\ 
$\partial a_{21}({\bf x}, i)$ 
$= \mu a_{11}(u^-_{\bf x}, i) 
+ a_{11}(u^+_{\bf x}, i) 
- a_{11}(u^-_{\bf x}, o_{\bf x}) 
a_{11}(o_{\bf x}, i)$, \\ 
$\partial a_{12}(i, {\bf x})$ 
$= - a_{11}(i, u^-_{\bf x}) 
- \mu a_{11}(i, u^+_{\bf x}) 
+ a_{11}(i, o_{\bf x}) 
a_{11}(o_{\bf x}, u^-_{\bf x})$, \\ 
$\partial a_{21}({\bf x}, 0)$ 
$= \mu a_{11}(u^-_{\bf x}, 0) 
+ a_{11}(u^+_{\bf x}, 0) 
- a_{11}(u^-_{\bf x}, o_{\bf x}) 
a_{11}(o_{\bf x}, 0)$, \\ 
$\partial a_{12}(0, {\bf x})$ 
$= - a_{11}(0, u^-_{\bf x}) 
- \mu a_{11}(0, u^+_{\bf x}) 
+ a_{11}(0, o_{\bf x}) 
a_{11}(o_{\bf x}, u^-_{\bf x})$, \\ 
$\partial a_{21}({\bf x}, 1)$ 
$= \mu a_{11}(u^-_{\bf x}, 1) 
+ a_{11}(u^+_{\bf x}, 1) 
- a_{11}(u^-_{\bf x}, o_{\bf x}) 
a_{11}(o_{\bf x}, 1)$, \\ 
$\partial a_{12}(1, {\bf x})$ 
$= - a_{11}(1, u^-_{\bf x}) 
- \mu a_{11}(1, u^+_{\bf x}) 
+ a_{11}(1, o_{\bf x}) 
a_{11}(o_{\bf x}, u^-_{\bf x})$, \\ 
$\partial a_{21}({\bf x}, 2)$ 
$= \mu a_{11}(u^-_{\bf x}, 2) 
+ a_{11}(u^+_{\bf x}, 2) 
- a_{11}(u^-_{\bf x}, o_{\bf x}) 
a_{11}(o_{\bf x}, 2)$, \\ 
$\partial a_{12}(2, {\bf x})$ 
$= - a_{11}(2, u^-_{\bf x}) 
- \mu a_{11}(2, u^+_{\bf x}) 
+ a_{11}(2, o_{\bf x}) 
a_{11}(o_{\bf x}, u^-_{\bf x})$, \\ 
$\partial a_{21}({\bf 0}, i)$ 
$= \mu a_{11}(2, i) 
+ a_{11}(0, i) 
- a_{11}(2, 1) 
a_{11}(1, i)$, \\ 
$\partial a_{12}(i, {\bf 0})$ 
$= - a_{11}(i, 2) 
- \mu a_{11}(i, 0) 
+ a_{11}(i, 1) 
a_{11}(1, 2)$, \\ 
$\partial a_{21}({\bf 0}, 0)$ 
$= \mu a_{11}(2, 0) 
+ a_{11}(0, 0) 
- a_{11}(2, 1) 
a_{11}(1, 0)$, \\ 
$\partial a_{12}(0, {\bf 0})$ 
$= - a_{11}(0, 2) 
- \mu a_{11}(0, 0) 
+ a_{11}(0, 1) 
a_{11}(1, 2)$, \\ 
$\partial a_{21}({\bf 0}, 1)$ 
$= \mu a_{11}(2, 1) 
+ a_{11}(0, 1) 
- a_{11}(2, 1) 
a_{11}(1, 1)$, \\ 
$\partial a_{12}(1, {\bf 0})$ 
$= - a_{11}(1, 2) 
- \mu a_{11}(1, 0) 
+ a_{11}(1, 1) 
a_{11}(1, 2)$, \\ 
$\partial a_{21}({\bf 0}, 2)$ 
$= \mu a_{11}(2, 2) 
+ a_{11}(0, 2) 
- a_{11}(2, 1) 
a_{11}(1, 2)$, \\ 
$\partial a_{12}(2, {\bf 0})$ 
$= - a_{11}(2, 2) 
- \mu a_{11}(2, 0) 
+ a_{11}(2, 1) 
a_{11}(1, 2)$, \\ 
$\partial a_{22}({\bf x}, {\bf y})$ 
$= \mu a_{12}(u^-_{\bf x}, {\bf y}) 
+ a_{12}(u^+_{\bf x}, {\bf y}) 
- a_{11}(u^-_{\bf x}, o_{\bf x}) 
a_{12}(o_{\bf x}, {\bf y})$ 
$+ a_{21}({\bf x}, u^-_{\bf y}) 
+ \mu a_{21}({\bf x}, u^+_{\bf y}) 
- a_{21}({\bf x}, o_{\bf y}) 
a_{11}(o_{\bf y}, u^-_{\bf y})$, \\ 
$\partial a_{22}({\bf x}, {\bf 0})$ 
$= \mu a_{12}(u^-_{\bf x}, {\bf 0}) 
+ a_{12}(u^+_{\bf x}, {\bf 0}) 
- a_{11}(u^-_{\bf x}, o_{\bf x}) 
a_{12}(o_{\bf x}, {\bf 0})$ 
$+ a_{21}({\bf x}, 2) 
+ \mu a_{21}({\bf x}, 0) 
- a_{21}({\bf x}, 1) 
a_{11}(1, 2)$, \\ 
$\partial a_{22}({\bf 0}, {\bf x})$ 
$= \mu a_{12}(2, {\bf x}) 
+ a_{12}(0, {\bf x}) 
- a_{11}(2, 1) 
a_{12}(1, {\bf x})$ 
$+ a_{21}({\bf 0}, u^-_{\bf x}) 
+ \mu a_{21}({\bf 0}, u^+_{\bf x}) 
- a_{21}({\bf 0}, o_{\bf x}) 
a_{11}(o_{\bf x}, u^-_{\bf x})$, \\ 
$\partial a_{22}({\bf 0}, {\bf 0})$ 
$= \mu a_{12}(2, {\bf 0}) 
+ a_{12}(0, {\bf 0}) 
- a_{11}(2, 1) 
a_{12}(1, {\bf 0})$ 
$+ a_{21}({\bf 0}, 2) 
+ \mu a_{21}({\bf 0}, 0) 
- a_{21}({\bf 0}, 1) 
a_{11}(1, 2)$, \\ 
$\partial a_{2}({\bf x})$ 
$= \mu a_{21}({\bf x}, u^+_{\bf x}) 
+ \mu a_{12}(u^-_{\bf x}, {\bf x}) 
- a_{11}(u^-_{\bf x}, o_{\bf x}) 
a_{12}(o_{\bf x}, {\bf x})$, \\ 
$\partial a_{2}({\bf 0})$ 
$= \mu a_{21}({\bf 0}, 0) 
+ \mu a_{12}(2, {\bf 0}) 
- a_{11}(2, 1) 
a_{12}(1, {\bf 0})$, \\ 
$\partial a_{31}^{--}({\mathfrak{p}}, i)$ 
$= \mu a_{21}({\bf tb}^-_{\mathfrak{p}}, i)$ 
$+ a_{21}({\bf tb}^+_{\mathfrak{p}}, i)$
$- a_{21}({\bf tb}^-_{\mathfrak{p}}, m^+_{\mathfrak{p}}) 
a_{11}(m^+_{\mathfrak{p}}, i)$ \\ 
$- \mu a_{21}({\bf mb}^-_{\mathfrak{p}}, i)$ 
$- a_{21}({\bf mb}^+_{\mathfrak{p}}, i)$
$+ a_{21}({\bf mb}^-_{\mathfrak{p}}, t_{\mathfrak{p}}) 
a_{11}(t_{\mathfrak{p}}, i)$ \\ 
$- a_{11}(b^{--}_{\mathfrak{p}}, m^-_{\mathfrak{p}}) 
a_{21}({\bf tm}_{\mathfrak{p}}, i)$
$- a_{12}(b^{--}_{\mathfrak{p}}, {\bf tm}_{\mathfrak{p}}) 
a_{11}(m^+_{\mathfrak{p}}, i)$ 
$+ \mu^{-1} a_{11}(b^{--}_{\mathfrak{p}}, t_{\mathfrak{p}}) 
a_{12}(t_{\mathfrak{p}}, {\bf tm}_{\mathfrak{p}}) 
a_{11}(m^+_{\mathfrak{p}}, i)$, \\ 
$\partial a_{13}^{--}(i, {\mathfrak{p}})$ 
$= a_{12}(i, {\bf tb}^-_{\mathfrak{p}})$ 
$+ \mu a_{12}(i, {\bf tb}^+_{\mathfrak{p}})$ 
$- a_{11}(i, m^+_{\mathfrak{p}}) 
a_{12}(m^+_{\mathfrak{p}}, {\bf tb}^-_{\mathfrak{p}})$ \\ 
$- a_{12}(i, {\bf mb}^-_{\mathfrak{p}})$ 
$- \mu a_{12}(i, {\bf mb}^+_{\mathfrak{p}})$ 
$+ a_{11}(i, t_{\mathfrak{p}}) 
a_{12}(t_{\mathfrak{p}}, {\bf mb}^-_{\mathfrak{p}})$ \\ 
$- a_{12}(i, {\bf tm}_{\mathfrak{p}}) 
a_{11}(m^-_{\mathfrak{p}}, b^{--}_{\mathfrak{p}})$ 
$- a_{11}(i, m^+_{\mathfrak{p}}) 
a_{21}({\bf tm}_{\mathfrak{p}}, b^{--}_{\mathfrak{p}})$ 
$+ a_{11}(i, m^+_{\mathfrak{p}}) 
a_{21}({\bf tm}_{\mathfrak{p}}, t_{\mathfrak{p}}) 
a_{11}(t_{\mathfrak{p}}, b^{--}_{\mathfrak{p}})$, \\ 
$\partial a_{31}^{--}({\mathfrak{p}}, 0)$ 
$= \mu a_{21}({\bf tb}^-_{\mathfrak{p}}, 0)$ 
$+ a_{21}({\bf tb}^+_{\mathfrak{p}}, 0)$
$- a_{21}({\bf tb}^-_{\mathfrak{p}}, m^+_{\mathfrak{p}}) 
a_{11}(m^+_{\mathfrak{p}}, 0)$ \\ 
$- \mu a_{21}({\bf mb}^-_{\mathfrak{p}}, 0)$ 
$- a_{21}({\bf mb}^+_{\mathfrak{p}}, 0)$
$+ a_{21}({\bf mb}^-_{\mathfrak{p}}, t_{\mathfrak{p}}) 
a_{11}(t_{\mathfrak{p}}, 0)$ \\ 
$- a_{11}(b^{--}_{\mathfrak{p}}, m^-_{\mathfrak{p}}) 
a_{21}({\bf tm}_{\mathfrak{p}}, 0)$
$- a_{12}(b^{--}_{\mathfrak{p}}, {\bf tm}_{\mathfrak{p}}) 
a_{11}(m^+_{\mathfrak{p}}, 0)$ 
$+ \mu^{-1} a_{11}(b^{--}_{\mathfrak{p}}, t_{\mathfrak{p}}) 
a_{12}(t_{\mathfrak{p}}, {\bf tm}_{\mathfrak{p}}) 
a_{11}(m^+_{\mathfrak{p}}, 0)$, \\  
$\partial a_{13}^{--}(0, {\mathfrak{p}})$ 
$= a_{12}(0, {\bf tb}^-_{\mathfrak{p}})$ 
$+ \mu a_{12}(0, {\bf tb}^+_{\mathfrak{p}})$ 
$- a_{11}(0, m^+_{\mathfrak{p}}) 
a_{12}(m^+_{\mathfrak{p}}, {\bf tb}^-_{\mathfrak{p}})$ \\ 
$- a_{12}(0, {\bf mb}^-_{\mathfrak{p}})$ 
$- \mu a_{12}(0, {\bf mb}^+_{\mathfrak{p}})$ 
$+ a_{11}(0, t_{\mathfrak{p}}) 
a_{12}(t_{\mathfrak{p}}, {\bf mb}^-_{\mathfrak{p}})$ \\ 
$- a_{12}(0, {\bf tm}_{\mathfrak{p}}) 
a_{11}(m^-_{\mathfrak{p}}, b^{--}_{\mathfrak{p}})$ 
$- a_{11}(0, m^+_{\mathfrak{p}}) 
a_{21}({\bf tm}_{\mathfrak{p}}, b^{--}_{\mathfrak{p}})$ 
$+ a_{11}(0, m^+_{\mathfrak{p}}) 
a_{21}({\bf tm}_{\mathfrak{p}}, t_{\mathfrak{p}}) 
a_{11}(t_{\mathfrak{p}}, b^{--}_{\mathfrak{p}})$, \\ 
$\partial a_{31}^{--}({\mathfrak{p}}, 1)$ 
$= \mu a_{21}({\bf tb}^-_{\mathfrak{p}}, 1)$ 
$+ a_{21}({\bf tb}^+_{\mathfrak{p}}, 1)$
$- a_{21}({\bf tb}^-_{\mathfrak{p}}, m^+_{\mathfrak{p}}) 
a_{11}(m^+_{\mathfrak{p}}, 1)$ \\ 
$- \mu a_{21}({\bf mb}^-_{\mathfrak{p}}, 1)$ 
$- a_{21}({\bf mb}^+_{\mathfrak{p}}, 1)$
$+ a_{21}({\bf mb}^-_{\mathfrak{p}}, t_{\mathfrak{p}}) 
a_{11}(t_{\mathfrak{p}}, 1)$ \\ 
$- a_{11}(b^{--}_{\mathfrak{p}}, m^-_{\mathfrak{p}}) 
a_{21}({\bf tm}_{\mathfrak{p}}, 1)$
$- a_{12}(b^{--}_{\mathfrak{p}}, {\bf tm}_{\mathfrak{p}}) 
a_{11}(m^+_{\mathfrak{p}}, 1)$ 
$+ \mu^{-1} a_{11}(b^{--}_{\mathfrak{p}}, t_{\mathfrak{p}}) 
a_{12}(t_{\mathfrak{p}}, {\bf tm}_{\mathfrak{p}}) 
a_{11}(m^+_{\mathfrak{p}}, 1)$, \\ 
$\partial a_{13}^{--}(1, {\mathfrak{p}})$ 
$= a_{12}(1, {\bf tb}^-_{\mathfrak{p}})$ 
$+ \mu a_{12}(1, {\bf tb}^+_{\mathfrak{p}})$ 
$- a_{11}(1, m^+_{\mathfrak{p}}) 
a_{12}(m^+_{\mathfrak{p}}, {\bf tb}^-_{\mathfrak{p}})$ \\ 
$- a_{12}(1, {\bf mb}^-_{\mathfrak{p}})$ 
$- \mu a_{12}(1, {\bf mb}^+_{\mathfrak{p}})$ 
$+ a_{11}(1, t_{\mathfrak{p}}) 
a_{12}(t_{\mathfrak{p}}, {\bf mb}^-_{\mathfrak{p}})$ \\ 
$- a_{12}(1, {\bf tm}_{\mathfrak{p}}) 
a_{11}(m^-_{\mathfrak{p}}, b^{--}_{\mathfrak{p}})$ 
$- a_{11}(1, m^+_{\mathfrak{p}}) 
a_{21}({\bf tm}_{\mathfrak{p}}, b^{--}_{\mathfrak{p}})$ 
$+ a_{11}(1, m^+_{\mathfrak{p}}) 
a_{21}({\bf tm}_{\mathfrak{p}}, t_{\mathfrak{p}}) 
a_{11}(t_{\mathfrak{p}}, b^{--}_{\mathfrak{p}})$, \\ 
$\partial a_{31}^{--}({\mathfrak{p}}, 2)$ 
$= \mu a_{21}({\bf tb}^-_{\mathfrak{p}}, 2)$ 
$+ a_{21}({\bf tb}^+_{\mathfrak{p}}, 2)$
$- a_{21}({\bf tb}^-_{\mathfrak{p}}, m^+_{\mathfrak{p}}) 
a_{11}(m^+_{\mathfrak{p}}, 2)$ \\ 
$- \mu a_{21}({\bf mb}^-_{\mathfrak{p}}, 2)$ 
$- a_{21}({\bf mb}^+_{\mathfrak{p}}, 2)$
$+ a_{21}({\bf mb}^-_{\mathfrak{p}}, t_{\mathfrak{p}}) 
a_{11}(t_{\mathfrak{p}}, 2)$ \\ 
$- a_{11}(b^{--}_{\mathfrak{p}}, m^-_{\mathfrak{p}}) 
a_{21}({\bf tm}_{\mathfrak{p}}, 2)$
$- a_{12}(b^{--}_{\mathfrak{p}}, {\bf tm}_{\mathfrak{p}}) 
a_{11}(m^+_{\mathfrak{p}}, 2)$ 
$+ \mu^{-1} a_{11}(b^{--}_{\mathfrak{p}}, t_{\mathfrak{p}}) 
a_{12}(t_{\mathfrak{p}}, {\bf tm}_{\mathfrak{p}}) 
a_{11}(m^+_{\mathfrak{p}}, 2)$, \\ 
$\partial a_{13}^{--}(2, {\mathfrak{p}})$ 
$= a_{12}(2, {\bf tb}^-_{\mathfrak{p}})$ 
$+ \mu a_{12}(2, {\bf tb}^+_{\mathfrak{p}})$ 
$- a_{11}(2, m^+_{\mathfrak{p}}) 
a_{12}(m^+_{\mathfrak{p}}, {\bf tb}^-_{\mathfrak{p}})$ \\ 
$- a_{12}(2, {\bf mb}^-_{\mathfrak{p}})$ 
$- \mu a_{12}(2, {\bf mb}^+_{\mathfrak{p}})$ 
$+ a_{11}(2, t_{\mathfrak{p}}) 
a_{12}(t_{\mathfrak{p}}, {\bf mb}^-_{\mathfrak{p}})$ \\ 
$- a_{12}(2, {\bf tm}_{\mathfrak{p}}) 
a_{11}(m^-_{\mathfrak{p}}, b^{--}_{\mathfrak{p}})$ 
$- a_{11}(2, m^+_{\mathfrak{p}}) 
a_{21}({\bf tm}_{\mathfrak{p}}, b^{--}_{\mathfrak{p}})$ 
$+ a_{11}(2, m^+_{\mathfrak{p}}) 
a_{21}({\bf tm}_{\mathfrak{p}}, t_{\mathfrak{p}}) 
a_{11}(t_{\mathfrak{p}}, b^{--}_{\mathfrak{p}})$, \\ 
$\partial a_{32}^{--}({\mathfrak{p}}, {\bf x})$ 
$= \mu a_{22}({\bf tb}^-_{\mathfrak{p}}, {\bf x}) 
+ a_{22}({\bf tb}^+_{\mathfrak{p}}, {\bf x}) 
- a_{21}({\bf tb}^-_{\mathfrak{p}}, m^+_{\mathfrak{p}}) 
a_{12}(m^+_{\mathfrak{p}}, {\bf x})$ \\ 
$- \mu a_{22}({\bf mb}^-_{\mathfrak{p}}, {\bf x}) 
- a_{22}({\bf mb}^+_{\mathfrak{p}}, {\bf x}) 
+ a_{21}({\bf mb}^-_{\mathfrak{p}}, t_{\mathfrak{p}}) 
a_{12}(t_{\mathfrak{p}}, {\bf x})$ \\ 
$- a_{11}(b^{--}_{\mathfrak{p}}, m^-_{\mathfrak{p}}) 
a_{22}({\bf tm}_{\mathfrak{p}}, {\bf x}) 
- a_{12}(b^{--}_{\mathfrak{p}}, {\bf tm}_{\mathfrak{p}}) 
a_{12}(m^+_{\mathfrak{p}}, {\bf x}) 
+ \mu^{-1} a_{11}(b^{--}_{\mathfrak{p}}, t_{\mathfrak{p}}) 
a_{12}(t_{\mathfrak{p}}, {\bf tm}_{\mathfrak{p}}) 
a_{12}(m^+_{\mathfrak{p}}, {\bf x})$ \\ 
$- a_{31}^{--}({\mathfrak{p}}, u^-_{\bf x}) 
- \mu a_{31}^{--}({\mathfrak{p}}, u^+_{\bf x}) 
+ a_{31}^{--}({\mathfrak{p}}, o_{\bf x}) 
a_{11}(o_{\bf x}, u^-_{\bf x})$, \\ 
$\partial a_{23}^{--}({\bf x}, {\mathfrak{p}})$ 
$= \mu a_{13}^{--}(u^-_{\bf x}, {\mathfrak{p}}) 
+ a_{13}^{--}(u^+_{\bf x}, {\mathfrak{p}}) 
- a_{11}(u^-_{\bf x}, o_{\bf x}) 
a_{13}^{--}(o_{\bf x}, {\mathfrak{p}})$ \\ 
$- a_{22}({\bf x}, {\bf tb}^-_{\mathfrak{p}}) 
- \mu a_{22}({\bf x}, {\bf tb}^+_{\mathfrak{p}}) 
+ a_{21}({\bf x}, m^+_{\mathfrak{p}}) 
a_{12}(m^+_{\mathfrak{p}}, {\bf tb}^-_{\mathfrak{p}})$ \\ 
$+ a_{22}({\bf x}, {\bf mb}^-_{\mathfrak{p}}) 
+ \mu a_{22}({\bf x}, {\bf mb}^+_{\mathfrak{p}}) 
- a_{21}({\bf x}, t_{\mathfrak{p}}) 
a_{12}(t_{\mathfrak{p}}, {\bf mb}^-_{\mathfrak{p}})$ \\ 
$+ a_{22}({\bf x}, {\bf tm}_{\mathfrak{p}}) 
a_{11}(m^-_{\mathfrak{p}}, b^{--}_{\mathfrak{p}}) 
+ a_{21}({\bf x}, m^+_{\mathfrak{p}}) 
a_{21}({\bf tm}_{\mathfrak{p}}, b^{--}_{\mathfrak{p}}) 
- a_{21}({\bf x}, m^+_{\mathfrak{p}}) 
a_{21}({\bf tm}_{\mathfrak{p}}, t_{\mathfrak{p}}) 
a_{11}(t_{\mathfrak{p}}, b^{--}_{\mathfrak{p}})$, \\ 
$\partial a_{32}^{--}({\mathfrak{p}}, {\bf 0})$ 
$= \mu a_{22}({\bf tb}^-_{\mathfrak{p}}, {\bf 0}) 
+ a_{22}({\bf tb}^+_{\mathfrak{p}}, {\bf 0}) 
- a_{21}({\bf tb}^-_{\mathfrak{p}}, m^+_{\mathfrak{p}}) 
a_{12}(m^+_{\mathfrak{p}}, {\bf 0})$ \\ 
$- \mu a_{22}({\bf mb}^-_{\mathfrak{p}}, {\bf 0}) 
- a_{22}({\bf mb}^+_{\mathfrak{p}}, {\bf 0}) 
+ a_{21}({\bf mb}^-_{\mathfrak{p}}, t_{\mathfrak{p}}) 
a_{12}(t_{\mathfrak{p}}, {\bf 0})$ \\ 
$- a_{11}(b^{--}_{\mathfrak{p}}, m^-_{\mathfrak{p}}) 
a_{22}({\bf tm}_{\mathfrak{p}}, {\bf 0}) 
- a_{12}(b^{--}_{\mathfrak{p}}, {\bf tm}_{\mathfrak{p}}) 
a_{12}(m^+_{\mathfrak{p}}, {\bf 0}) 
+ \mu^{-1} a_{11}(b^{--}_{\mathfrak{p}}, t_{\mathfrak{p}}) 
a_{12}(t_{\mathfrak{p}}, {\bf tm}_{\mathfrak{p}}) 
a_{12}(m^+_{\mathfrak{p}}, {\bf 0})$ \\ 
$- a_{31}^{--}({\mathfrak{p}}, 2) 
- \mu a_{31}^{--}({\mathfrak{p}}, 0) 
+ a_{31}^{--}({\mathfrak{p}}, 1) 
a_{11}(1, 2)$, \\ 
$\partial a_{23}^{--}({\bf 0}, {\mathfrak{p}})$ 
$= \mu a_{13}^{--}(2, {\mathfrak{p}}) 
+ a_{13}^{--}(0, {\mathfrak{p}}) 
- a_{11}(2, 1) 
a_{13}^{--}(1, {\mathfrak{p}})$ \\ 
$- a_{22}({\bf 0}, {\bf tb}^-_{\mathfrak{p}}) 
- \mu a_{22}({\bf 0}, {\bf tb}^+_{\mathfrak{p}}) 
+ a_{21}({\bf 0}, m^+_{\mathfrak{p}}) 
a_{12}(m^+_{\mathfrak{p}}, {\bf tb}^-_{\mathfrak{p}})$ \\ 
$+ a_{22}({\bf 0}, {\bf mb}^-_{\mathfrak{p}}) 
+ \mu a_{22}({\bf 0}, {\bf mb}^+_{\mathfrak{p}}) 
- a_{21}({\bf 0}, t_{\mathfrak{p}}) 
a_{12}(t_{\mathfrak{p}}, {\bf mb}^-_{\mathfrak{p}})$ \\ 
$+ a_{22}({\bf 0}, {\bf tm}_{\mathfrak{p}}) 
a_{11}(m^-_{\mathfrak{p}}, b^{--}_{\mathfrak{p}}) 
+ a_{21}({\bf 0}, m^+_{\mathfrak{p}}) 
a_{21}({\bf tm}_{\mathfrak{p}}, b^{--}_{\mathfrak{p}}) 
- a_{21}({\bf 0}, m^+_{\mathfrak{p}}) 
a_{21}({\bf tm}_{\mathfrak{p}}, t_{\mathfrak{p}}) 
a_{11}(t_{\mathfrak{p}}, b^{--}_{\mathfrak{p}})$, \\ 
$\partial a_{3}^{--}({\mathfrak{p}})$ 
$= \mu a_{31}^{--}({\mathfrak{p}}, b^{++}_{\mathfrak{p}})$ 
$- \mu a_{13}^{--}(b^{--}_{\mathfrak{p}}, {\mathfrak{p}})$ 
$+ a_{11}(b^{--}_{\mathfrak{p}}, m^-_{\mathfrak{p}}) 
a_{13}^{--}(m^-_{\mathfrak{p}}, {\mathfrak{p}})$ \\ 
$+ (a_{11}(b^{--}_{\mathfrak{p}}, t_{\mathfrak{p}}) 
- \mu^{-1} a_{11}(b^{--}_{\mathfrak{p}}, m^-_{\mathfrak{p}}) 
a_{11}(m^-_{\mathfrak{p}}, t_{\mathfrak{p}})) 
a_{13}^{--}(t_{\mathfrak{p}}, {\mathfrak{p}})$ \\ 
$+ a_{2}({\bf tb}^-_{\mathfrak{p}})$ 
$+ a_{2}({\bf mb}^+_{\mathfrak{p}})$ 
$+ \mu^{-1} a_{11}(b^{--}_{\mathfrak{p}}, m^-_{\mathfrak{p}}) 
a_{2}({\bf tm}_{\mathfrak{p}}) 
a_{11}(m^-_{\mathfrak{p}}, b^{--}_{\mathfrak{p}})$ 
$- a_{2}({\bf tb}^+_{\mathfrak{p}})$ 
$- a_{2}({\bf mb}^-_{\mathfrak{p}})$ \\ 
$- \mu a_{22}({\bf tb}^-_{\mathfrak{p}}, {\bf mb}^+_{\mathfrak{p}})$ 
$+ a_{11}(b^{--}_{\mathfrak{p}}, m^-_{\mathfrak{p}}) 
a_{22}({\bf tm}_{\mathfrak{p}}, {\bf mb}^+_{\mathfrak{p}})$ 
$- \mu^{-1} a_{11}(b^{--}_{\mathfrak{p}}, m^-_{\mathfrak{p}}) 
a_{22}({\bf tm}_{\mathfrak{p}}, {\bf tb}^-_{\mathfrak{p}})$ \\ 
$+ \mu a_{22}({\bf mb}^-_{\mathfrak{p}}, {\bf tb}^+_{\mathfrak{p}})$ 
$- a_{21}({\bf mb}^-_{\mathfrak{p}}, t_{\mathfrak{p}}) 
a_{12}(t_{\mathfrak{p}}, {\bf tb}^+_{\mathfrak{p}})$ \\ 
$+ (\mu^{-1} a_{11}(b^{--}_{\mathfrak{p}}, m^-_{\mathfrak{p}}) 
a_{21}({\bf tm}_{\mathfrak{p}}, m^+_{\mathfrak{p}}) 
+ a_{12}(b^{--}_{\mathfrak{p}}, {\bf tm}_{\mathfrak{p}}) 
- \mu^{-1} a_{11}(b^{--}_{\mathfrak{p}}, t_{\mathfrak{p}}) 
a_{12}(t_{\mathfrak{p}}, {\bf tm}_{\mathfrak{p}}))$ \\ 
$(a_{12}(m^+_{\mathfrak{p}}, {\bf tb}^-_{\mathfrak{p}}) 
+ a_{21}({\bf tm}_{\mathfrak{p}}, b^{--}_{\mathfrak{p}}) 
- a_{21}({\bf tm}_{\mathfrak{p}}, t_{\mathfrak{p}}) 
a_{11}(t_{\mathfrak{p}}, b^{--}_{\mathfrak{p}}))$ \\ 
$+ (a_{21}({\bf tb}^-_{\mathfrak{p}}, m^+_{\mathfrak{p}}) 
+ a_{12}(b^{--}_{\mathfrak{p}}, {\bf tm}_{\mathfrak{p}}) 
- \mu^{-1} a_{11}(b^{--}_{\mathfrak{p}}, t_{\mathfrak{p}}) 
a_{12}(t_{\mathfrak{p}}, {\bf tm}_{\mathfrak{p}})) 
a_{12}(m^+_{\mathfrak{p}}, {\bf mb}^+_{\mathfrak{p}})$, \\ 
$\partial a_{33}^{--}({\mathfrak{p}}, {\mathfrak{q}})$ 
$= \mu a_{23}^{--}({\bf tb}^-_{\mathfrak{p}}, {\mathfrak{q}}) 
+ a_{23}^{--}({\bf tb}^+_{\mathfrak{p}}, {\mathfrak{q}}) 
- a_{21}({\bf tb}^-_{\mathfrak{p}}, m^+_{\mathfrak{p}}) 
a_{13}^{--}(m^+_{\mathfrak{p}}, {\mathfrak{q}})$ \\ 
$- \mu a_{23}^{--}({\bf mb}^-_{\mathfrak{p}}, {\mathfrak{q}}) 
- a_{23}^{--}({\bf mb}^+_{\mathfrak{p}}, {\mathfrak{q}}) 
+ a_{21}({\bf mb}^-_{\mathfrak{p}}, t_{\mathfrak{p}}) 
a_{13}^{--}(t_{\mathfrak{p}}, {\mathfrak{q}})$ \\ 
$- a_{11}(b^{--}_{\mathfrak{p}}, m^-_{\mathfrak{p}}) 
a_{23}^{--}({\bf tm}_{\mathfrak{p}}, {\mathfrak{q}}) 
- a_{12}(b^{--}_{\mathfrak{p}}, {\bf tm}_{\mathfrak{p}}) 
a_{13}^{--}(m^+_{\mathfrak{p}}, {\mathfrak{q}})$ 
$+ \mu^{-1} a_{11}(b^{--}_{\mathfrak{p}}, t_{\mathfrak{p}}) 
a_{12}(t_{\mathfrak{p}}, {\bf tm}_{\mathfrak{p}}) 
a_{13}^{--}(m^+_{\mathfrak{p}}, {\mathfrak{q}})$ \\ 
$+ a_{32}^{--}({\mathfrak{p}}, {\bf tb}^-_{\mathfrak{q}}) 
+ \mu a_{32}^{--}({\mathfrak{p}}, {\bf tb}^+_{\mathfrak{q}}) 
- a_{31}^{--}({\mathfrak{p}}, m^+_{\mathfrak{q}}) 
a_{12}(m^+_{\mathfrak{q}}, {\bf tb}^-_{\mathfrak{q}})$ \\ 
$- a_{32}^{--}({\mathfrak{p}}, {\bf mb}^-_{\mathfrak{q}}) 
- \mu a_{32}^{--}({\mathfrak{p}}, {\bf mb}^+_{\mathfrak{q}}) 
+ a_{31}^{--}({\mathfrak{p}}, t_{\mathfrak{q}}) 
a_{12}(t_{\mathfrak{q}}, {\bf mb}^-_{\mathfrak{q}})$ \\ 
$- a_{32}^{--}({\mathfrak{p}}, {\bf tm}_{\mathfrak{q}}) 
a_{11}(m^-_{\mathfrak{q}}, b^{--}_{\mathfrak{q}}) 
- a_{31}^{--}({\mathfrak{p}}, m^+_{\mathfrak{q}}) 
a_{21}({\bf tm}_{\mathfrak{q}}, b^{--}_{\mathfrak{q}})$ 
$+ a_{31}^{--}({\mathfrak{p}}, m^+_{\mathfrak{q}}) 
a_{21}({\bf tm}_{\mathfrak{q}}, t_{\mathfrak{q}}) 
a_{11}(t_{\mathfrak{q}}, b^{--}_{\mathfrak{q}})$, \\ 
$\partial a_{b1}({\textsl{k}}, i)$ 
$= a_{21}({\bf dc}_{\textsl{k}}, i)$, 
$\partial a_{1b}(i, {\textsl{k}})$ 
$= a_{12}(i, {\bf dc}_{\textsl{k}})$, 
$\partial a_{b1}({\textsl{k}}, 0)$ 
$= a_{21}({\bf dc}_{\textsl{k}}, 0)$, 
$\partial a_{1b}(0, {\textsl{k}})$ 
$= a_{12}(0, {\bf dc}_{\textsl{k}})$, \\  
$\partial a_{b1}({\textsl{k}}, 1)$ 
$= a_{21}({\bf dc}_{\textsl{k}}, 1)$,  
$\partial a_{1b}(1, {\textsl{k}})$ 
$= a_{12}(1, {\bf dc}_{\textsl{k}})$, 
$\partial a_{b1}({\textsl{k}}, 2)$ 
$= a_{21}({\bf dc}_{\textsl{k}}, 2)$,  
$\partial a_{1b}(2, {\textsl{k}})$ 
$= a_{12}(2, {\bf dc}_{\textsl{k}})$, \\  
$\partial a_{b2}({\textsl{k}}, {\bf x})$ 
$= a_{22}({\bf dc}_{\textsl{k}}, {\bf x}) 
- a_{b1}^{--}({\textsl{k}}, u^-_{\bf x}) 
- \mu a_{b1}^{--}({\textsl{k}}, u^+_{\bf x}) 
+ a_{b1}^{--}({\textsl{k}}, o_{\bf x}) 
a_{11}(o_{\bf x}, u^-_{\bf x})$, \\ 
$\partial a_{2b}({\bf x}, {\textsl{k}})$ 
$= \mu a_{1b}^{--}(u^-_{\bf x}, {\textsl{k}}) 
+ a_{1b}^{--}(u^+_{\bf x}, {\textsl{k}}) 
- a_{11}(u^-_{\bf x}, o_{\bf x}) 
a_{1b}^{--}(o_{\bf x}, {\textsl{k}}) 
- a_{22}({\bf x}, {\bf dc}_{\textsl{k}})$, \\  
$\partial a_{b2}({\textsl{k}}, {\bf 0})$ 
$= a_{22}({\bf dc}_{\textsl{k}}, {\bf 0}) 
- a_{b1}^{--}({\textsl{k}}, 2) 
- \mu a_{b1}^{--}({\textsl{k}}, 0) 
+ a_{b1}^{--}({\textsl{k}}, 1) 
a_{11}(1, 2)$, \\  
$\partial a_{2b}({\bf 0}, {\textsl{k}})$ 
$= \mu a_{1b}^{--}(2, {\textsl{k}}) 
+ a_{1b}^{--}(0, {\textsl{k}}) 
- a_{11}(2, 1) 
a_{1b}^{--}(1, {\textsl{k}}) 
- a_{22}({\bf 0}, {\bf dc}_{\textsl{k}})$, \\ 
$\partial a_{b2}({\textsl{k}}, {\bf dc}_{\textsl{k}})$ 
$= a_{22}({\bf dc}_{\textsl{k}}, {\bf dc}_{\textsl{k}})$, 
$\partial a_{2b}({\bf dc}_{\textsl{k}}, {\textsl{k}})$ 
$= - a_{22}({\bf dc}_{\textsl{k}}, {\bf dc}_{\textsl{k}})$, \\  
$\partial a_{b3}^{--}({\textsl{k}}, {\mathfrak{p}})$ 
$= a_{23}^{--}({\bf dc}_{\textsl{k}}, {\mathfrak{p}}) 
+ a_{b2}({\textsl{k}}, {\bf tb}^-_{\mathfrak{p}})$ 
$+ \mu a_{b2}({\textsl{k}}, {\bf tb}^+_{\mathfrak{p}})$ 
$- a_{b1}({\textsl{k}}, m^+_{\mathfrak{p}}) 
a_{12}(m^+_{\mathfrak{p}}, {\bf tb}^-_{\mathfrak{p}})$ \\ 
$- a_{b2}({\textsl{k}}, {\bf mb}^-_{\mathfrak{p}})$ 
$- \mu a_{b2}({\textsl{k}}, {\bf mb}^+_{\mathfrak{p}})$ 
$+ a_{b1}({\textsl{k}}, t_{\mathfrak{p}}) 
a_{12}(t_{\mathfrak{p}}, {\bf mb}^-_{\mathfrak{p}})$ \\ 
$- a_{b2}({\textsl{k}}, {\bf tm}_{\mathfrak{p}}) 
a_{11}(m^-_{\mathfrak{p}}, b^{--}_{\mathfrak{p}})$ 
$- a_{b1}({\textsl{k}}, m^+_{\mathfrak{p}}) 
a_{21}({\bf tm}_{\mathfrak{p}}, b^{--}_{\mathfrak{p}})$ 
$+ a_{b1}({\textsl{k}}, m^+_{\mathfrak{p}}) 
a_{21}({\bf tm}_{\mathfrak{p}}, t_{\mathfrak{p}}) 
a_{11}(t_{\mathfrak{p}}, b^{--}_{\mathfrak{p}})$, \\  
$\partial a_{3b}^{--}({\mathfrak{p}}, {\textsl{k}})$ 
$= \mu a_{2b}({\bf tb}^-_{\mathfrak{p}}, {\textsl{k}})$ 
$+ a_{2b}({\bf tb}^+_{\mathfrak{p}}, {\textsl{k}})$
$- a_{21}({\bf tb}^-_{\mathfrak{p}}, m^+_{\mathfrak{p}}) 
a_{1b}(m^+_{\mathfrak{p}}, {\textsl{k}})$ \\ 
$- \mu a_{2b}({\bf mb}^-_{\mathfrak{p}}, {\textsl{k}})$ 
$- a_{2b}({\bf mb}^+_{\mathfrak{p}}, {\textsl{k}})$
$+ a_{21}({\bf mb}^-_{\mathfrak{p}}, t_{\mathfrak{p}}) 
a_{1b}(t_{\mathfrak{p}}, {\textsl{k}})$ \\ 
$- a_{11}(b^{--}_{\mathfrak{p}}, m^-_{\mathfrak{p}}) 
a_{2b}({\bf tm}_{\mathfrak{p}}, {\textsl{k}})$
$- a_{12}(b^{--}_{\mathfrak{p}}, {\bf tm}_{\mathfrak{p}}) 
a_{1b}(m^+_{\mathfrak{p}}, {\textsl{k}})$ 
$+ \mu^{-1} a_{11}(b^{--}_{\mathfrak{p}}, t_{\mathfrak{p}}) 
a_{12}(t_{\mathfrak{p}}, {\bf tm}_{\mathfrak{p}}) 
a_{1b}(m^+_{\mathfrak{p}}, {\textsl{k}})$ \\ 
$+ a_{32}^{--}({\mathfrak{p}}, {\bf dc}_{\textsl{k}})$, \\  
$\partial a_{bb}({\textsl{k}}, {\textsl{l}})$ 
$= a_{2b}({\bf dc}_{\textsl{k}}, {\textsl{l}})$ 
$- a_{b2}({\textsl{k}}, {\bf dc}_{\textsl{l}})$,  
$\partial a_{b}({\textsl{k}})$ 
$= a_{2}({\bf dc}_{\textsl{k}})$ 
$- \mu a_{b1}({\textsl{k}}, sh_{\textsl{k}})$ 
$+ a_{1b}(sh_{\textsl{k}}, {\textsl{k}})$, \\  
where $i \in \{ 3, \cdots, n \}$, 
${\bf x}, {\bf y} \in \{ {\bf 1}, \cdots, {\bf m} \}$, 
${\mathfrak{p}}, {\mathfrak{q}} \in \{ {\mathfrak{1}}, \cdots, {\mathfrak{t}} \}$, 
${\textsl{k}}, {\textsl{l}} \in \{ {\textsl{1}}, \cdots, {\textsl{b}} \}$.

We define a tame isomorphism 
$\varphi \colon (CR^{--}(D_2), \partial) 
\to (CR^{--1}(D_2), \partial^1)$ 
by \\ 
$\varphi(a_{11}(0, i)) 
= a_{11}^1(0, i) 
- \mu a_{11}^1(2, i) 
+ a_{11}^1(2, 1) 
a_{11}^1(1, i)$, \\ 
$\varphi(a_{11}(i, 0)) 
= \mu^{-1} (a_{11}^1(i, 0) 
- a_{11}^1(i, 2) 
+ a_{11}^1(i, 1) 
a_{11}^1(1, 2))$, \\ 
$\varphi(a_{11}(i, j)) = a_{11}^1(i, j)$, 
$\varphi(a_{21}({\bf x}, \ell)) = a_{21}^1({\bf x}, \ell)$, 
$\varphi(a_{12}(\ell, {\bf x})) = a_{12}^1(\ell, {\bf x})$, \\ 
$\varphi(a_{22}({\bf x}, {\bf y})) = a_{22}^1({\bf x}, {\bf y})$, 
$\varphi(a_{2}({\bf x})) = a_{2}^1({\bf x})$, 
$\varphi(a_{31}^{--}({\mathfrak{p}}, \ell)) = a_{31}^{--1}({\mathfrak{p}}, \ell)$, 
$\varphi(a_{13}^{--}(\ell, {\mathfrak{p}})) = a_{13}^{--1}(\ell, {\mathfrak{p}})$, \\ 
$\varphi(a_{32}^{--}({\mathfrak{p}}, {\bf x})) 
= a_{32}^{--1}({\mathfrak{p}}, {\bf x})$, 
$\varphi(a_{23}^{--}({\bf x}, {\mathfrak{p}})) 
= a_{23}^{--1}({\bf x}, {\mathfrak{p}})$, \\ 
$\varphi(a_{3}^{--}({\mathfrak{p}})) = a_{3}^{--1}({\mathfrak{p}})$, 
$\varphi(a_{33}^{--}({\mathfrak{p}}, {\mathfrak{q}})) 
= a_{33}^{--1}({\mathfrak{p}}, {\mathfrak{q}})$, \\ 
$\varphi(a_{b1}({\textsl{k}}, \ell)) = a_{b1}^1({\textsl{k}}, \ell)$, 
$\varphi(a_{1b}(\ell, {\textsl{k}})) = a_{1b}^1(\ell, {\textsl{k}})$, 
$\varphi(a_{b2}({\textsl{k}}, {\bf x})) 
= a_{b2}^1({\textsl{k}}, {\bf x})$, 
$\varphi(a_{2b}({\bf x}, {\textsl{k}})) 
= a_{2b}^1({\bf x}, {\textsl{k}})$, \\ 
$\varphi(a_{b3}^{--}({\textsl{k}}, {\mathfrak{p}})) 
= a_{b3}^{--1}({\textsl{k}}, {\mathfrak{p}})$, 
$\varphi(a_{3b}^{--}({\mathfrak{p}}, {\textsl{k}})) 
= a_{3b}^{--1}({\mathfrak{p}}, {\textsl{k}})$, 
$\varphi(a_{b}({\textsl{k}})) = a_{b}^1({\textsl{k}})$, 
$\varphi(a_{bb}({\textsl{k}}, {\textsl{l}})) 
= a_{bb}^1({\textsl{k}}, {\textsl{l}})$, \\ 
where $i \neq j \in \{ 1, \cdots, n \}$, 
$\ell \in \{0, 1, \cdots, n \}$, 
${\bf x}, {\bf y} \in \{ {\bf 0}, {\bf 1}, \cdots, {\bf m} \}$, 
${\mathfrak{p}}, {\mathfrak{q}} \in \{ {\mathfrak{1}}, \cdots, {\mathfrak{t}} \}$, 
${\textsl{k}}, {\textsl{l}} \in \{ {\textsl{1}}, \cdots, {\textsl{b}} \}$. 
For example, from the equation 
$\partial a_{21}({\bf 0}, i)$ 
$= \mu a_{11}(2, i) 
+ a_{11}(0, i) 
- a_{11}(2, 1) 
a_{11}(1, i)$ 
in $(CR^{--}(D_2), \partial)$, 
we obtain 
$\partial^1 a_{21}^1({\bf 0}, i)$ 
$= \mu a_{11}^1(2, i) 
+ (a_{11}^1(0, i) 
- \mu a_{11}^1(2, i) 
+ a_{11}^1(2, 1) 
a_{11}^1(1, i)) 
- a_{11}^1(2, 1) 
a_{11}^1(1, i)$ 
$= a_{11}^1(0, i)$ 
in $(CR^{--1}(D_2), \partial^1)$. 
We notice equations 
$\partial^1 a_{21}^1({\bf 0}, i) = a_{11}^1(0, i)$, 
$\partial^1 a_{11}^1(0, i) = 0$,  
and 
$\partial^1 a_{12}^1(i, {\bf 0}) = - a_{11}^1(i, 0)$, 
$\partial^1 a_{11}^1(i, 0) = 0$ 
in $(CR^{--1}(D_2), \partial^1)$.

We eliminate pairs of generators 
$(a_{21}^1({\bf 0}, i), a_{11}^1(0, i))$ 
and 
$(a_{12}^1(i, {\bf 0}), a_{11}^1(i, 0))$ 
by a sequence of destabilizations on $(CR^{--1}(D_2), \partial^1)$ 
for $i \in \{ 1, \cdots, n \}$, 
and we obtain a differential graded algebra 
$(CR^{--2}(D_2), \partial^2)$. 
For example, from the equation 
$\partial^1 a_{2}^1({\bf 0})$ 
$= \mu a_{21}^1({\bf 0}, 0) 
+ \mu a_{12}^1(2, {\bf 0}) 
- a_{11}^1(2, 1) 
a_{12}^1(1, {\bf 0})$ in $(CR^{--1}(D_2), \partial^1)$, 
we obtain 
$\partial^2 a_{2}^2({\bf 0})$ 
$= \mu a_{21}^2({\bf 0}, 0)$ in $(CR^{--2}(D_2), \partial^2)$. 
We notice 
equations 
$\partial^2 a_{2}^2({\bf 0}) 
= \mu a_{21}^2({\bf 0}, 0)$, 
$\partial^2 a_{21}^2({\bf 0}, 0) = 0$, and 
$\partial^2 a_{22}^2({\bf 0}, {\bf 0}) 
= a_{12}^2(0, {\bf 0}) 
+ \mu a_{21}^2({\bf 0}, 0)$, 
$\partial^2 a_{12}^2(0, {\bf 0}) = 0$ 
in $(CR^{--2}(D_2), \partial^2)$.

We eliminate the pair of generators 
$(a_{2}^2({\bf 0}), a_{21}^2({\bf 0}, 0))$ 
by a destabilization on $(CR^{--2}(D_2), \partial^2)$, 
and we obtain $(CR^{--3}(D_2), \partial^3)$. 
From equations 
$\partial^2 a_{22}^2({\bf 0}, {\bf 0}) 
= a_{12}^2(0, {\bf 0}) 
+ \mu a_{21}^2({\bf 0}, 0)$ and 
$\partial^2 a_{12}^2(0, {\bf 0}) = 0$ 
in $(CR^{--2}(D_2), \partial^2)$, 
we obtain equations 
$\partial^3 a_{22}^3({\bf 0}, {\bf 0}) 
= a_{12}^3(0, {\bf 0})$ 
and 
$\partial^3 a_{12}^3(0, {\bf 0}) = 0$ 
in $(CR^{--3}(D_2), \partial^3)$, 
respectively. 
We eliminate 
the pair of generators $(a_{22}^3({\bf 0}, {\bf 0}), a_{12}^3(0, {\bf 0}))$ 
by a destabilization on $(CR^{--3}(D_2), \partial^3)$, 
and we obtain $(CR^{--4}(D_2), \partial^4)$.

We define a tame isomorphism 
$\varphi^4 \colon (CR^{--4}(D_2), \partial^4) 
\to (CR^{--5}(D_2), \partial^5)$ 
by \\ 
$\varphi^4 (a_{21}^4({\bf x}, 0)) 
= \mu^{-1} (a_{21}^5({\bf x}, 0) 
- a_{21}^5({\bf x}, 2) 
+ a_{21}^5({\bf x}, 1) 
a_{11}^5(1, 2))$, \\ 
$\varphi^4 (a_{12}^4(0, {\bf x})) 
= a_{12}^5(0, {\bf x}) 
- \mu a_{12}^5(2, {\bf x}) 
+ a_{11}^5(2, 1) 
a_{12}^5(1, {\bf x})$, \\ 
$\varphi^4(a_{11}^4(i, j)) = a_{11}^5(i, j)$, 
$\varphi^4(a_{21}^4({\bf x}, i)) = a_{21}^5({\bf x}, i)$, 
$\varphi^4(a_{12}^4(i, {\bf x})) = a_{12}^5(i, {\bf x})$, \\ 
$\varphi^4(a_{22}^4({\bf x}, {\bf y})) = a_{22}^5({\bf x}, {\bf y})$, 
$\varphi^4(a_{2}^4({\bf x})) = a_{2}^5({\bf x})$, \\ 
$\varphi^4(a_{31}^{--4}({\mathfrak{p}}, \ell)) 
= a_{31}^{--5}({\mathfrak{p}}, \ell)$, 
$\varphi^4(a_{13}^{--4}(\ell, {\mathfrak{p}})) 
= a_{13}^{--5}(\ell, {\mathfrak{p}})$, \\ 
$\varphi^4(a_{32}^{--4}({\mathfrak{p}}, {\bf z})) 
= a_{32}^{--5}({\mathfrak{p}}, {\bf z})$, 
$\varphi^4(a_{23}^{--4}({\bf z}, {\mathfrak{p}})) 
= a_{23}^{--5}({\bf z}, {\mathfrak{p}})$, \\ 
$\varphi^4(a_{3}^{--4}({\mathfrak{p}})) = a_{3}^{--5}({\mathfrak{p}})$, 
$\varphi^4(a_{33}^{--4}({\mathfrak{p}}, {\mathfrak{q}})) 
= a_{33}^{--5}({\mathfrak{p}}, {\mathfrak{q}})$, \\ 
$\varphi^4(a_{b1}^4({\textsl{k}}, \ell)) = a_{b1}^5({\textsl{k}}, \ell)$, 
$\varphi^4(a_{1b}^4(\ell, {\textsl{k}})) = a_{1b}^5(\ell, {\textsl{k}})$, \\ 
$\varphi^4(a_{b2}^4({\textsl{k}}, {\bf z})) 
= a_{b2}^5({\textsl{k}}, {\bf z})$, 
$\varphi^4(a_{2b}^4({\bf z}, {\textsl{k}})) 
= a_{2b}^5({\bf z}, {\textsl{k}})$, \\ 
$\varphi^4(a_{b3}^{--4}({\textsl{k}}, {\mathfrak{p}})) 
= a_{b3}^{--5}({\textsl{k}}, {\mathfrak{p}})$, 
$\varphi^4(a_{3b}^{--4}({\mathfrak{p}}, {\textsl{k}})) 
= a_{3b}^{--5}({\mathfrak{p}}, {\textsl{k}})$, \\ 
$\varphi^4(a_{b}^4({\textsl{k}})) = a_{b}^5({\textsl{k}})$, 
$\varphi^4(a_{bb}^4({\textsl{k}}, {\textsl{l}})) 
= a_{bb}^5({\textsl{k}}, {\textsl{l}})$, \\ 
where $i \neq j \in \{ 1, \cdots, n \}$, 
$\ell \in \{0, 1, \cdots, n \}$, 
${\bf x}, {\bf y} \in \{ {\bf 1}, \cdots, {\bf m} \}$, 
${\bf z} \in \{ {\bf 0}, {\bf 1}, \cdots, {\bf m} \}$, 
${\mathfrak{p}}, {\mathfrak{q}} \in \{ {\mathfrak{1}}, \cdots, {\mathfrak{t}} \}$, 
${\textsl{k}}, {\textsl{l}} \in \{ {\textsl{1}}, \cdots, {\textsl{b}} \}$. 
For example, from the equation \\ 
$\partial^4 a_{21}^4({\bf x}, 0)$ 
$= \mu (\mu^{-1} (- a_{11}^4(u^-_{\bf x}, 2) 
+ a_{11}^4(u^-_{\bf x}, 1) 
a_{11}^4(1, 2))) 
+ (\mu^{-1} (- a_{11}^4(u^+_{\bf x}, 2) 
+ a_{11}^4(u^+_{\bf x}, 1) 
a_{11}^4(1, 2)))$ \\ 
$- a_{11}^4(u^-_{\bf x}, o_{\bf x}) 
(\mu^{-1} (- a_{11}^4(o_{\bf x}, 2) 
+ a_{11}^4(o_{\bf x}, 1) 
a_{11}^4(1, 2)))$ 
in $(CR^{--4}(D_2), \partial^4)$, 
we obtain \\ 
$\partial^5 (\mu^{-1} (a_{21}^5({\bf x}, 0) 
- a_{21}^5({\bf x}, 2) 
+ a_{21}^5({\bf x}, 1) 
a_{11}^5(1, 2)))$ \\ 
$= \mu (\mu^{-1} (- a_{11}^5(u^-_{\bf x}, 2) 
+ a_{11}^5(u^-_{\bf x}, 1) 
a_{11}^5(1, 2))) 
+ (\mu^{-1} (- a_{11}^5(u^+_{\bf x}, 2) 
+ a_{11}^5(u^+_{\bf x}, 1) 
a_{11}^5(1, 2)))$ \\ 
$- a_{11}^5(u^-_{\bf x}, o_{\bf x}) 
(\mu^{-1} (- a_{11}^5(o_{\bf x}, 2) 
+ a_{11}^5(o_{\bf x}, 1) 
a_{11}^5(1, 2)))$ 
in $(CR^{--5}(D_2), \partial^5)$. \\ 
Since we have equations 
$\partial^5 a_{21}^5({\bf x}, 1)$ 
$= \mu a_{11}^5(u^-_{\bf x}, 1) 
+ a_{11}^5(u^+_{\bf x}, 1) 
- a_{11}^5(u^-_{\bf x}, o_{\bf x}) 
a_{11}^5(o_{\bf x}, 1)$ and  \\ 
$\partial^5 a_{21}^5({\bf x}, 2)$ 
$= \mu a_{11}^5(u^-_{\bf x}, 2) 
+ a_{11}^5(u^+_{\bf x}, 2) 
- a_{11}^5(u^-_{\bf x}, o_{\bf x}) 
a_{11}^5(o_{\bf x}, 2)$ 
in $(CR^{--5}(D_2), \partial^5)$, \\ 
we obtain 
$\partial^5 a_{21}^5({\bf x}, 0) = 0$ 
in $(CR^{--5}(D_2), \partial^5)$. 
We notice equations 
$\partial^5 a_{22}^5({\bf x}, {\bf 0})$ 
$= a_{21}^5({\bf x}, 0)$, 
$\partial^5 a_{21}^5({\bf x}, 0) = 0$, and 
$\partial^5 a_{22}^5({\bf 0}, {\bf x})$ 
$= a_{12}^5(0, {\bf x})$, 
$\partial^5 a_{12}^5(0, {\bf x}) = 0$ 
in $(CR^{--5}(D_2), \partial^5)$. 

We eliminate pairs of generators 
$(a_{22}^5({\bf x}, {\bf 0}), a_{21}^5({\bf x}, 0))$ 
and 
$(a_{22}^5({\bf 0}, {\bf x}), a_{12}^5(0, {\bf x}))$ 
by a sequence of destabilizations 
on $(CR^{--5}(D_2), \partial^5)$ 
for ${\bf x} \in \{ {\bf 1}, \cdots, {\bf m} \}$, 
and we obtain $(CR^{--6}(D_2), \partial^6)$.

We define a tame isomorphism 
$\varphi^6 \colon (CR^{--6}(D_2), \partial^6) 
\to (CR^{--7}(D_2), \partial^7)$ 
by \\ 
$\varphi^6 (a_{31}^{--6}({\mathfrak{p}}, 0)) 
= \mu^{-1} (a_{31}^{--7}({\mathfrak{p}}, 0) 
- a_{31}^{--7}({\mathfrak{p}}, 2) 
+ a_{31}^{--7}({\mathfrak{p}}, 1) 
a_{11}^7(1, 2))$, \\ 
$\varphi^6 (a_{13}^{--6}(0, {\mathfrak{p}})) 
= a_{13}^{--7}(0, {\mathfrak{p}}) 
- \mu a_{13}^{--7}(2, {\mathfrak{p}}) 
+ a_{11}^7(2, 1) 
a_{13}^{--7}(1, {\mathfrak{p}})$, \\ 
$\varphi^6(a_{11}^6(i, j)) = a_{11}^7(i, j)$, 
$\varphi^6(a_{21}^6({\bf w}, i)) = a_{21}^7({\bf w}, i)$, 
$\varphi^6(a_{12}^6(i, {\bf w})) = a_{12}^7(i, {\bf w})$, \\ 
$\varphi^6(a_{22}^6({\bf w}, {\bf z})) = a_{22}^7({\bf w}, {\bf z})$, 
$\varphi^6(a_{2}^6({\bf w})) = a_{2}^7({\bf w})$, \\ 
$\varphi^6(a_{31}^{--6}({\mathfrak{p}}, i)) = a_{31}^{--7}({\mathfrak{p}}, i)$, 
$\varphi^6(a_{13}^{--6}(i, {\mathfrak{p}})) = a_{13}^{--7}(i, {\mathfrak{p}})$, \\ 
$\varphi^6(a_{32}^{--6}({\mathfrak{p}}, {\bf x})) 
= a_{32}^{--7}({\mathfrak{p}}, {\bf x})$, 
$\varphi^6(a_{23}^{--6}({\bf x}, {\mathfrak{p}})) 
= a_{23}^{--7}({\bf x}, {\mathfrak{p}})$, \\ 
$\varphi^6(a_{3}^{--6}({\mathfrak{p}})) = a_{3}^{--1}({\mathfrak{p}})$, 
$\varphi^6(a_{33}^{--6}({\mathfrak{p}}, {\mathfrak{q}})) 
= a_{33}^{--7}({\mathfrak{p}}, {\mathfrak{q}})$, \\ 
$\varphi^6(a_{b1}^6({\textsl{k}}, \ell)) = a_{b1}^7({\textsl{k}}, \ell)$, 
$\varphi^6(a_{1b}^6(\ell, {\textsl{k}})) = a_{1b}^7(\ell, {\textsl{k}})$, \\ 
$\varphi^6(a_{b2}^6({\textsl{k}}, {\bf x})) 
= a_{b2}^7({\textsl{k}}, {\bf x})$, 
$\varphi^6(a_{2b}^6({\bf x}, {\textsl{k}})) 
= a_{2b}^7({\bf x}, {\textsl{k}})$, \\ 
$\varphi^6(a_{b3}^{--6}({\textsl{k}}, {\mathfrak{p}})) 
= a_{b3}^{--7}({\textsl{k}}, {\mathfrak{p}})$, 
$\varphi^6(a_{3b}^{--6}({\mathfrak{p}}, {\textsl{k}})) 
= a_{3b}^{--7}({\mathfrak{p}}, {\textsl{k}})$, \\ 
$\varphi^6(a_{b}^6({\textsl{k}})) = a_{b}^7({\textsl{k}})$, 
$\varphi^6(a_{bb}^6({\textsl{k}}, {\textsl{l}})) 
= a_{bb}^7({\textsl{k}}, {\textsl{l}})$, \\ 
where $i \neq j \in \{ 1, \cdots, n \}$, 
$\ell \in \{0, 1, \cdots, n \}$, 
${\bf x}, {\bf y} \in \{ {\bf 0}, {\bf 1}, \cdots, {\bf m} \}$, 
${\bf w}, {\bf z} \in \{ {\bf 1}, \cdots, {\bf m} \}$, 
${\mathfrak{p}}, {\mathfrak{q}} \in \{ {\mathfrak{1}}, \cdots, {\mathfrak{t}} \}$, 
${\textsl{k}}, {\textsl{l}} \in \{ {\textsl{1}}, \cdots, {\textsl{b}} \}$. 
For example, from the equation \\ 
$\partial^6 a_{31}^{--6}({\mathfrak{p}}, 0)$ 
$= \mu (\mu^{-1} (- a_{21}^6({\bf tb}^-_{\mathfrak{p}}, 2) 
+ a_{21}^6({\bf tb}^-_{\mathfrak{p}}, 1) 
a_{11}^6(1, 2)))$ \\ 
$+ (\mu^{-1} (- a_{21}^6({\bf tb}^+_{\mathfrak{p}}, 2) 
+ a_{21}^6({\bf tb}^+_{\mathfrak{p}}, 1) 
a_{11}^6(1, 2)))$ \\ 
$- a_{21}^6({\bf tb}^-_{\mathfrak{p}}, m^+_{\mathfrak{p}}) 
(\mu^{-1} (- a_{11}^6(m^+_{\mathfrak{p}}, 2) 
+ a_{11}^6(m^+_{\mathfrak{p}}, 1) 
a_{11}^6(1, 2)))$ \\ 
$- \mu (\mu^{-1} (- a_{21}^6({\bf mb}^-_{\mathfrak{p}}, 2) 
+ a_{21}^6({\bf mb}^-_{\mathfrak{p}}, 1) 
a_{11}^6(1, 2)))$ \\ 
$- (\mu^{-1} (- a_{21}^6({\bf mb}^+_{\mathfrak{p}}, 2) 
+ a_{21}^6({\bf mb}^+_{\mathfrak{p}}, 1) 
a_{11}^6(1, 2)))$ \\ 
$+ a_{21}^6({\bf mb}^-_{\mathfrak{p}}, t_{\mathfrak{p}}) 
(\mu^{-1} (- a_{11}^6(t_{\mathfrak{p}}, 2) 
+ a_{11}^6(t_{\mathfrak{p}}, 1) 
a_{11}^6(1, 2)))$ \\ 
$- a_{11}^6(b^{--}_{\mathfrak{p}}, m^-_{\mathfrak{p}}) 
(\mu^{-1} (- a_{21}^6({\bf tm}_{\mathfrak{p}}, 2) 
+ a_{21}^6({\bf tm}_{\mathfrak{p}}, 1) 
a_{11}^6(1, 2)))$ \\ 
$- a_{12}^6(b^{--}_{\mathfrak{p}}, {\bf tm}_{\mathfrak{p}}) 
(\mu^{-1} (- a_{11}^6(m^+_{\mathfrak{p}}, 2) 
+ a_{11}^6(m^+_{\mathfrak{p}}, 1) 
a_{11}^6(1, 2)))$ \\ 
$+ \mu^{-1} a_{11}^6(b^{--}_{\mathfrak{p}}, t_{\mathfrak{p}}) 
a_{12}^6(t_{\mathfrak{p}}, {\bf tm}_{\mathfrak{p}}) 
(\mu^{-1} (- a_{11}^6(m^+_{\mathfrak{p}}, 2) 
+ a_{11}^6(m^+_{\mathfrak{p}}, 1) 
a_{11}^6(1, 2)))$ 
in $(CR^{--6}(D_2), \partial^6)$, \\ 
we obtain 
$\partial^7 (\mu^{-1} (a_{31}^{--7}({\mathfrak{p}}, 0) 
- a_{31}^{--7}({\mathfrak{p}}, 2) 
+ a_{31}^{--7}({\mathfrak{p}}, 1) 
a_{11}^7(1, 2)))$ \\ 
$= \mu (\mu^{-1} (- a_{21}^7({\bf tb}^-_{\mathfrak{p}}, 2) 
+ a_{21}^7({\bf tb}^-_{\mathfrak{p}}, 1) 
a_{11}^7(1, 2)))$ \\ 
$+ (\mu^{-1} (- a_{21}^7({\bf tb}^+_{\mathfrak{p}}, 2) 
+ a_{21}^7({\bf tb}^+_{\mathfrak{p}}, 1) 
a_{11}^7(1, 2)))$ \\ 
$- a_{21}^7({\bf tb}^-_{\mathfrak{p}}, m^+_{\mathfrak{p}}) 
(\mu^{-1} (- a_{11}^7(m^+_{\mathfrak{p}}, 2) 
+ a_{11}^7(m^+_{\mathfrak{p}}, 1) 
a_{11}^7(1, 2)))$ \\ 
$- \mu (\mu^{-1} (- a_{21}^7({\bf mb}^-_{\mathfrak{p}}, 2) 
+ a_{21}^7({\bf mb}^-_{\mathfrak{p}}, 1) 
a_{11}^7(1, 2)))$ \\ 
$- (\mu^{-1} (- a_{21}^7({\bf mb}^+_{\mathfrak{p}}, 2) 
+ a_{21}^7({\bf mb}^+_{\mathfrak{p}}, 1) 
a_{11}^7(1, 2)))$ \\ 
$+ a_{21}^7({\bf mb}^-_{\mathfrak{p}}, t_{\mathfrak{p}}) 
(\mu^{-1} (- a_{11}^7(t_{\mathfrak{p}}, 2) 
+ a_{11}^7(t_{\mathfrak{p}}, 1) 
a_{11}^7(1, 2)))$ \\ 
$- a_{11}^7(b^{--}_{\mathfrak{p}}, m^-_{\mathfrak{p}}) 
(\mu^{-1} (- a_{21}^7({\bf tm}_{\mathfrak{p}}, 2) 
+ a_{21}^7({\bf tm}_{\mathfrak{p}}, 1) 
a_{11}^7(1, 2)))$ \\ 
$- a_{12}^7(b^{--}_{\mathfrak{p}}, {\bf tm}_{\mathfrak{p}}) 
(\mu^{-1} (- a_{11}^7(m^+_{\mathfrak{p}}, 2) 
+ a_{11}^7(m^+_{\mathfrak{p}}, 1) 
a_{11}^7(1, 2)))$ \\ 
$+ \mu^{-1} a_{11}^7(b^{--}_{\mathfrak{p}}, t_{\mathfrak{p}}) 
a_{12}^7(t_{\mathfrak{p}}, {\bf tm}_{\mathfrak{p}}) 
(\mu^{-1} (- a_{11}^7(m^+_{\mathfrak{p}}, 2) 
+ a_{11}^7(m^+_{\mathfrak{p}}, 1) 
a_{11}^7(1, 2)))$ 
in $(CR^{--7}(D_2), \partial^7)$. \\ 
Since we have equations 
$\partial^7 a_{31}^{--7}({\mathfrak{p}}, 1)$ 
$= \mu a_{21}^7({\bf tb}^-_{\mathfrak{p}}, 1)$ 
$+ a_{21}^7({\bf tb}^+_{\mathfrak{p}}, 1)$
$- a_{21}^7({\bf tb}^-_{\mathfrak{p}}, m^+_{\mathfrak{p}}) 
a_{11}^7(m^+_{\mathfrak{p}}, 1)$ \\ 
$- \mu a_{21}^7({\bf mb}^-_{\mathfrak{p}}, 1)$ 
$- a_{21}^7({\bf mb}^+_{\mathfrak{p}}, 1)$
$+ a_{21}^7({\bf mb}^-_{\mathfrak{p}}, t_{\mathfrak{p}}) 
a_{11}^7(t_{\mathfrak{p}}, 1)$ \\ 
$- a_{11}^7(b^{--}_{\mathfrak{p}}, m^-_{\mathfrak{p}}) 
a_{21}^7({\bf tm}_{\mathfrak{p}}, 1)$
$- a_{12}^7(b^{--}_{\mathfrak{p}}, {\bf tm}_{\mathfrak{p}}) 
a_{11}^7(m^+_{\mathfrak{p}}, 1)$ 
$+ \mu^{-1} a_{11}^7(b^{--}_{\mathfrak{p}}, t_{\mathfrak{p}}) 
a_{12}^7(t_{\mathfrak{p}}, {\bf tm}_{\mathfrak{p}}) 
a_{11}^7(m^+_{\mathfrak{p}}, 1)$, \\ 
$\partial^7 a_{31}^{--7}({\mathfrak{p}}, 2)$ 
$= \mu a_{21}^7({\bf tb}^-_{\mathfrak{p}}, 2)$ 
$+ a_{21}^7({\bf tb}^+_{\mathfrak{p}}, 2)$
$- a_{21}^7({\bf tb}^-_{\mathfrak{p}}, m^+_{\mathfrak{p}}) 
a_{11}^7(m^+_{\mathfrak{p}}, 2)$ \\ 
$- \mu a_{21}^7({\bf mb}^-_{\mathfrak{p}}, 2)$ 
$- a_{21}^7({\bf mb}^+_{\mathfrak{p}}, 2)$
$+ a_{21}^7({\bf mb}^-_{\mathfrak{p}}, t_{\mathfrak{p}}) 
a_{11}^7(t_{\mathfrak{p}}, 2)$ \\ 
$- a_{11}^7(b^{--}_{\mathfrak{p}}, m^-_{\mathfrak{p}}) 
a_{21}^7({\bf tm}_{\mathfrak{p}}, 2)$
$- a_{12}^7(b^{--}_{\mathfrak{p}}, {\bf tm}_{\mathfrak{p}}) 
a_{11}^7(m^+_{\mathfrak{p}}, 2)$ 
$+ \mu^{-1} a_{11}^7(b^{--}_{\mathfrak{p}}, t_{\mathfrak{p}}) 
a_{12}^7(t_{\mathfrak{p}}, {\bf tm}_{\mathfrak{p}}) 
a_{11}^7(m^+_{\mathfrak{p}}, 2)$ \\ 
in $(CR^{--7}(D_2), \partial^7)$, 
we obtain  
$\partial a_{31}^{--7}({\mathfrak{p}}, 0) = 0$ 
in $(CR^{--7}(D_2), \partial^7)$. 
We notice equations 
$\partial^7 a_{32}^{--7}({\mathfrak{p}}, {\bf 0})$ 
$= - a_{31}^{--7}({\mathfrak{p}}, 0)$, 
$\partial a_{31}^{--7}({\mathfrak{p}}, 0) = 0$, and 
$\partial^7 a_{23}^{--7}({\bf 0}, {\mathfrak{p}})$ 
$= a_{13}^{--7}(0, {\mathfrak{p}})$, 
$\partial a_{13}^{--7}(0, {\mathfrak{p}}) = 0$ 
in $(CR^{--7}(D_2), \partial^7)$.

We eliminate pairs of generators 
$(a_{32}^{--7}({\mathfrak{p}}, {\bf 0}), a_{31}^{--7}({\mathfrak{p}}, 0))$ 
and 
$(a_{23}^{--7}({\bf 0}, {\mathfrak{p}}), a_{13}^{--7}(0, {\mathfrak{p}}))$ 
by a sequence of destabilizations on $(CR^{--7}(D_2), \partial^7)$ 
for ${\mathfrak{p}} \in \{ {\mathfrak{1}}, \cdots, {\mathfrak{t}} \}$, 
and we obtain $(CR^{--8}(D_2), \partial^8)$.

We define a tame isomorphism 
$\varphi^8 \colon (CR^{--8}(D_2), \partial^8) 
\to (CR^{--9}(D_2), \partial^9)$ 
by \\ 
$\varphi^8 (a_{b1}^8({\textsl{k}}, 0)) 
= \mu^{-1} (a_{b1}^9({\textsl{k}}, 0) 
- a_{b1}^9({\textsl{k}}, 2) 
+ a_{b1}^9({\textsl{k}}, 1) 
a_{11}^9(1, 2))$, \\ 
$\varphi^8 (a_{1b}^8(0, {\textsl{k}})) 
= a_{1b}^9(0, {\textsl{k}}) 
- \mu a_{1b}^9(2, {\textsl{k}}) 
+ a_{11}^9(2, 1) 
a_{1b}^9(1, {\textsl{k}})$, \\ 
$\varphi^8(a_{11}^8(i, j)) = a_{11}^9(i, j)$, 
$\varphi^8(a_{21}^8({\bf w}, i)) = a_{21}^9({\bf w}, i)$, 
$\varphi^8(a_{12}^8(i, {\bf w})) = a_{12}^9(i, {\bf w})$, \\ 
$\varphi^8(a_{22}^8({\bf w}, {\bf z})) = a_{22}^9({\bf w}, {\bf z})$, 
$\varphi^8(a_{2}^8({\bf w})) = a_{2}^9({\bf w})$, \\ 
$\varphi^8(a_{31}^{--8}({\mathfrak{p}}, i)) = a_{31}^{--9}({\mathfrak{p}}, i)$, 
$\varphi^8(a_{13}^{--8}(i, {\mathfrak{p}})) = a_{13}^{--9}(i, {\mathfrak{p}})$, \\ 
$\varphi^8(a_{32}^{--8}({\mathfrak{p}}, {\bf w})) 
= a_{32}^{--9}({\mathfrak{p}}, {\bf w})$, 
$\varphi^8(a_{23}^{--8}({\bf w}, {\mathfrak{p}})) 
= a_{23}^{--9}({\bf w}, {\mathfrak{p}})$, \\ 
$\varphi^8(a_{3}^{--8}({\mathfrak{p}})) = a_{3}^{--9}({\mathfrak{p}})$, 
$\varphi^8(a_{33}^{--8}({\mathfrak{p}}, {\mathfrak{q}})) 
= a_{33}^{--9}({\mathfrak{p}}, {\mathfrak{q}})$, \\ 
$\varphi^8(a_{b1}^8({\textsl{k}}, i)) = a_{b1}^9({\textsl{k}}, i)$, 
$\varphi^8(a_{1b}^8(i, {\textsl{k}})) = a_{1b}^9(i, {\textsl{k}})$, \\ 
$\varphi^8(a_{b2}^8({\textsl{k}}, {\bf x})) 
= a_{b2}^9({\textsl{k}}, {\bf x})$, 
$\varphi^8(a_{2b}^8({\bf x}, {\textsl{k}})) 
= a_{2b}^9({\bf x}, {\textsl{k}})$, \\ 
$\varphi^8(a_{b3}^{--8}({\textsl{k}}, {\mathfrak{p}})) 
= a_{b3}^{--9}({\textsl{k}}, {\mathfrak{p}})$, 
$\varphi^8(a_{3b}^{--8}({\mathfrak{p}}, {\textsl{k}})) 
= a_{3b}^{--9}({\mathfrak{p}}, {\textsl{k}})$, \\ 
$\varphi^8(a_{b}^8({\textsl{k}})) = a_{b}^9({\textsl{k}})$, 
$\varphi^8(a_{bb}^8({\textsl{k}}, {\textsl{l}})) 
= a_{bb}^9({\textsl{k}}, {\textsl{l}})$, \\ 
where $i \neq j \in \{ 1, \cdots, n \}$, 
${\bf x}, {\bf y} \in \{ {\bf 0}, {\bf 1}, \cdots, {\bf m} \}$, 
${\bf w}, {\bf z} \in \{ {\bf 1}, \cdots, {\bf m} \}$, 
${\mathfrak{p}}, {\mathfrak{q}} \in \{ {\mathfrak{1}}, \cdots, {\mathfrak{t}} \}$, 
${\textsl{k}}, {\textsl{l}} \in \{ {\textsl{1}}, \cdots, {\textsl{b}} \}$. 
For example, from the equation \\ 
$\partial^8 a_{b1}^8({\textsl{k}}, 0)$ 
$= \mu^{-1} (- a_{21}^8({\bf dc}_{\textsl{k}}, 2) 
+ a_{21}^8({\bf dc}_{\textsl{k}}, 1) 
a_{11}^8(1, 2))$ 
in $(CR^{--8}(D_2), \partial^8)$, 
we obtain \\ 
$\partial^9 (\mu^{-1} (a_{b1}^9({\textsl{k}}, 0) 
- a_{b1}^9({\textsl{k}}, 2) 
+ a_{b1}^9({\textsl{k}}, 1) 
a_{11}^9(1, 2)))$ 
$= \mu^{-1} (- a_{21}^9({\bf dc}_{\textsl{k}}, 2) 
+ a_{21}^8({\bf dc}_{\textsl{k}}, 1) 
a_{11}^8(1, 2))$ \\ 
in $(CR^{--9}(D_2), \partial^9)$. 
Since we have equations 
$\partial^9 a_{b1}^9({\textsl{k}}, 1)$ 
$= a_{21}^9({\bf dc}_{\textsl{k}}, 1)$ and 
$\partial^9 a_{b1}^9({\textsl{k}}, 2)$ 
$= a_{21}^9({\bf dc}_{\textsl{k}}, 2)$ 
in 
$(CR^{--9}(D_2), \partial^9)$, 
we obtain 
$\partial^9 a_{b1}^9({\textsl{k}}, 0) = 0$ 
in $(CR^{--9}(D_2), \partial^9)$. 
We notice equations 
$\partial^9 a_{b2}^9({\textsl{k}}, {\bf 0})$ 
$= - a_{b1}^9({\textsl{k}}, 0)$, 
$\partial^9 a_{b1}^9({\textsl{k}}, 0) = 0$, and 
$\partial^9 a_{2b}^9({\bf 0}, {\textsl{k}})$ 
$= a_{1b}^9(0, {\textsl{k}})$, 
$\partial^9 a_{1b}^9(0, {\textsl{k}}) = 0$ 
in $(CR^{--9}(D_2), \partial^9)$.

We eliminate pairs of generators 
$(a_{b2}^9({\textsl{k}}, {\bf 0}), a_{b1}^9({\textsl{k}}, 0))$ 
and 
$(a_{2b}^9({\bf 0}, {\textsl{k}}), a_{1b}^9(0, {\textsl{k}}))$ 
by a sequence of destabilizations 
on $(CR^{--9}(D_2), \partial^9)$ 
for ${\textsl{k}} \in \{ {\textsl{1}}, \cdots, {\textsl{b}} \}$, 
and we obtain $(CR^{--0}(D_2), \partial^0)$.

It is straightforward to see that 
$(CR^{--0}(D_2), \partial^0)$ 
is isomorphic to $(CR^{--} (D_1), \partial)$. 
This shows that 
$(CR^{--}(D_2), \partial)$ is stably tame isomorphic to 
$(CR^{--}(D_1), \partial)$. 
\end{proof} 

We say that the series of tame isomorphisms and destabilizations 
performed in the proof of Proposition \ref{roseman-3} 
is a {\it destabilization} {\it along} ${\bf 0} \to 0$, 
from the double curve with label ${\bf 0}$ to the sheet with label $0$, 
on $(CR^{--}(D_2), \partial)$.

\section{Roseman move I}

\begin{proposition} \label{roseman-1} 
Suppose that a diagram $D_2$ of a surface-knot $F$ 
is obtained from a diagram $D_1$ of $F$ 
by applying one Roseman move of type I. 
See Figure \ref{roseman14}. 
Then $(CR^{--} (D_2), \partial)$ is stably tame isomorphic to 
$(CR^{--} (D_1), \partial)$. 
\end{proposition}

\begin{figure}
\begin{center}
\includegraphics{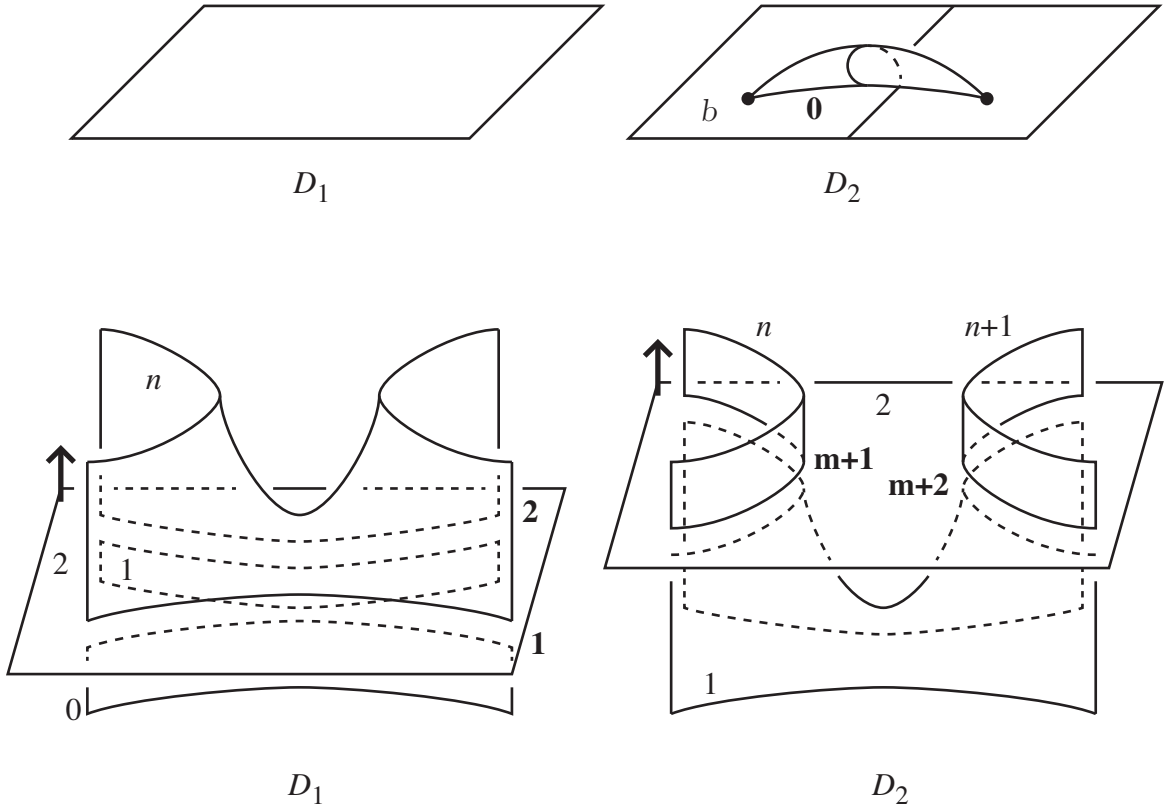}
\caption{Roseman move of type I (upper pair) and 
Roseman move of type IV (lower pair)}
\label{roseman14}
\end{center}
\end{figure}

\begin{proof} 
Let ${\textsl{b}}$ 
denote a label on 
the branch point of positive sign of $D_2$ 
created by the Roseman move, and 
let ${\bf 0}$ denote a label on 
the double curve of $D_2$ 
created by the move. 
See Figure \ref{roseman14}. 
Perform a destabilization 
along ${\textsl{b}} \to {\bf 0}$ on $(CR^{--}(D_2), \partial)$, and 
we obtain a differential graded algebra 
that is isomorphic to $(CR^{--}(D_1), \partial)$. 
This shows that $(CR^{--} (D_2), \partial)$ is stably tame isomorphic to 
$(CR^{--} (D_1), \partial)$. 
\end{proof}

\section{Roseman move IV}

\begin{proposition} \label{roseman-4} 
Suppose that a diagram $D_2$ of a surface-knot $F$ 
is obtained from a diagram $D_1$ of $F$ 
by applying one Roseman move of type IV. 
See Figure \ref{roseman14}. 
Then $(CR^{--} (D_1), \partial)$ 
is stably tame isomorphic to $(CR^{--} (D_2), \partial)$. 
\end{proposition}

\begin{proof} 
Let $0, 1, 2, n$ denote labels on sheets of $D_1$ 
involved in the Roseman move, 
and let ${\bf 1}, {\bf 2}$ denote labels on double curves of $D_1$ 
involved in the move. 
See Figure \ref{roseman14}. 
Perform destabilizations on $(CR^{--}(D_1), \partial)$ 
along ${\bf 1} \to 0$, and ${\bf 2} \to n$ in this order, 
and we obtain $(CR^{--1}(D_1), \partial^1)$. 

Let $1, 2, n, n+1$ denote labels on sheets of $D_2$ 
involved in the move, 
and let ${\bf m+1}, {\bf m+2}$ denote labels on double curves of $D_2$ 
involved in the move. 
Perform destabilizations on $(CR^{--}(D_2), \partial)$ 
along ${\bf m+2} \to n+1$, and ${\bf m+1} \to n$ in this order, 
and we obtain $(CR^{--2}(D_2), \partial^2)$. 

It is straightforward to see that 
$(CR^{--1}(D_1), \partial^1)$ is isomorphic to 
$(CR^{--2}(D_2), \partial^2)$. 
This shows that 
$(CR^{--}(D_1), \partial)$ is stably tame isomorphic to 
$(CR^{--}(D_2), \partial)$. 
\end{proof}

\section{Roseman move V}

\begin{proposition} \label{roseman-5} 
Suppose that a diagram $D_2$ of a surface-knot $F$ 
is obtained from a diagram $D_1$ of $F$ 
by applying one Roseman move of type V. 
See Figure \ref{roseman5}. 
Then $(CR^{--} (D_2), \partial)$ 
is stably tame isomorphic to $(CR^{--} (D_1), \partial)$. 
\end{proposition}

\begin{figure} 
\begin{center} 
\includegraphics{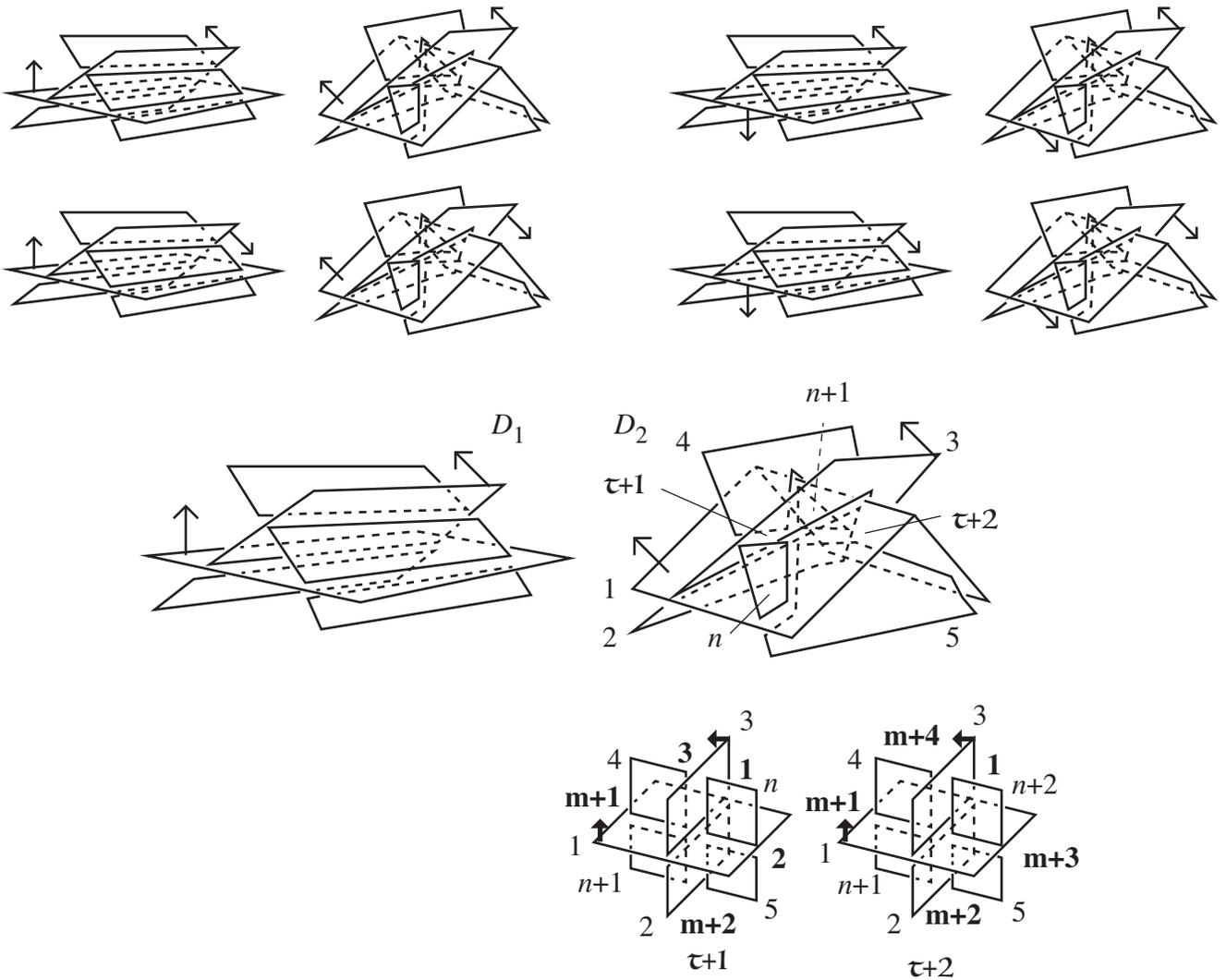} 
\caption{Roseman move of type V} 
\label{roseman5} 
\end{center} 
\end{figure}

\begin{proof} 
Three pieces of $F$ 
with three different heights, 
top, middle and bottom, are involved in 
the Roseman move. 
We suppose that 
the top sheet looks moving, and that 
the middle and the bottom sheets look fixed. 
Considering directions of positive normals 
to the top sheet and the middle sheets, 
there are four cases to study, as illustrated in the upper two rows of 
Figure \ref{roseman5}. 
We study the case illustrated in the lower two rows of Figure \ref{roseman5}. 
Other cases are studied similarly.

Let ${\mathfrak{t+1}}$ and ${\mathfrak{t+2}}$ 
denote labels on triple points of $D_2$ created by the move. 
See Figure \ref{roseman5}. 
Let ${\bf 1}, {\bf 2}, {\bf 3}$, ${\bf m+1}, {\bf m+2}, {\bf m+3}, {\bf m+4}$ 
denote labels on double curves of $D_2$ 
involved in the move, and 
let $1, 2, 3, 4, 5, n, n+1, n+2$ denote 
labels on sheets of $D_2$ 
involved in the move. 
Perform a sequence of destabilizations on $(CR^{--} (D_2), \partial)$ 
along 
${\bf m+2} \to n+1$, 
${\mathfrak{t+1}} \to {\bf m+1}$, 
${\bf m+3} \to n+2$, 
and 
${\mathfrak{t+2}} \to {\bf m+4}$ in this order, and 
we obtain a differential graded algebra 
that is isomorphic to $(CR^{--} (D_1), \partial)$. 
This shows that 
$(CR^{--}(D_2), \partial)$ is stably tame isomorphic to 
$(CR^{--}(D_1), \partial)$. 
\end{proof}

\section{Roseman move VI}

\begin{proposition} \label{roseman-6-1} 
Suppose that a diagram $D_2$ of a surface-knot $F$ 
is obtained from a diagram $D_1$ of $F$ 
by applying one Roseman move of type VI. 
Suppose also that the branch point involved in the Roseman move 
is of positive sign. 
See pairs in the upper two rows of Figure \ref{roseman6}. 
Then $(CR^{--} (D_1), \partial)$ 
is stably tame isomorphic to $(CR^{--} (D_2), \partial)$. 
\end{proposition}

\begin{figure}
\begin{center}
\includegraphics{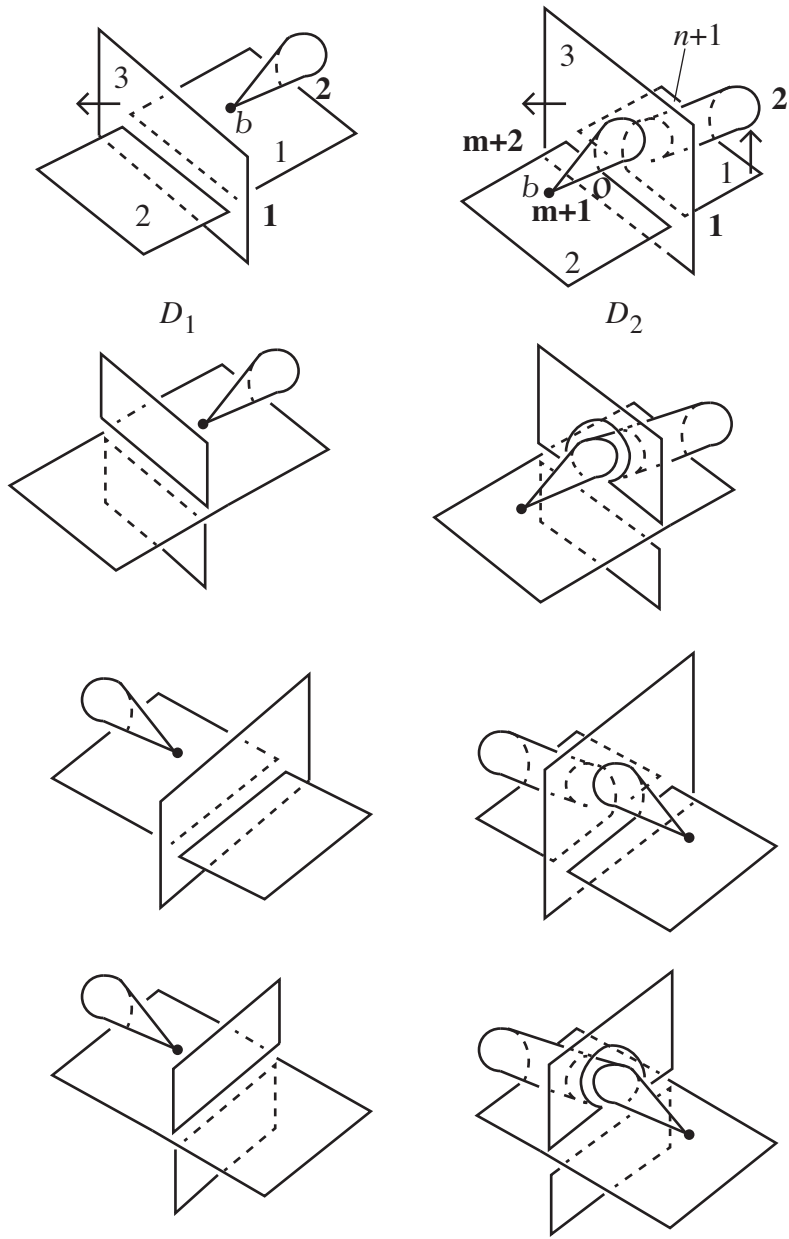}
\caption{Roseman move of type VI}
\label{roseman6}
\end{center}
\end{figure}

\begin{proof} 
There are two cases to study: 
the sheet containing the branch point of positive sign is 
either the under-sheet, as illustrated in the first row of 
Figure \ref{roseman6}, 
or the over-sheet, as illustrated in the second row of 
Figure \ref{roseman6}. 
In each case, 
considering the direction of the positive normal to the over-sheet, 
there are two cases to study. 
We study the case illustrated the first row in Figure \ref{roseman6}. 
Other cases are studied similarly. 

Let ${\mathfrak{0}}$ denote a label on the triple point of $D_2$ 
created by the Roseman move. 
See Figure \ref{roseman6}. 
Let ${\textsl{b}}$ denote a label on 
the branch point of positive sign of $D_2$ 
involved in the move. 
Let ${\bf 1}, {\bf 2}, {\bf m+1}, {\bf m+2}$ 
denote labels on double curves of $D_2$ involved in the move, and 
let $1, 2, 3, n+1$ denote labels on sheets of $D_2$ 
involved in the move. 
Perform a sequence of destabilizations on $(CR^{--} (D_2), \partial)$ 
along 
${\bf m+2} \to n+1$, ${\bf 1} \to 2$, 
and 
${\mathfrak{0}} \to {\bf m+1}$ in this order, and 
we obtain 
$(CR^{--2} (D_2), \partial^2)$.

Let ${\textsl{b}}$ denote a label on 
the branch point of positive sign of $D_1$ 
involved in the move. 
Let ${\bf 1}, {\bf 2}$ denote labels on double curves of $D_1$ 
involved in the move, and 
let $1, 2, 3$ denote labels on sheets of $D_1$ 
involved in the move. 
Perform a destabilization on $(CR^{--} (D_1), \partial)$ 
along ${\bf 1} \to 2$, and 
we obtain $(CR^{--1} (D_1), \partial^1)$. 

It is straightforward to see that 
$(CR^{--1} (D_1), \partial^1)$ is 
isomorphic to $(CR^{--2} (D_2), \partial^2)$. 
This shows that 
$(CR^{--}(D_1), \partial)$ is stably tame isomorphic to 
$(CR^{--}(D_2), \partial)$. 
\end{proof} 

Similar arguments as above prove the following proposition.  

\begin{proposition} \label{roseman-6-2} 
Suppose that a diagram $D_2$ of a surface-knot $F$ 
is obtained from a diagram $D_1$ of $F$ 
by applying one Roseman move of type VI. 
Suppose also that the branch point involved in the Roseman move 
is of negative sign. 
See pairs in the lower two rows of Figure \ref{roseman6}. 
Then $(CR^{--} (D_1), \partial)$ 
is stably tame isomorphic to $(CR^{--} (D_2), \partial)$. 
\end{proposition}

\section{Roseman move VII}

\begin{proposition} \label{roseman-7} 
Suppose that a diagram $D_2$ of a surface-knot $F$ 
is obtained from a diagram $D_1$ of $F$ 
by applying one Roseman move of type VII. 
See Figure \ref{roseman7}. 
Then $(CR^{--} (D_1), \partial)$ 
is stably tame isomorphic to $(CR^{--} (D_2), \partial)$. 
\end{proposition}

\begin{figure}
\begin{center}
\includegraphics{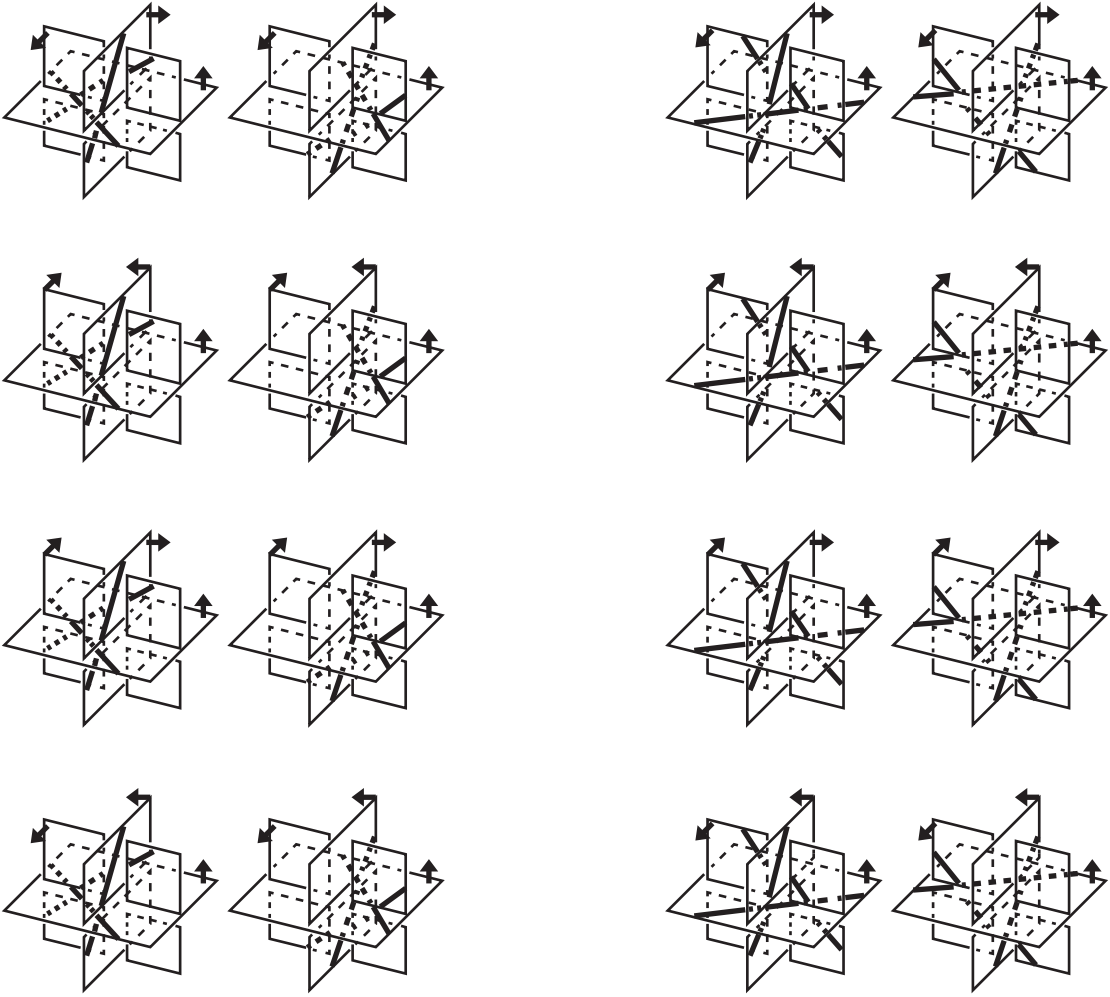}
\caption{Roseman move of type VII}
\label{roseman7-21}
\end{center}
\end{figure}

\begin{figure}
\begin{center}
\includegraphics{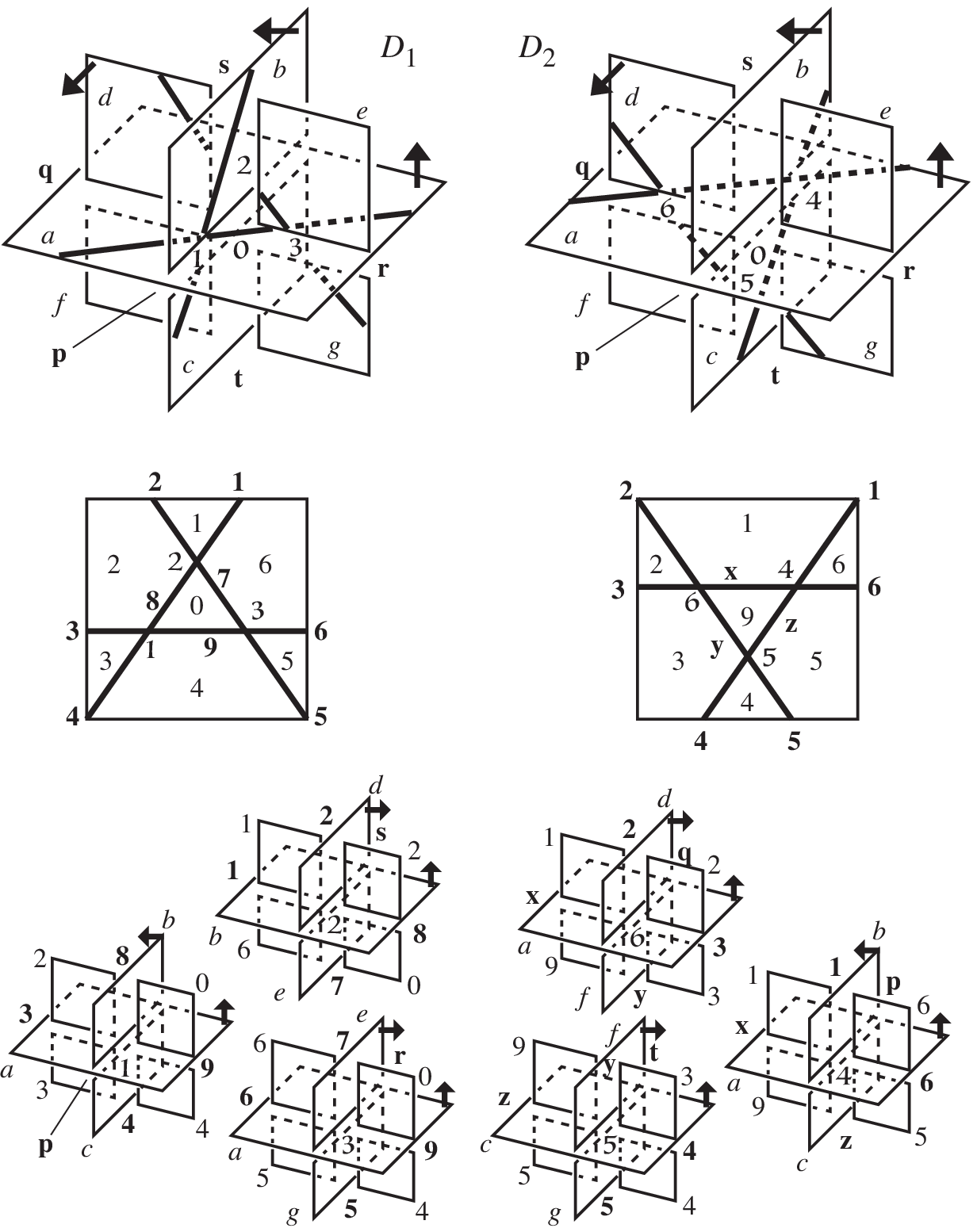}
\caption{Roseman move of type VII: 
labels around triple point labeled ${\mathfrak{0}}$ (top row), 
labels on bottom sheet (middle row), and 
labels around other triple points (bottom row)}
\label{roseman7}
\end{center}
\end{figure}

\begin{proof} 
Four pieces of $F$ with four different heights, 
top, upper-middle, lower-middle, bottom, 
are involved in the Roseman move. 
We suppose that the bottom sheets look moving, 
and that sheets with other three heights look fixed. 
See Figure \ref{roseman7-21}. 
The bottom sheets are indicated 
by intersections with other sheets around the triple point. 
Considering directions of positive normals to 
the top, upper-middle and lower-middle sheets around the triple point, 
there are 8 cases to study, 
as illustrated in Figure \ref{roseman7-21}. 
We study the case illustrated in Figure \ref{roseman7}. 
Other cases are studied similarly.

Let ${\mathfrak{0}}, {\mathfrak{1}}, {\mathfrak{2}}, {\mathfrak{3}}$ 
denote labels on triple points of $D_1$ 
involved in the Roseman move. 
See Figure \ref{roseman7}. 
Let ${\bf p}, {\bf q}, {\bf r}, {\bf s}, {\bf t}$, 
${\bf 1}, {\bf 2}, {\bf 3}, {\bf 4}$, ${\bf 5}, {\bf 6}, {\bf 7}, {\bf 8}, {\bf 9}$ 
denote labels on double curves of $D_1$ involved in the move, and 
let $a, b, c, d, e, f, g$, $0, 1, 2, 3, 4, 5, 6$ 
denote labels on sheets of $D_1$ involved in the move. 
Perform a sequence of destabilizations on $(CR^{--} (D_1), \partial)$ 
along ${\bf 9} \to 0$, 
${\mathfrak{1}} \to {\bf 8}$, 
${\bf 6} \to 6$, 
${\mathfrak{3}} \to {\bf 7}$, 
${\bf 1} \to 1$, 
and 
${\mathfrak{2}} \to {\bf 2}$ in this order, and 
we obtain 
$(CR^{--1} (D_1), \partial^1)$.

Let ${\mathfrak{0}}, {\mathfrak{4}}, {\mathfrak{5}}, {\mathfrak{6}}$ 
denote labels on triple points of $D_2$ 
involved in the move. 
Let 
${\bf p}, {\bf q}, {\bf r}, {\bf s}, {\bf t}$, 
${\bf 1}, {\bf 2}, {\bf 3}, {\bf 4}$, ${\bf 5}, {\bf 6}, {\bf x}, {\bf y}, {\bf z}$ 
denote labels on double curves of $D_2$ involved in the move, and 
let $a, b, c, d, e, f, g$, $1, 2, 3, 4, 5, 6, 9$ 
denote labels on sheets of $D_2$ involved in the move. 
Perform a sequence of destabilizations on $(CR^{--} (D_2), \partial)$ 
along ${\bf z} \to 9$, 
${\mathfrak{5}} \to {\bf y}$, 
${\bf 6} \to 6$, 
${\mathfrak{4}} \to {\bf 1}$, 
${\bf x} \to 1$, 
and 
${\mathfrak{6}} \to {\bf 2}$ in this order, and 
we obtain 
$(CR^{--2} (D_2), \partial^2)$.

It is straightforward to see that 
$(CR^{--1} (D_1), \partial^1)$ is 
isomorphic to $(CR^{--2} (D_2), \partial^2)$. 
This shows that 
$(CR^{--}(D_1), \partial)$ is stably tame isomorphic to 
$(CR^{--}(D_2), \partial)$. 
\end{proof}

\section{Proof of Theorem \ref{main-theorem}} 

We may assume by Theorem \ref{roseman-yashiro} that 
$D_2$ is obtained from $D_1$ by a finite sequence of Roseman moves 
of types I, III, IV, V, VI and VII. 
Propositions \ref{roseman-3}, \ref{roseman-1}, \ref{roseman-4}, 
\ref{roseman-5}, \ref{roseman-6-1}, \ref{roseman-6-2}, \ref{roseman-7} 
complete the proof of Theorem \ref{main-theorem}.

\section{$(CR^{-+}(D), \partial)$, $(CR^{+-}(D), \partial)$, 
and $(CR^{++}(D), \partial)$}

In this section, 
we define three differential graded algebras 
$(CR^{-+}(D), \partial)$, $(CR^{+-}(D), \partial)$, $(CR^{++}(D), \partial)$, 
where $D$ denotes a diagram of a surface-knot in ${\mathbb{R}}^4$. 
A differential $\partial$ of these three differential graded algebras 
on generators 
involving no triple point is 
defined to be 
the same as the differential $\partial$ of $(CR^{--}(D), \partial)$ 
on the corresponding generators. 
In the following, we define differentials on generators involving triple points. 
Let $D$ denote a diagram of a surface-knot in ${\mathbb{R}}^4$. 
We label sheets of $D$ by $1, \cdots, n$, 
connected components of double curves of $D$ 
by ${\bf 1}, \cdots, {\bf m}$, 
triple points of $D$ 
by ${\mathfrak{1}}, \cdots, {\mathfrak{t}}$, 
and 
branch points of positive sign of $D$ 
by ${\textsl{1}}, \cdots, {\textsl{b}}$. 
We suppose 
$i \in \{ 1, \cdots, n \}$, 
${\bf x} \in \{ {\bf 1}, \cdots, {\bf m} \}$, 
${\mathfrak{p}}, {\mathfrak{q}} \in \{ {\mathfrak{1}}, \cdots, {\mathfrak{t}} \}$, 
${\textsl{k}} \in \{ {\textsl{1}}, \cdots, {\textsl{b}} \}.$ 

\subsection{$(CR^{-+}(D), \partial)$} 
We define 
a differential $\partial$ 
on generators of $CR^{-+}(D)$ 
involving triple points 
as follows. \\ 
$\partial a_{31}^{-+}({\mathfrak{p}}, i)$ 
$= \mu a_{21}({\bf tb}^-_{\mathfrak{p}}, i) 
+ a_{21}({\bf tb}^+_{\mathfrak{p}}, i) 
- a_{21}({\bf tb}^+_{\mathfrak{p}}, m^+_{\mathfrak{p}}) 
a_{11}(m^+_{\mathfrak{p}}, i)$ \\ 
$- \mu a_{21}({\bf mb}^-_{\mathfrak{p}}, i) 
- a_{21}({\bf mb}^+_{\mathfrak{p}}, i) 
+ a_{21}({\bf mb}^-_{\mathfrak{p}}, t_{\mathfrak{p}}) 
a_{11}(t_{\mathfrak{p}}, i)$ \\ 
$+ \mu a_{21}({\bf mb}^-_{\mathfrak{p}}, m^-_{\mathfrak{p}}) 
a_{11}(m^-_{\mathfrak{p}}, i) 
+ a_{21}({\bf mb}^+_{\mathfrak{p}}, m^+_{\mathfrak{p}}) 
a_{11}(m^+_{\mathfrak{p}}, i) 
- a_{21}({\bf mb}^-_{\mathfrak{p}}, m^-_{\mathfrak{p}}) 
a_{11}(m^-_{\mathfrak{p}}, t_{\mathfrak{p}}) 
a_{11}(t_{\mathfrak{p}}, i)$ \\ 
$- a_{11}(b^{-+}_{\mathfrak{p}}, m^-_{\mathfrak{p}}) 
a_{21}({\bf tm}_{\mathfrak{p}}, i) 
- a_{12}(b^{-+}_{\mathfrak{p}}, {\bf tm}_{\mathfrak{p}}) 
a_{11}(m^+_{\mathfrak{p}}, i) 
+ \mu^{-1} a_{11}(b^{-+}_{\mathfrak{p}}, t_{\mathfrak{p}}) 
a_{12}(t_{\mathfrak{p}}, {\bf tm}_{\mathfrak{p}}) 
a_{11}(m^+_{\mathfrak{p}}, i)$, \\  
$\partial a_{13}^{-+}(i, {\mathfrak{p}})$ 
$= a_{12}(i, {\bf tb}^-_{\mathfrak{p}}) 
+ \mu a_{12}(i, {\bf tb}^+_{\mathfrak{p}}) 
- a_{11}(i, m^+_{\mathfrak{p}}) 
a_{12}(m^+_{\mathfrak{p}}, {\bf tb}^+_{\mathfrak{p}})$ \\ 
$- a_{12}(i, {\bf mb}^-_{\mathfrak{p}}) 
- \mu a_{12}(i, {\bf mb}^+_{\mathfrak{p}}) 
+ a_{11}(i, t_{\mathfrak{p}}) 
a_{12}(t_{\mathfrak{p}}, {\bf mb}^-_{\mathfrak{p}})$ \\ 
$+ \mu^{-1} a_{11}(i, m^-_{\mathfrak{p}}) 
a_{12}(m^-_{\mathfrak{p}}, {\bf mb}^-_{\mathfrak{p}}) 
+ a_{11}(i, m^+_{\mathfrak{p}}) 
a_{12}(m^+_{\mathfrak{p}}, {\bf mb}^+_{\mathfrak{p}}) 
- \mu^{-1} a_{11}(i, t_{\mathfrak{p}}) 
a_{11}(t_{\mathfrak{p}}, m^-_{\mathfrak{p}}) 
a_{12}(m^-_{\mathfrak{p}}, {\bf mb}^-_{\mathfrak{p}})$ \\ 
$- a_{12}(i, {\bf tm}_{\mathfrak{p}}) 
a_{11}(m^-_{\mathfrak{p}}, b^{-+}_{\mathfrak{p}}) 
- a_{11}(i, m^+_{\mathfrak{p}}) 
a_{21}({\bf tm}_{\mathfrak{p}}, b^{-+}_{\mathfrak{p}}) 
+ a_{11}(i, m^+_{\mathfrak{p}}) 
a_{21}({\bf tm}_{\mathfrak{p}}, t_{\mathfrak{p}}) 
a_{11}(t_{\mathfrak{p}}, b^{-+}_{\mathfrak{p}})$, \\  
$\partial a_{32}^{-+}({\mathfrak{p}}, {\bf x})$ 
$= \mu a_{22}({\bf tb}^-_{\mathfrak{p}}, {\bf x}) 
+ a_{22}({\bf tb}^+_{\mathfrak{p}}, {\bf x}) 
- a_{21}({\bf tb}^+_{\mathfrak{p}}, m^+_{\mathfrak{p}}) 
a_{12}(m^+_{\mathfrak{p}}, {\bf x})$ \\ 
$- \mu a_{22}({\bf mb}^-_{\mathfrak{p}}, {\bf x}) 
- a_{22}({\bf mb}^+_{\mathfrak{p}}, {\bf x}) 
+ a_{21}({\bf mb}^-_{\mathfrak{p}}, t_{\mathfrak{p}}) 
a_{12}(t_{\mathfrak{p}}, {\bf x})$ \\ 
$+ \mu a_{21}({\bf mb}^-_{\mathfrak{p}}, m^-_{\mathfrak{p}}) 
a_{12}(m^-_{\mathfrak{p}}, {\bf x}) 
+ a_{21}({\bf mb}^+_{\mathfrak{p}}, m^+_{\mathfrak{p}}) 
a_{12}(m^+_{\mathfrak{p}}, {\bf x}) 
- a_{21}({\bf mb}^-_{\mathfrak{p}}, m^-_{\mathfrak{p}}) 
a_{11}(m^-_{\mathfrak{p}}, t_{\mathfrak{p}}) 
a_{12}(t_{\mathfrak{p}}, {\bf x})$ \\  
$- a_{11}(b^{-+}_{\mathfrak{p}}, m^-_{\mathfrak{p}}) 
a_{22}({\bf tm}_{\mathfrak{p}}, {\bf x}) 
- a_{12}(b^{-+}_{\mathfrak{p}}, {\bf tm}_{\mathfrak{p}}) 
a_{12}(m^+_{\mathfrak{p}}, {\bf x}) 
+ \mu^{-1} a_{11}(b^{-+}_{\mathfrak{p}}, t_{\mathfrak{p}}) 
a_{12}(t_{\mathfrak{p}}, {\bf tm}_{\mathfrak{p}}) 
a_{12}(m^+_{\mathfrak{p}}, {\bf x})$ \\ 
$- a_{31}^{-+}({\mathfrak{p}}, u^-_{\bf x}) 
- \mu a_{31}^{-+}({\mathfrak{p}}, u^+_{\bf x}) 
+ a_{31}^{-+}({\mathfrak{p}}, o_{\bf x}) 
a_{11}(o_{\bf x}, u^-_{\bf x})$, \\  
$\partial a_{23}^{-+}({\bf x}, {\mathfrak{p}})$ 
$= \mu a_{13}^{-+}(u^-_{\bf x}, {\mathfrak{p}}) 
+ a_{13}^{-+}(u^+_{\bf x}, {\mathfrak{p}}) 
- a_{11}(u^-_{\bf x}, o_{\bf x}) 
a_{13}^{-+}(o_{\bf x}, {\mathfrak{p}})$ \\ 
$- a_{22}({\bf x}, {\bf tb}^-_{\mathfrak{p}}) 
- \mu a_{22}({\bf x}, {\bf tb}^+_{\mathfrak{p}}) 
+ a_{21}({\bf x}, m^+_{\mathfrak{p}}) 
a_{12}(m^+_{\mathfrak{p}}, {\bf tb}^+_{\mathfrak{p}})$ \\ 
$+ a_{22}({\bf x}, {\bf mb}^-_{\mathfrak{p}}) 
+ \mu a_{22}({\bf x}, {\bf mb}^+_{\mathfrak{p}}) 
- a_{21}({\bf x}, t_{\mathfrak{p}}) 
a_{12}(t_{\mathfrak{p}}, {\bf mb}^-_{\mathfrak{p}})$ \\ 
$- \mu^{-1} a_{21}({\bf x}, m^-_{\mathfrak{p}}) 
a_{12}(m^-_{\mathfrak{p}}, {\bf mb}^-_{\mathfrak{p}}) 
- a_{21}({\bf x}, m^+_{\mathfrak{p}}) 
a_{12}(m^+_{\mathfrak{p}}, {\bf mb}^+_{\mathfrak{p}}) 
+ \mu^{-1} a_{21}({\bf x}, t_{\mathfrak{p}}) 
a_{11}(t_{\mathfrak{p}}, m^-_{\mathfrak{p}}) 
a_{12}(m^-_{\mathfrak{p}}, {\bf mb}^-_{\mathfrak{p}})$ \\ 
$+ a_{22}({\bf x}, {\bf tm}_{\mathfrak{p}}) 
a_{11}(m^-_{\mathfrak{p}}, b^{-+}_{\mathfrak{p}}) 
+ a_{21}({\bf x}, m^+_{\mathfrak{p}}) 
a_{21}({\bf tm}_{\mathfrak{p}}, b^{-+}_{\mathfrak{p}}) 
- a_{21}({\bf x}, m^+_{\mathfrak{p}}) 
a_{21}({\bf tm}_{\mathfrak{p}}, t_{\mathfrak{p}}) 
a_{11}(t_{\mathfrak{p}}, b^{-+}_{\mathfrak{p}})$, \\  
$\partial a_{3}^{-+}({\mathfrak{p}})$ 
$= a_{2}({\bf tb}^+_{\mathfrak{p}})$ 
$- a_{2}({\bf tb}^-_{\mathfrak{p}})$ 
$- a_{2}({\bf mb}^+_{\mathfrak{p}})$ 
$+ \mu a_{22}({\bf tb}^-_{\mathfrak{p}}, {\bf mb}^+_{\mathfrak{p}})$ \\ 
$- a_{11}(b^{--}_{\mathfrak{p}}, m^-_{\mathfrak{p}}) 
\mu^{-1} (a_{22}({\bf tm}_{\mathfrak{p}}, {\bf tm}_{\mathfrak{p}})$ 
$- a_{2}({\bf tm}_{\mathfrak{p}})) 
a_{11}(m^-_{\mathfrak{p}}, b^{--}_{\mathfrak{p}})$ \\ 
$- a_{11}(b^{--}_{\mathfrak{p}}, m^-_{\mathfrak{p}}) 
(a_{22}({\bf tm}_{\mathfrak{p}}, {\bf mb}^+_{\mathfrak{p}}) 
- \mu^{-1} a_{22}({\bf tm}_{\mathfrak{p}}, {\bf tb}^-_{\mathfrak{p}}))$ \\ 
$+ ((\mu^{-1} a_{11}(b^{--}_{\mathfrak{p}}, m^-_{\mathfrak{p}}) 
a_{21}({\bf tm}_{\mathfrak{p}}, m^+_{\mathfrak{p}}) 
+ a_{12}(b^{--}_{\mathfrak{p}}, {\bf tm}_{\mathfrak{p}}) 
- \mu^{-1} a_{11}(b^{--}_{\mathfrak{p}}, t_{\mathfrak{p}}) 
a_{12}(t_{\mathfrak{p}}, {\bf tm}_{\mathfrak{p}}))$ \\ 
$(a_{12}(m^+_{\mathfrak{p}}, {\bf tm}_{\mathfrak{p}}) 
a_{11}(m^-_{\mathfrak{p}}, b^{--}_{\mathfrak{p}}) 
+ a_{21}({\bf tm}_{\mathfrak{p}}, m^-_{\mathfrak{p}}) 
a_{11}(m^-_{\mathfrak{p}}, b^{--}_{\mathfrak{p}}) 
- a_{21}({\bf tm}_{\mathfrak{p}}, t_{\mathfrak{p}}) 
a_{11}(t_{\mathfrak{p}}, m^-_{\mathfrak{p}}) 
a_{11}(m^-_{\mathfrak{p}}, b^{--}_{\mathfrak{p}})))$ \\ 
$+ \mu^{-1} 
((a_{11}(b^{--}_{\mathfrak{p}}, m^-_{\mathfrak{p}}) 
a_{12}(m^-_{\mathfrak{p}}, {\bf tm}_{\mathfrak{p}}) 
- \mu^{-1} a_{11}(b^{--}_{\mathfrak{p}}, m^-_{\mathfrak{p}}) 
a_{11}(m^-_{\mathfrak{p}}, t_{\mathfrak{p}}) 
a_{12}(t_{\mathfrak{p}}, {\bf tm}_{\mathfrak{p}})$ \\ 
$- \mu a_{12}(b^{--}_{\mathfrak{p}}, {\bf tm}_{\mathfrak{p}}) 
+ a_{11}(b^{--}_{\mathfrak{p}}, t_{\mathfrak{p}}) 
a_{12}(t_{\mathfrak{p}}, {\bf tm}_{\mathfrak{p}}))$ \\ 
$(a_{12}(m^+_{\mathfrak{p}}, {\bf mb}^+_{\mathfrak{p}}) 
+ a_{12}(m^+_{\mathfrak{p}}, {\bf tb}^-_{\mathfrak{p}})
+ a_{21}({\bf tm}_{\mathfrak{p}}, b^{--}_{\mathfrak{p}}) 
- a_{21}({\bf tm}_{\mathfrak{p}}, t_{\mathfrak{p}}) 
a_{11}(t_{\mathfrak{p}}, b^{--}_{\mathfrak{p}})))$ \\ 
$+ (\mu^{-1} a_{11}(b^{--}_{\mathfrak{p}}, m^-_{\mathfrak{p}}) 
a_{21}({\bf tm}_{\mathfrak{p}}, m^+_{\mathfrak{p}}) 
- a_{21}({\bf tb}^-_{\mathfrak{p}}, m^+_{\mathfrak{p}}))$ 
$a_{12}(m^+_{\mathfrak{p}}, {\bf mb}^+_{\mathfrak{p}})$ \\ 
$- \mu a_{31}^{-+}({\mathfrak{p}}, b^{++}_{\mathfrak{p}})$ 
$+ \mu a_{13}^{-+}(b^{--}_{\mathfrak{p}}, {\mathfrak{p}})$ 
$- a_{11}(b^{--}_{\mathfrak{p}}, m^-_{\mathfrak{p}}) 
a_{13}^{-+}(m^-_{\mathfrak{p}}, {\mathfrak{p}})$ \\ 
$- (a_{11}(b^{--}_{\mathfrak{p}}, t_{\mathfrak{p}}) 
- \mu^{-1} a_{11}(b^{--}_{\mathfrak{p}}, m^-_{\mathfrak{p}}) 
a_{11}(m^-_{\mathfrak{p}}, t_{\mathfrak{p}})) 
a_{13}^{-+}(t_{\mathfrak{p}}, {\mathfrak{p}})$ \\ 
$- (\mu a_{11}(b^{--}_{\mathfrak{p}}, m^+_{\mathfrak{p}}) 
- a_{11}(b^{--}_{\mathfrak{p}}, m^-_{\mathfrak{p}}) 
a_{11}(m^-_{\mathfrak{p}}, m^+_{\mathfrak{p}}) 
- a_{11}(b^{--}_{\mathfrak{p}}, t_{\mathfrak{p}}) 
a_{11}(t_{\mathfrak{p}}, m^+_{\mathfrak{p}})$ \\ 
$+ \mu^{-1} a_{11}(b^{--}_{\mathfrak{p}}, m^-_{\mathfrak{p}}) 
a_{11}(m^-_{\mathfrak{p}}, t_{\mathfrak{p}}) 
a_{11}(t_{\mathfrak{p}}, m^+_{\mathfrak{p}}))$ 
$a_{13}^{-+}(m^+_{\mathfrak{p}}, {\mathfrak{p}})$ \\ 
$+ a_{31}^{-+}({\mathfrak{p}}, m^+_{\mathfrak{p}}) 
a_{11}(m^+_{\mathfrak{p}}, b^{++}_{\mathfrak{p}})$ 
$- \mu a_{22}({\bf mb}^-_{\mathfrak{p}}, {\bf tb}^+_{\mathfrak{p}})$ \\ 
$+ (a_{11}(b^{--}_{\mathfrak{p}}, m^-_{\mathfrak{p}}) 
a_{11}(m^-_{\mathfrak{p}}, m^+_{\mathfrak{p}}) 
- \mu^{-1} a_{11}(b^{--}_{\mathfrak{p}}, m^-_{\mathfrak{p}}) 
a_{11}(m^-_{\mathfrak{p}}, t_{\mathfrak{p}}) 
a_{11}(t_{\mathfrak{p}}, m^+_{\mathfrak{p}})$ 
$- \mu a_{11}(b^{--}_{\mathfrak{p}}, m^+_{\mathfrak{p}})$ \\ 
$+ a_{11}(b^{--}_{\mathfrak{p}}, t_{\mathfrak{p}}) 
a_{11}(t_{\mathfrak{p}}, m^+_{\mathfrak{p}}))$ 
$a_{22}({\bf tm}_{\mathfrak{p}}, {\bf mb}^-_{\mathfrak{p}})$ \\ 
$+ a_{22}({\bf mb}^-_{\mathfrak{p}}, {\bf tm}_{\mathfrak{p}}) 
a_{11}(m^+_{\mathfrak{p}}, b^{++}_{\mathfrak{p}})$ \\ 
$- (a_{11}(b^{--}_{\mathfrak{p}}, m^-_{\mathfrak{p}}) 
a_{11}(m^-_{\mathfrak{p}}, m^+_{\mathfrak{p}}) 
- \mu^{-1} a_{11}(b^{--}_{\mathfrak{p}}, m^-_{\mathfrak{p}}) 
a_{11}(m^-_{\mathfrak{p}}, t_{\mathfrak{p}}) 
a_{11}(t_{\mathfrak{p}}, m^+_{\mathfrak{p}})$ 
$- \mu a_{11}(b^{--}_{\mathfrak{p}}, m^+_{\mathfrak{p}})$ \\ 
$+ a_{11}(b^{--}_{\mathfrak{p}}, t_{\mathfrak{p}}) 
a_{11}(t_{\mathfrak{p}}, m^+_{\mathfrak{p}}))$ 
$(- \mu^{-1} a_{12}(m^+_{\mathfrak{p}}, {\bf tm}_{\mathfrak{p}}) 
a_{12}(m^-_{\mathfrak{p}}, {\bf mb}^-_{\mathfrak{p}})$ 
$+ a_{21}({\bf tm}_{\mathfrak{p}}, t_{\mathfrak{p}}) 
a_{12}(t_{\mathfrak{p}}, {\bf mb}^-_{\mathfrak{p}}))$ \\ 
$- \mu^{-1} (a_{11}(b^{--}_{\mathfrak{p}}, m^-_{\mathfrak{p}}) 
a_{12}(m^-_{\mathfrak{p}}, {\bf tm}_{\mathfrak{p}}) 
- \mu a_{12}(b^{--}_{\mathfrak{p}}, {\bf tm}_{\mathfrak{p}}) 
+ a_{11}(b^{--}_{\mathfrak{p}}, t_{\mathfrak{p}}) 
a_{12}(t_{\mathfrak{p}}, {\bf tm}_{\mathfrak{p}})$ \\ 
$- \mu^{-1} a_{11}(b^{--}_{\mathfrak{p}}, m^-_{\mathfrak{p}}) 
a_{11}(m^-_{\mathfrak{p}}, t_{\mathfrak{p}}) 
a_{12}(t_{\mathfrak{p}}, {\bf tm}_{\mathfrak{p}}))$ 
$a_{12}(m^-_{\mathfrak{p}}, {\bf mb}^-_{\mathfrak{p}})$ \\ 
$- \mu^{-1} a_{21}({\bf mb}^-_{\mathfrak{p}}, t_{\mathfrak{p}}) 
a_{12}(t_{\mathfrak{p}}, {\bf tm}_{\mathfrak{p}}) 
a_{11}(m^+_{\mathfrak{p}}, b^{++}_{\mathfrak{p}})$ 
$+ a_{21}({\bf mb}^-_{\mathfrak{p}}, t_{\mathfrak{p}}) 
a_{12}(t_{\mathfrak{p}}, {\bf tb}^+_{\mathfrak{p}})$ \\ 
$+ a_{21}({\bf mb}^-_{\mathfrak{p}}, m^-_{\mathfrak{p}}) 
(a_{21}({\bf tm}_{\mathfrak{p}}, m^+_{\mathfrak{p}}) 
a_{11}(m^+_{\mathfrak{p}}, b^{++}_{\mathfrak{p}}) 
- \mu a_{21}({\bf tm}_{\mathfrak{p}}, b^{++}_{\mathfrak{p}}))$ 
$+ a_{2}({\bf mb}^-_{\mathfrak{p}})$, \\  
$\partial a_{33}^{-+}({\mathfrak{p}}, {\mathfrak{q}})$ 
$= \mu a_{23}^{-+}({\bf tb}^-_{\mathfrak{p}}, {\mathfrak{q}}) 
+ a_{23}^{-+}({\bf tb}^+_{\mathfrak{p}}, {\mathfrak{q}}) 
- a_{21}({\bf tb}^+_{\mathfrak{p}}, m^+_{\mathfrak{p}}) 
a_{13}^{-+}(m^+_{\mathfrak{p}}, {\mathfrak{q}})$ \\ 
$- \mu a_{23}^{-+}({\bf mb}^-_{\mathfrak{p}}, {\mathfrak{q}}) 
- a_{23}^{-+}({\bf mb}^+_{\mathfrak{p}}, {\mathfrak{q}}) 
+ a_{21}({\bf mb}^-_{\mathfrak{p}}, t_{\mathfrak{p}}) 
a_{13}^{-+}(t_{\mathfrak{p}}, {\mathfrak{q}})$ \\ 
$+ \mu a_{21}({\bf mb}^-_{\mathfrak{p}}, m^-_{\mathfrak{p}}) 
a_{13}^{-+}(m^-_{\mathfrak{p}}, {\mathfrak{q}}) 
+ a_{21}({\bf mb}^+_{\mathfrak{p}}, m^+_{\mathfrak{p}}) 
a_{13}^{-+}(m^+_{\mathfrak{p}}, {\mathfrak{q}})$ 
$- a_{21}({\bf mb}^-_{\mathfrak{p}}, m^-_{\mathfrak{p}}) 
a_{11}(m^-_{\mathfrak{p}}, t_{\mathfrak{p}}) 
a_{13}^{-+}(t_{\mathfrak{p}}, {\mathfrak{q}})$ \\ 
$- a_{11}(b^{-+}_{\mathfrak{p}}, m^-_{\mathfrak{p}}) 
a_{23}^{-+}({\bf tm}_{\mathfrak{p}}, {\mathfrak{q}}) 
- a_{12}(b^{-+}_{\mathfrak{p}}, {\bf tm}_{\mathfrak{p}}) 
a_{13}^{-+}(m^+_{\mathfrak{p}}, {\mathfrak{q}})$ 
$+ \mu^{-1} a_{11}(b^{-+}_{\mathfrak{p}}, t_{\mathfrak{p}}) 
a_{12}(t_{\mathfrak{p}}, {\bf tm}_{\mathfrak{p}}) 
a_{13}^{-+}(m^+_{\mathfrak{p}}, {\mathfrak{q}})$ \\ 
$+ a_{32}^{-+}({\mathfrak{p}}, {\bf tb}^-_{\mathfrak{q}}) 
+ \mu a_{32}^{-+}({\mathfrak{p}}, {\bf tb}^+_{\mathfrak{q}}) 
- a_{31}^{-+}({\mathfrak{p}}, m^+_{\mathfrak{q}}) 
a_{12}(m^+_{\mathfrak{q}}, {\bf tb}^+_{\mathfrak{q}})$ \\ 
$- a_{32}^{-+}({\mathfrak{p}}, {\bf mb}^-_{\mathfrak{q}}) 
- \mu a_{32}^{-+}({\mathfrak{p}}^{-+}, {\bf mb}^+_{\mathfrak{q}}) 
+ a_{31}^{-+}({\mathfrak{p}}, t_{\mathfrak{q}}) 
a_{12}(t_{\mathfrak{q}}, {\bf mb}^-_{\mathfrak{q}})$ \\ 
$+ \mu^{-1} a_{31}^{-+}({\mathfrak{p}}, m^-_{\mathfrak{q}}) 
a_{12}(m^-_{\mathfrak{q}}, {\bf mb}^-_{\mathfrak{q}}) 
+ a_{31}^{-+}({\mathfrak{p}}, m^+_{\mathfrak{q}}) 
a_{12}(m^+_{\mathfrak{q}}, {\bf mb}^+_{\mathfrak{q}})$ 
$- \mu^{-1} a_{31}^{-+}({\mathfrak{p}}, t_{\mathfrak{q}}) 
a_{11}(t_{\mathfrak{q}}, m^-_{\mathfrak{q}}) 
a_{12}(m^-_{\mathfrak{q}}, {\bf mb}^-_{\mathfrak{q}})$ \\ 
$- a_{32}^{-+}({\mathfrak{p}}, {\bf tm}_{\mathfrak{q}}) 
a_{11}(m^-_{\mathfrak{q}}, b^{-+}_{\mathfrak{q}}) 
- a_{31}^{-+}({\mathfrak{p}}, m^+_{\mathfrak{q}}) 
a_{21}({\bf tm}_{\mathfrak{q}}, b^{-+}_{\mathfrak{q}})$ 
$+ a_{31}^{-+}({\mathfrak{p}}, m^+_{\mathfrak{q}}) 
a_{21}({\bf tm}_{\mathfrak{q}}, t_{\mathfrak{q}}) 
a_{11}(t_{\mathfrak{q}}, b^{-+}_{\mathfrak{q}})$, \\ 
$\partial a_{3b}^{-+}({\mathfrak{p}}, {\textsl{k}})$ 
$= \mu a_{2b}({\bf tb}^-_{\mathfrak{p}}, {\textsl{k}}) 
+ a_{2b}({\bf tb}^+_{\mathfrak{p}}, {\textsl{k}}) 
- a_{21}({\bf tb}^+_{\mathfrak{p}}, m^+_{\mathfrak{p}}) 
a_{1b}(m^+_{\mathfrak{p}}, {\textsl{k}})$ \\ 
$- \mu a_{2b}({\bf mb}^-_{\mathfrak{p}}, {\textsl{k}}) 
- a_{2b}({\bf mb}^+_{\mathfrak{p}}, {\textsl{k}}) 
+ a_{21}({\bf mb}^-_{\mathfrak{p}}, t_{\mathfrak{p}}) 
a_{1b}(t_{\mathfrak{p}}, {\textsl{k}})$ \\ 
$+ \mu a_{21}({\bf mb}^-_{\mathfrak{p}}, m^-_{\mathfrak{p}}) 
a_{1b}(m^-_{\mathfrak{p}}, {\textsl{k}}) 
+ a_{21}({\bf mb}^+_{\mathfrak{p}}, m^+_{\mathfrak{p}}) 
a_{1b}(m^+_{\mathfrak{p}}, {\textsl{k}}) 
- a_{21}({\bf mb}^-_{\mathfrak{p}}, m^-_{\mathfrak{p}}) 
a_{11}(m^-_{\mathfrak{p}}, t_{\mathfrak{p}}) 
a_{1b}(t_{\mathfrak{p}}, {\textsl{k}})$ \\ 
$- a_{11}(b^{-+}_{\mathfrak{p}}, m^-_{\mathfrak{p}}) 
a_{2b}({\bf tm}_{\mathfrak{p}}, {\textsl{k}}) 
- a_{12}(b^{-+}_{\mathfrak{p}}, {\bf tm}_{\mathfrak{p}}) 
a_{1b}(m^+_{\mathfrak{p}}, {\textsl{k}}) 
+ \mu^{-1} a_{11}(b^{-+}_{\mathfrak{p}}, t_{\mathfrak{p}}) 
a_{12}(t_{\mathfrak{p}}, {\bf tm}_{\mathfrak{p}}) 
a_{1b}(m^+_{\mathfrak{p}}, {\textsl{k}})$ \\ 
$+ a_{32}^{-+}({\mathfrak{p}}, {\bf dc}_{\textsl{k}})$, \\ 
$\partial a_{b3}^{-+}({\textsl{k}}, {\mathfrak{p}})$ 
$= a_{23}^{-+}({\bf dc}_{\textsl{k}}, {\mathfrak{p}})$ 
$+ a_{b2}({\textsl{k}}, {\bf tb}^-_{\mathfrak{p}}) 
+ \mu a_{b2}({\textsl{k}}, {\bf tb}^+_{\mathfrak{p}}) 
- a_{b1}({\textsl{k}}, m^+_{\mathfrak{p}}) 
a_{12}(m^+_{\mathfrak{p}}, {\bf tb}^+_{\mathfrak{p}})$ \\ 
$- a_{b2}({\textsl{k}}, {\bf mb}^-_{\mathfrak{p}}) 
- \mu a_{b2}({\textsl{k}}, {\bf mb}^+_{\mathfrak{p}}) 
+ a_{b1}({\textsl{k}}, t_{\mathfrak{p}}) 
a_{12}(t_{\mathfrak{p}}, {\bf mb}^-_{\mathfrak{p}})$ \\ 
$+ \mu^{-1} a_{b1}({\textsl{k}}, m^-_{\mathfrak{p}}) 
a_{12}(m^-_{\mathfrak{p}}, {\bf mb}^-_{\mathfrak{p}}) 
+ a_{b1}({\textsl{k}}, m^+_{\mathfrak{p}}) 
a_{12}(m^+_{\mathfrak{p}}, {\bf mb}^+_{\mathfrak{p}}) 
- \mu^{-1} a_{b1}({\textsl{k}}, t_{\mathfrak{p}}) 
a_{11}(t_{\mathfrak{p}}, m^-_{\mathfrak{p}}) 
a_{12}(m^-_{\mathfrak{p}}, {\bf mb}^-_{\mathfrak{p}})$ \\ 
$- a_{b2}({\textsl{k}}, {\bf tm}_{\mathfrak{p}}) 
a_{11}(m^-_{\mathfrak{p}}, b^{-+}_{\mathfrak{p}}) 
- a_{b1}({\textsl{k}}, m^+_{\mathfrak{p}}) 
a_{21}({\bf tm}_{\mathfrak{p}}, b^{-+}_{\mathfrak{p}}) 
+ a_{b1}({\textsl{k}}, m^+_{\mathfrak{p}}) 
a_{21}({\bf tm}_{\mathfrak{p}}, t_{\mathfrak{p}}) 
a_{11}(t_{\mathfrak{p}}, b^{-+}_{\mathfrak{p}})$. \\ 
It is straightforward to see that the equation $\partial \circ \partial = 0$ 
holds on generators of $CR^{-+}(D)$.

\subsection{$(CR^{+-}(D), \partial)$} 
We define 
a differential $\partial$ 
on generators of $CR^{+-}(D)$ 
involving triple points as follows. \\ 
$\partial a_{31}^{+-}({\mathfrak{p}}, i)$ 
$= \mu a_{21}({\bf tb}^-_{\mathfrak{p}}, i) 
+ a_{21}({\bf tb}^+_{\mathfrak{p}}, i) 
- a_{21}({\bf tb}^-_{\mathfrak{p}}, m^-_{\mathfrak{p}}) 
a_{11}(m^-_{\mathfrak{p}}, i)$ \\ 
$- \mu a_{21}({\bf mb}^-_{\mathfrak{p}}, i) 
- a_{21}({\bf mb}^+_{\mathfrak{p}}, i) 
+ a_{21}({\bf mb}^+_{\mathfrak{p}}, t_{\mathfrak{p}}) 
a_{11}(t_{\mathfrak{p}}, i)$ \\ 
$- \mu a_{21}({\bf tb}^-_{\mathfrak{p}}, t_{\mathfrak{p}}) 
a_{11}(t_{\mathfrak{p}}, i) 
- a_{21}({\bf tb}^+_{\mathfrak{p}}, t_{\mathfrak{p}}) 
a_{11}(t_{\mathfrak{p}}, i) 
+ a_{21}({\bf tb}^-_{\mathfrak{p}}, t_{\mathfrak{p}}) 
a_{11}(t_{\mathfrak{p}}, m^-_{\mathfrak{p}}) 
a_{11}(m^-_{\mathfrak{p}}, i)$ \\ 
$- a_{11}(b^{+-}_{\mathfrak{p}}, m^+_{\mathfrak{p}}) 
a_{21}({\bf tm}_{\mathfrak{p}}, i) 
- a_{12}(b^{+-}_{\mathfrak{p}}, {\bf tm}_{\mathfrak{p}}) 
a_{11}(m^-_{\mathfrak{p}}, i) 
+ a_{11}(b^{+-}_{\mathfrak{p}}, m^+_{\mathfrak{p}}) 
a_{21}({\bf tm}_{\mathfrak{p}}, t_{\mathfrak{p}}) 
a_{11}(t_{\mathfrak{p}}, i)$, \\  
$\partial a_{13}^{+-}(i, {\mathfrak{p}})$ 
$= a_{12}(i, {\bf tb}^-_{\mathfrak{p}}) 
+ \mu a_{12}(i, {\bf tb}^+_{\mathfrak{p}}) 
- a_{11}(i, m^-_{\mathfrak{p}}) 
a_{12}(m^-_{\mathfrak{p}}, {\bf tb}^-_{\mathfrak{p}})$ \\ 
$- a_{12}(i, {\bf mb}^-_{\mathfrak{p}}) 
- \mu a_{12}(i, {\bf mb}^+_{\mathfrak{p}}) 
+ a_{11}(i, t_{\mathfrak{p}}) 
a_{12}(t_{\mathfrak{p}}, {\bf mb}^+_{\mathfrak{p}})$ \\ 
$- \mu^{-1} a_{11}(i, t_{\mathfrak{p}}) 
a_{12}(t_{\mathfrak{p}}, {\bf tb}^-_{\mathfrak{p}}) 
- a_{11}(i, t_{\mathfrak{p}}) 
a_{12}(t_{\mathfrak{p}}, {\bf tb}^+_{\mathfrak{p}}) 
+ \mu^{-1} a_{11}(i, m^-_{\mathfrak{p}}) 
a_{11}(m^-_{\mathfrak{p}}, t_{\mathfrak{p}}) 
a_{12}(t_{\mathfrak{p}}, {\bf tb}^-_{\mathfrak{p}})$ \\ 
$- a_{12}(i, {\bf tm}_{\mathfrak{p}}) 
a_{11}(m^+_{\mathfrak{p}}, b^{+-}_{\mathfrak{p}}) 
- a_{11}(i, m^-_{\mathfrak{p}}) 
a_{21}({\bf tm}_{\mathfrak{p}}, b^{+-}_{\mathfrak{p}}) 
+ \mu^{-1} a_{11}(i, t_{\mathfrak{p}}) 
a_{12}(t_{\mathfrak{p}}, {\bf tm}_{\mathfrak{p}}) 
a_{11}(m^+_{\mathfrak{p}}, b^{+-}_{\mathfrak{p}})$, \\  
$\partial a_{32}^{+-}({\mathfrak{p}}, {\bf x})$ 
$= \mu a_{22}({\bf tb}^-_{\mathfrak{p}}, {\bf x}) 
+ a_{22}({\bf tb}^+_{\mathfrak{p}}, {\bf x}) 
- a_{21}({\bf tb}^-_{\mathfrak{p}}, m^-_{\mathfrak{p}}) 
a_{12}(m^-_{\mathfrak{p}}, {\bf x})$ \\ 
$- \mu a_{22}({\bf mb}^-_{\mathfrak{p}}, {\bf x}) 
- a_{22}({\bf mb}^+_{\mathfrak{p}}, {\bf x}) 
+ a_{21}({\bf mb}^+_{\mathfrak{p}}, t_{\mathfrak{p}}) 
a_{12}(t_{\mathfrak{p}}, {\bf x})$ \\ 
$- \mu a_{21}({\bf tb}^-_{\mathfrak{p}}, t_{\mathfrak{p}}) 
a_{12}(t_{\mathfrak{p}}, {\bf x}) 
- a_{21}({\bf tb}^+_{\mathfrak{p}}, t_{\mathfrak{p}}) 
a_{12}(t_{\mathfrak{p}}, {\bf x}) 
+ a_{21}({\bf tb}^-_{\mathfrak{p}}, t_{\mathfrak{p}}) 
a_{11}(t_{\mathfrak{p}}, m^-_{\mathfrak{p}}) 
a_{12}(m^-_{\mathfrak{p}}, {\bf x})$ \\ 
$- a_{11}(b^{+-}_{\mathfrak{p}}, m^+_{\mathfrak{p}}) 
a_{22}({\bf tm}_{\mathfrak{p}}, {\bf x}) 
- a_{12}(b^{+-}_{\mathfrak{p}}, {\bf tm}_{\mathfrak{p}}) 
a_{12}(m^-_{\mathfrak{p}}, {\bf x}) 
+ a_{11}(b^{+-}_{\mathfrak{p}}, m^+_{\mathfrak{p}}) 
a_{21}({\bf tm}_{\mathfrak{p}}, t_{\mathfrak{p}}) 
a_{12}(t_{\mathfrak{p}}, {\bf x})$ \\ 
$- a_{31}^{+-}({\mathfrak{p}}, u^-_{\bf x}) 
- \mu a_{31}^{+-}({\mathfrak{p}}, u^+_{\bf x}) 
+ a_{31}^{+-}({\mathfrak{p}}, o_{\bf x}) 
a_{11}(o_{\bf x}, u^-_{\bf x})$, \\  
$\partial a_{23}^{+-}({\bf x}, {\mathfrak{p}})$ 
$= \mu a_{13}^{+-}(u^-_{\bf x}, {\mathfrak{p}}) 
+ a_{13}^{+-}(u^+_{\bf x}, {\mathfrak{p}}) 
- a_{11}(u^-_{\bf x}, o_{\bf x}) 
a_{13}^{+-}(o_{\bf x}, {\mathfrak{p}})$ \\ 
$- a_{22}({\bf x}, {\bf tb}^-_{\mathfrak{p}}) 
- \mu a_{22}({\bf x}, {\bf tb}^+_{\mathfrak{p}}) 
+ a_{21}({\bf x}, m^-_{\mathfrak{p}}) 
a_{12}(m^-_{\mathfrak{p}}, {\bf tb}^-_{\mathfrak{p}})$ \\ 
$+ a_{22}({\bf x}, {\bf mb}^-_{\mathfrak{p}}) 
+ \mu a_{22}({\bf x}, {\bf mb}^+_{\mathfrak{p}}) 
- a_{21}({\bf x}, t_{\mathfrak{p}}) 
a_{12}(t_{\mathfrak{p}}, {\bf mb}^+_{\mathfrak{p}})$ \\ 
$+ \mu^{-1} a_{21}({\bf x}, t_{\mathfrak{p}}) 
a_{12}(t_{\mathfrak{p}}, {\bf tb}^-_{\mathfrak{p}}) 
+ a_{21}({\bf x}, t_{\mathfrak{p}}) 
a_{12}(t_{\mathfrak{p}}, {\bf tb}^+_{\mathfrak{p}}) 
- \mu^{-1} a_{21}({\bf x}, m^-_{\mathfrak{p}}) 
a_{11}(m^-_{\mathfrak{p}}, t_{\mathfrak{p}}) 
a_{12}(t_{\mathfrak{p}}, {\bf tb}^-_{\mathfrak{p}})$ \\ 
$+ a_{22}({\bf x}, {\bf tm}_{\mathfrak{p}}) 
a_{11}(m^+_{\mathfrak{p}}, b^{+-}_{\mathfrak{p}}) 
+ a_{21}({\bf x}, m^-_{\mathfrak{p}}) 
a_{21}({\bf tm}_{\mathfrak{p}}, b^{+-}_{\mathfrak{p}}) 
- \mu^{-1} a_{21}({\bf x}, t_{\mathfrak{p}}) 
a_{12}(t_{\mathfrak{p}}, {\bf tm}_{\mathfrak{p}}) 
a_{11}(m^+_{\mathfrak{p}}, b^{+-}_{\mathfrak{p}})$, \\ 
$\partial a_{3}^{+-}({\mathfrak{p}})$ 
$= a_{2}({\bf tb}^+_{\mathfrak{p}})$ 
$- a_{2}({\bf tb}^-_{\mathfrak{p}})$ 
$- a_{2}({\bf mb}^+_{\mathfrak{p}})$ 
$+ \mu a_{22}({\bf tb}^-_{\mathfrak{p}}, {\bf mb}^+_{\mathfrak{p}})$ 
$+ a_{2}({\bf mb}^-_{\mathfrak{p}})$ \\ 
$- \mu^{-1} a_{11}(b^{--}_{\mathfrak{p}}, m^-_{\mathfrak{p}}) 
a_{2}({\bf tm}_{\mathfrak{p}}) 
a_{11}(m^-_{\mathfrak{p}}, b^{--}_{\mathfrak{p}})$ 
$- \mu a_{22}({\bf mb}^-_{\mathfrak{p}}, {\bf tb}^+_{\mathfrak{p}})$ 
$- a_{11}(b^{--}_{\mathfrak{p}}, m^-_{\mathfrak{p}}) 
a_{22}({\bf tm}_{\mathfrak{p}}, {\bf mb}^+_{\mathfrak{p}})$ \\ 
$- a_{22}({\bf tb}^-_{\mathfrak{p}}, {\bf tm}_{\mathfrak{p}}) 
(\mu a_{11}(m^-_{\mathfrak{p}}, b^{++}_{\mathfrak{p}}) 
- a_{11}(m^-_{\mathfrak{p}}, t_{\mathfrak{p}}) 
a_{11}(t_{\mathfrak{p}}, b^{++}_{\mathfrak{p}}))$ 
$+ a_{21}({\bf mb}^-_{\mathfrak{p}}, t_{\mathfrak{p}}) 
a_{12}(t_{\mathfrak{p}}, {\bf tb}^+_{\mathfrak{p}})$ \\ 
$- (a_{21}({\bf tb}^-_{\mathfrak{p}}, m^+_{\mathfrak{p}}) 
+ a_{12}(b^{--}_{\mathfrak{p}}, {\bf tm}_{\mathfrak{p}}) 
- \mu^{-1} a_{11}(b^{--}_{\mathfrak{p}}, t_{\mathfrak{p}}) 
a_{12}(t_{\mathfrak{p}}, {\bf tm}_{\mathfrak{p}}))$ 
$(a_{12}(m^+_{\mathfrak{p}}, {\bf mb}^+_{\mathfrak{p}}) 
+ \mu a_{21}({\bf tm}_{\mathfrak{p}}, b^{++}_{\mathfrak{p}}))$ \\ 
$- \mu^{-1} a_{11}(b^{--}_{\mathfrak{p}}, m^-_{\mathfrak{p}}) 
(a_{21}({\bf tm}_{\mathfrak{p}}, m^+_{\mathfrak{p}}) 
+ a_{12}(m^-_{\mathfrak{p}}, {\bf tm}_{\mathfrak{p}}) 
- \mu^{-1} a_{11}(m^-_{\mathfrak{p}}, t_{\mathfrak{p}}) 
a_{12}(t_{\mathfrak{p}}, {\bf tm}_{\mathfrak{p}}))$ \\ 
$(a_{12}(m^+_{\mathfrak{p}}, {\bf tb}^-_{\mathfrak{p}}) 
+ a_{21}({\bf tm}_{\mathfrak{p}}, b^{--}_{\mathfrak{p}}) 
- a_{21}({\bf tm}_{\mathfrak{p}}, t_{\mathfrak{p}}) 
a_{11}(t_{\mathfrak{p}}, b^{--}_{\mathfrak{p}}))$ \\ 
$- \mu a_{31}^{+-}({\mathfrak{p}}, b^{++}_{\mathfrak{p}})$ 
$+ \mu a_{13}^{+-}(b^{--}_{\mathfrak{p}}, {\mathfrak{p}})$ 
$- a_{11}(b^{--}_{\mathfrak{p}}, m^-_{\mathfrak{p}}) 
a_{13}^{+-}(m^-_{\mathfrak{p}}, {\mathfrak{p}})$ 
$+ a_{31}^{+-}({\mathfrak{p}}, t_{\mathfrak{p}}) 
a_{11}(t_{\mathfrak{p}}, b^{++}_{\mathfrak{p}})$, \\ 
$\partial a_{33}^{+-}({\mathfrak{p}}, {\mathfrak{q}})$ 
$= \mu a_{23}^{+-}({\bf tb}^-_{\mathfrak{p}}, {\mathfrak{q}}) 
+ a_{23}^{+-}({\bf tb}^+_{\mathfrak{p}}, {\mathfrak{q}}) 
- a_{21}({\bf tb}^-_{\mathfrak{p}}, m^-_{\mathfrak{p}}) 
a_{13}^{+-}(m^-_{\mathfrak{p}}, {\mathfrak{q}})$ \\ 
$- \mu a_{23}^{+-}({\bf mb}^-_{\mathfrak{p}}, {\mathfrak{q}}) 
- a_{23}^{+-}({\bf mb}^+_{\mathfrak{p}}, {\mathfrak{q}}) 
+ a_{21}({\bf mb}^+_{\mathfrak{p}}, t_{\mathfrak{p}}) 
a_{13}^{+-}(t_{\mathfrak{p}}, {\mathfrak{q}})$ \\ 
$- \mu a_{21}({\bf tb}^-_{\mathfrak{p}}, t_{\mathfrak{p}}) 
a_{13}^{+-}(t_{\mathfrak{p}}, {\mathfrak{q}}) 
- a_{21}({\bf tb}^+_{\mathfrak{p}}, t_{\mathfrak{p}}) 
a_{13}^{+-}(t_{\mathfrak{p}}, {\mathfrak{q}})$ 
$+ a_{21}({\bf tb}^-_{\mathfrak{p}}, t_{\mathfrak{p}}) 
a_{11}(t_{\mathfrak{p}}, m^-_{\mathfrak{p}}) 
a_{13}^{+-}(m^-_{\mathfrak{p}}, {\mathfrak{q}})$ \\ 
$- a_{11}(b^{+-}_{\mathfrak{p}}, m^+_{\mathfrak{p}}) 
a_{23}^{+-}({\bf tm}_{\mathfrak{p}}, {\mathfrak{q}}) 
- a_{12}(b^{+-}_{\mathfrak{p}}, {\bf tm}_{\mathfrak{p}}) 
a_{13}^{+-}(m^-_{\mathfrak{p}}, {\mathfrak{q}})$ 
$+ a_{11}(b^{+-}_{\mathfrak{p}}, m^+_{\mathfrak{p}}) 
a_{21}({\bf tm}_{\mathfrak{p}}, t_{\mathfrak{p}}) 
a_{13}^{+-}(t_{\mathfrak{p}}, {\mathfrak{q}})$ \\ 
$+ a_{32}^{+-}({\mathfrak{p}}, {\bf tb}^-_{\mathfrak{q}}) 
+ \mu a_{32}^{+-}({\mathfrak{p}}, {\bf tb}^+_{\mathfrak{q}}) 
- a_{31}^{+-}({\mathfrak{p}}, m^-_{\mathfrak{q}}) 
a_{12}(m^-_{\mathfrak{q}}, {\bf tb}^-_{\mathfrak{q}})$ \\ 
$- a_{32}^{+-}({\mathfrak{p}}, {\bf mb}^-_{\mathfrak{q}}) 
- \mu a_{32}^{+-}({\mathfrak{p}}, {\bf mb}^+_{\mathfrak{q}}) 
+ a_{31}^{+-}({\mathfrak{p}}, t_{\mathfrak{q}}) 
a_{12}(t_{\mathfrak{q}}, {\bf mb}^+_{\mathfrak{q}})$ \\ 
$- \mu^{-1} a_{31}^{+-}({\mathfrak{p}}, t_{\mathfrak{q}}) 
a_{12}(t_{\mathfrak{q}}, {\bf tb}^-_{\mathfrak{q}}) 
- a_{31}^{+-}({\mathfrak{p}}, t_{\mathfrak{q}}) 
a_{12}(t_{\mathfrak{q}}, {\bf tb}^+_{\mathfrak{q}})$ 
$+ \mu^{-1} a_{31}^{+-}({\mathfrak{p}}, m^-_{\mathfrak{q}}) 
a_{11}(m^-_{\mathfrak{q}}, t_{\mathfrak{q}}) 
a_{12}(t_{\mathfrak{q}}, {\bf tb}^-_{\mathfrak{q}})$ \\ 
$- a_{32}^{+-}({\mathfrak{p}}, {\bf tm}_{\mathfrak{q}}) 
a_{11}(m^+_{\mathfrak{q}}, b^{+-}_{\mathfrak{q}}) 
- a_{31}^{+-}({\mathfrak{p}}, m^-_{\mathfrak{q}}) 
a_{21}({\bf tm}_{\mathfrak{q}}, b^{+-}_{\mathfrak{q}})$ 
$+ \mu^{-1} a_{31}^{+-}({\mathfrak{p}}, t_{\mathfrak{q}}) 
a_{12}(t_{\mathfrak{q}}, {\bf tm}_{\mathfrak{q}}) 
a_{11}(m^+_{\mathfrak{q}}, b^{+-}_{\mathfrak{q}})$, \\ 
$\partial a_{3b}^{+-}({\mathfrak{p}}, {\textsl{k}})$ 
$= \mu a_{2b}({\bf tb}^-_{\mathfrak{p}}, {\textsl{k}}) 
+ a_{2b}({\bf tb}^+_{\mathfrak{p}}, {\textsl{k}}) 
- a_{21}({\bf tb}^-_{\mathfrak{p}}, m^-_{\mathfrak{p}}) 
a_{1b}(m^-_{\mathfrak{p}}, {\textsl{k}})$ \\ 
$- \mu a_{2b}({\bf mb}^-_{\mathfrak{p}}, {\textsl{k}}) 
- a_{2b}({\bf mb}^+_{\mathfrak{p}}, {\textsl{k}}) 
+ a_{21}({\bf mb}^+_{\mathfrak{p}}, t_{\mathfrak{p}}) 
a_{1b}(t_{\mathfrak{p}}, {\textsl{k}})$ \\ 
$- \mu a_{21}({\bf tb}^-_{\mathfrak{p}}, t_{\mathfrak{p}}) 
a_{1b}(t_{\mathfrak{p}}, {\textsl{k}}) 
- a_{21}({\bf tb}^+_{\mathfrak{p}}, t_{\mathfrak{p}}) 
a_{1b}(t_{\mathfrak{p}}, {\textsl{k}}) 
+ a_{21}({\bf tb}^-_{\mathfrak{p}}, t_{\mathfrak{p}}) 
a_{11}(t_{\mathfrak{p}}, m^-_{\mathfrak{p}}) 
a_{1b}(m^-_{\mathfrak{p}}, {\textsl{k}})$ \\ 
$- a_{11}(b^{+-}_{\mathfrak{p}}, m^+_{\mathfrak{p}}) 
a_{2b}({\bf tm}_{\mathfrak{p}}, {\textsl{k}}) 
- a_{12}(b^{+-}_{\mathfrak{p}}, {\bf tm}_{\mathfrak{p}}) 
a_{1b}(m^-_{\mathfrak{p}}, {\textsl{k}}) 
+ a_{11}(b^{+-}_{\mathfrak{p}}, m^+_{\mathfrak{p}}) 
a_{21}({\bf tm}_{\mathfrak{p}}, t_{\mathfrak{p}}) 
a_{1b}(t_{\mathfrak{p}}, {\textsl{k}})$ \\ 
$+ a_{32}^{+-}({\mathfrak{p}}, {\bf dc}_{\textsl{k}})$, \\ 
$\partial a_{b3}^{+-}({\textsl{k}}, {\mathfrak{p}})$ 
$= a_{23}^{+-}({\bf dc}_{\textsl{k}}, {\mathfrak{p}})$ 
$+ a_{b2}({\textsl{k}}, {\bf tb}^-_{\mathfrak{p}}) 
+ \mu a_{b2}({\textsl{k}}, {\bf tb}^+_{\mathfrak{p}}) 
- a_{b1}({\textsl{k}}, m^-_{\mathfrak{p}}) 
a_{12}(m^-_{\mathfrak{p}}, {\bf tb}^-_{\mathfrak{p}})$ \\ 
$- a_{b2}({\textsl{k}}, {\bf mb}^-_{\mathfrak{p}}) 
- \mu a_{b2}({\textsl{k}}, {\bf mb}^+_{\mathfrak{p}}) 
+ a_{b1}({\textsl{k}}, t_{\mathfrak{p}}) 
a_{12}(t_{\mathfrak{p}}, {\bf mb}^+_{\mathfrak{p}})$ \\ 
$- \mu^{-1} a_{b1}({\textsl{k}}, t_{\mathfrak{p}}) 
a_{12}(t_{\mathfrak{p}}, {\bf tb}^-_{\mathfrak{p}}) 
- a_{b1}({\textsl{k}}, t_{\mathfrak{p}}) 
a_{12}(t_{\mathfrak{p}}, {\bf tb}^+_{\mathfrak{p}}) 
+ \mu^{-1} a_{b1}({\textsl{k}}, m^-_{\mathfrak{p}}) 
a_{11}(m^-_{\mathfrak{p}}, t_{\mathfrak{p}}) 
a_{12}(t_{\mathfrak{p}}, {\bf tb}^-_{\mathfrak{p}})$ \\ 
$- a_{b2}({\textsl{k}}, {\bf tm}_{\mathfrak{p}}) 
a_{11}(m^+_{\mathfrak{p}}, b^{+-}_{\mathfrak{p}}) 
- a_{b1}({\textsl{k}}, m^-_{\mathfrak{p}}) 
a_{21}({\bf tm}_{\mathfrak{p}}, b^{+-}_{\mathfrak{p}}) 
+ \mu^{-1} a_{b1}({\textsl{k}}, t_{\mathfrak{p}}) 
a_{12}(t_{\mathfrak{p}}, {\bf tm}_{\mathfrak{p}}) 
a_{11}(m^+_{\mathfrak{p}}, b^{+-}_{\mathfrak{p}})$. \\ 
It is straightforward to see that the equation $\partial \circ \partial = 0$ 
holds on generators of $CR^{+-}(D)$. 

\subsection{$(CR^{++}(D), \partial)$} 
We define 
a differential $\partial$ 
on generators of $CR^{++}(D)$ 
involving triple points as follows. \\ 
$\partial a_{31}^{++}({\mathfrak{p}}, i)$ 
$= \mu a_{21}({\bf tb}^-_{\mathfrak{p}}, i) 
+ a_{21}({\bf tb}^+_{\mathfrak{p}}, i) 
- a_{21}({\bf tb}^+_{\mathfrak{p}}, m^-_{\mathfrak{p}}) 
a_{11}(m^-_{\mathfrak{p}}, i)$ \\ 
$- \mu a_{21}({\bf mb}^-_{\mathfrak{p}}, i) 
- a_{21}({\bf mb}^+_{\mathfrak{p}}, i) 
+ a_{21}({\bf mb}^+_{\mathfrak{p}}, t_{\mathfrak{p}}) 
a_{11}(t_{\mathfrak{p}}, i)$ \\ 
$- \mu a_{21}({\bf tb}^-_{\mathfrak{p}}, t_{\mathfrak{p}}) 
a_{11}(t_{\mathfrak{p}}, i) 
- a_{21}({\bf tb}^+_{\mathfrak{p}}, t_{\mathfrak{p}}) 
a_{11}(t_{\mathfrak{p}}, i) 
+ a_{21}({\bf tb}^+_{\mathfrak{p}}, t_{\mathfrak{p}}) 
a_{11}(t_{\mathfrak{p}}, m^-_{\mathfrak{p}}) 
a_{11}(m^-_{\mathfrak{p}}, i)$ \\ 
$+ \mu a_{21}({\bf mb}^-_{\mathfrak{p}}, m^-_{\mathfrak{p}}) 
a_{11}(m^-_{\mathfrak{p}}, i) 
+ a_{21}({\bf mb}^+_{\mathfrak{p}}, m^+_{\mathfrak{p}}) 
a_{11}(m^+_{\mathfrak{p}}, i) 
- a_{21}({\bf mb}^+_{\mathfrak{p}}, m^+_{\mathfrak{p}}) 
a_{11}(m^+_{\mathfrak{p}}, t_{\mathfrak{p}}) 
a_{11}(t_{\mathfrak{p}}, i)$ \\ 
$- a_{11}(b^{++}_{\mathfrak{p}}, m^+_{\mathfrak{p}}) 
a_{21}({\bf tm}_{\mathfrak{p}}, i) 
- a_{12}(b^{++}_{\mathfrak{p}}, {\bf tm}_{\mathfrak{p}}) 
a_{11}(m^-_{\mathfrak{p}}, i) 
+ a_{11}(b^{++}_{\mathfrak{p}}, m^+_{\mathfrak{p}}) 
a_{21}({\bf tm}_{\mathfrak{p}}, t_{\mathfrak{p}}) 
a_{11}(t_{\mathfrak{p}}, i)$, \\ 
$\partial a_{13}^{++}(i, {\mathfrak{p}})$ 
$= a_{12}(i, {\bf tb}^-_{\mathfrak{p}}) 
+ \mu a_{12}(i, {\bf tb}^+_{\mathfrak{p}}) 
- a_{11}(i, m^-_{\mathfrak{p}}) 
a_{12}(m^-_{\mathfrak{p}}, {\bf tb}^+_{\mathfrak{p}})$ \\ 
$- a_{12}(i, {\bf mb}^-_{\mathfrak{p}}) 
- \mu a_{12}(i, {\bf mb}^+_{\mathfrak{p}}) 
+ a_{11}(i, t_{\mathfrak{p}}) 
a_{12}(t_{\mathfrak{p}}, {\bf mb}^+_{\mathfrak{p}})$ \\ 
$- \mu^{-1} a_{11}(i, t_{\mathfrak{p}}) 
a_{12}(t_{\mathfrak{p}}, {\bf tb}^-_{\mathfrak{p}}) 
- a_{11}(i, t_{\mathfrak{p}}) 
a_{12}(t_{\mathfrak{p}}, {\bf tb}^+_{\mathfrak{p}}) 
+ \mu^{-1} a_{11}(i, m^-_{\mathfrak{p}}) 
a_{11}(m^-_{\mathfrak{p}}, t_{\mathfrak{p}}) 
a_{12}(t_{\mathfrak{p}}, {\bf tb}^+_{\mathfrak{p}})$ \\ 
$+ \mu^{-1} a_{11}(i, m^-_{\mathfrak{p}}) 
a_{12}(m^-_{\mathfrak{p}}, {\bf mb}^-_{\mathfrak{p}}) 
+ a_{11}(i, m^+_{\mathfrak{p}}) 
a_{12}(m^+_{\mathfrak{p}}, {\bf mb}^+_{\mathfrak{p}}) 
- \mu^{-1} a_{11}(i, t_{\mathfrak{p}}) 
a_{11}(t_{\mathfrak{p}}, m^+_{\mathfrak{p}}) 
a_{12}(m^+_{\mathfrak{p}}, {\bf mb}^+_{\mathfrak{p}})$ \\ 
$- a_{12}(i, {\bf tm}_{\mathfrak{p}}) 
a_{11}(m^+_{\mathfrak{p}}, b^{++}_{\mathfrak{p}}) 
- a_{11}(i, m^-_{\mathfrak{p}}) 
a_{21}({\bf tm}_{\mathfrak{p}}, b^{++}_{\mathfrak{p}}) 
+ \mu^{-1} a_{11}(i, t_{\mathfrak{p}}) 
a_{12}(t_{\mathfrak{p}}, {\bf tm}_{\mathfrak{p}}) 
a_{11}(m^+_{\mathfrak{p}}, b^{++}_{\mathfrak{p}})$, \\ 
$\partial a_{32}^{++}({\mathfrak{p}}, {\bf x})$ 
$= \mu a_{22}({\bf tb}^-_{\mathfrak{p}}, {\bf x}) 
+ a_{22}({\bf tb}^+_{\mathfrak{p}}, {\bf x}) 
- a_{21}({\bf tb}^+_{\mathfrak{p}}, m^-_{\mathfrak{p}}) 
a_{12}(m^-_{\mathfrak{p}}, {\bf x})$ \\ 
$- \mu a_{22}({\bf mb}^-_{\mathfrak{p}}, {\bf x}) 
- a_{22}({\bf mb}^+_{\mathfrak{p}}, {\bf x}) 
+ a_{21}({\bf mb}^+_{\mathfrak{p}}, t_{\mathfrak{p}}) 
a_{12}(t_{\mathfrak{p}}, {\bf x})$ \\ 
$- \mu a_{21}({\bf tb}^-_{\mathfrak{p}}, t_{\mathfrak{p}}) 
a_{12}(t_{\mathfrak{p}}, {\bf x}) 
- a_{21}({\bf tb}^+_{\mathfrak{p}}, t_{\mathfrak{p}}) 
a_{12}(t_{\mathfrak{p}}, {\bf x}) 
+ a_{21}({\bf tb}^+_{\mathfrak{p}}, t_{\mathfrak{p}}) 
a_{11}(t_{\mathfrak{p}}, m^-_{\mathfrak{p}}) 
a_{12}(m^-_{\mathfrak{p}}, {\bf x})$ \\ 
$+ \mu a_{21}({\bf mb}^-_{\mathfrak{p}}, m^-_{\mathfrak{p}}) 
a_{12}(m^-_{\mathfrak{p}}, {\bf x}) 
+ a_{21}({\bf mb}^+_{\mathfrak{p}}, m^+_{\mathfrak{p}}) 
a_{12}(m^+_{\mathfrak{p}}, {\bf x}) 
- a_{21}({\bf mb}^+_{\mathfrak{p}}, m^+_{\mathfrak{p}}) 
a_{11}(m^+_{\mathfrak{p}}, t_{\mathfrak{p}}) 
a_{12}(t_{\mathfrak{p}}, {\bf x})$ \\ 
$- a_{11}(b^{++}_{\mathfrak{p}}, m^+_{\mathfrak{p}}) 
a_{22}({\bf tm}_{\mathfrak{p}}, {\bf x}) 
- a_{12}(b^{++}_{\mathfrak{p}}, {\bf tm}_{\mathfrak{p}}) 
a_{12}(m^-_{\mathfrak{p}}, {\bf x}) 
+ a_{11}(b^{++}_{\mathfrak{p}}, m^+_{\mathfrak{p}}) 
a_{21}({\bf tm}_{\mathfrak{p}}, t_{\mathfrak{p}}) 
a_{12}(t_{\mathfrak{p}}, {\bf x})$ \\ 
$- a_{31}^{++}({\mathfrak{p}}, u^-_{\bf x}) 
- \mu a_{31}^{++}({\mathfrak{p}}, u^+_{\bf x}) 
+ a_{31}^{++}({\mathfrak{p}}, o_{\bf x}) 
a_{11}(o_{\bf x}, u^-_{\bf x})$, \\ 
$\partial a_{23}^{++}({\bf x}, {\mathfrak{p}})$ 
$= \mu a_{13}^{++}(u^-_{\bf x}, {\mathfrak{p}}) 
+ a_{13}^{++}(u^+_{\bf x}, {\mathfrak{p}}) 
- a_{11}(u^-_{\bf x}, o_{\bf x}) 
a_{13}^{++}(o_{\bf x}, {\mathfrak{p}})$ \\ 
$- a_{22}({\bf x}, {\bf tb}^-_{\mathfrak{p}}) 
- \mu a_{22}({\bf x}, {\bf tb}^+_{\mathfrak{p}}) 
+ a_{21}({\bf x}, m^-_{\mathfrak{p}}) 
a_{12}(m^-_{\mathfrak{p}}, {\bf tb}^+_{\mathfrak{p}})$ \\ 
$+ a_{22}({\bf x}, {\bf mb}^-_{\mathfrak{p}}) 
+ \mu a_{22}({\bf x}, {\bf mb}^+_{\mathfrak{p}}) 
- a_{21}({\bf x}, t_{\mathfrak{p}}) 
a_{12}(t_{\mathfrak{p}}, {\bf mb}^+_{\mathfrak{p}})$ \\ 
$+ \mu^{-1} a_{21}({\bf x}, t_{\mathfrak{p}}) 
a_{12}(t_{\mathfrak{p}}, {\bf tb}^-_{\mathfrak{p}}) 
+ a_{21}({\bf x}, t_{\mathfrak{p}}) 
a_{12}(t_{\mathfrak{p}}, {\bf tb}^+_{\mathfrak{p}}) 
- \mu^{-1} a_{21}({\bf x}, m^-_{\mathfrak{p}}) 
a_{11}(m^-_{\mathfrak{p}}, t_{\mathfrak{p}}) 
a_{12}(t_{\mathfrak{p}}, {\bf tb}^+_{\mathfrak{p}})$ \\ 
$- \mu^{-1} a_{21}({\bf x}, m^-_{\mathfrak{p}}) 
a_{12}(m^-_{\mathfrak{p}}, {\bf mb}^-_{\mathfrak{p}}) 
- a_{21}({\bf x}, m^+_{\mathfrak{p}}) 
a_{12}(m^+_{\mathfrak{p}}, {\bf mb}^+_{\mathfrak{p}}) 
+ \mu^{-1} a_{21}({\bf x}, t_{\mathfrak{p}}) 
a_{11}(t_{\mathfrak{p}}, m^+_{\mathfrak{p}}) 
a_{12}(m^+_{\mathfrak{p}}, {\bf mb}^+_{\mathfrak{p}})$ \\ 
$+ a_{22}({\bf x}, {\bf tm}_{\mathfrak{p}}) 
a_{11}(m^+_{\mathfrak{p}}, b^{++}_{\mathfrak{p}}) 
+ a_{21}({\bf x}, m^-_{\mathfrak{p}}) 
a_{21}({\bf tm}_{\mathfrak{p}}, b^{++}_{\mathfrak{p}}) 
- \mu^{-1} a_{21}({\bf x}, t_{\mathfrak{p}}) 
a_{12}(t_{\mathfrak{p}}, {\bf tm}_{\mathfrak{p}}) 
a_{11}(m^+_{\mathfrak{p}}, b^{++}_{\mathfrak{p}})$, \\ 
$\partial a_{33}^{++}({\mathfrak{p}}, {\mathfrak{q}})$ 
$= \mu a_{23}^{++}({\bf tb}^-_{\mathfrak{p}}, {\mathfrak{q}}) 
+ a_{23}^{++}({\bf tb}^+_{\mathfrak{p}}, {\mathfrak{q}}) 
- a_{21}({\bf tb}^+_{\mathfrak{p}}, m^-_{\mathfrak{p}}) 
a_{13}^{++}(m^-_{\mathfrak{p}}, {\mathfrak{q}})$ \\ 
$- \mu a_{23}^{++}({\bf mb}^-_{\mathfrak{p}}, {\mathfrak{q}}) 
- a_{23}^{++}({\bf mb}^+_{\mathfrak{p}}, {\mathfrak{q}}) 
+ a_{21}({\bf mb}^+_{\mathfrak{p}}, t_{\mathfrak{p}}) 
a_{13}^{++}(t_{\mathfrak{p}}, {\mathfrak{q}})$ \\ 
$- \mu a_{21}({\bf tb}^-_{\mathfrak{p}}, t_{\mathfrak{p}}) 
a_{13}^{++}(t_{\mathfrak{p}}, {\mathfrak{q}}) 
- a_{21}({\bf tb}^+_{\mathfrak{p}}, t_{\mathfrak{p}}) 
a_{13}^{++}(t_{\mathfrak{p}}, {\mathfrak{q}})$ 
$+ a_{21}({\bf tb}^+_{\mathfrak{p}}, t_{\mathfrak{p}}) 
a_{11}(t_{\mathfrak{p}}, m^-_{\mathfrak{p}}) 
a_{13}^{++}(m^-_{\mathfrak{p}}, {\mathfrak{q}})$ \\ 
$+ \mu a_{21}({\bf mb}^-_{\mathfrak{p}}, m^-_{\mathfrak{p}}) 
a_{13}^{++}(m^-_{\mathfrak{p}}, {\mathfrak{q}}) 
+ a_{21}({\bf mb}^+_{\mathfrak{p}}, m^+_{\mathfrak{p}}) 
a_{13}^{++}(m^+_{\mathfrak{p}}, {\mathfrak{q}})$ 
$- a_{21}({\bf mb}^+_{\mathfrak{p}}, m^+_{\mathfrak{p}}) 
a_{11}(m^+_{\mathfrak{p}}, t_{\mathfrak{p}}) 
a_{13}^{++}(t_{\mathfrak{p}}, {\mathfrak{q}})$ \\ 
$- a_{11}(b^{++}_{\mathfrak{p}}, m^+_{\mathfrak{p}}) 
a_{23}^{++}({\bf tm}_{\mathfrak{p}}, {\mathfrak{q}}) 
- a_{12}(b^{++}_{\mathfrak{p}}, {\bf tm}_{\mathfrak{p}}) 
a_{13}^{++}(m^-_{\mathfrak{p}}, {\mathfrak{q}})$ 
$+ a_{11}(b^{++}_{\mathfrak{p}}, m^+_{\mathfrak{p}}) 
a_{21}({\bf tm}_{\mathfrak{p}}, t_{\mathfrak{p}}) 
a_{13}^{++}(t_{\mathfrak{p}}, {\mathfrak{q}})$ \\ 
$+ a_{32}^{++}({\mathfrak{p}}, {\bf tb}^-_{\mathfrak{q}}) 
+ \mu a_{32}^{++}({\mathfrak{p}}, {\bf tb}^+_{\mathfrak{q}}) 
- a_{31}^{++}({\mathfrak{p}}, m^-_{\mathfrak{q}}) 
a_{12}(m^-_{\mathfrak{q}}, {\bf tb}^+_{\mathfrak{q}})$ \\ 
$- a_{32}^{++}({\mathfrak{p}}, {\bf mb}^-_{\mathfrak{q}}) 
- \mu a_{32}^{++}({\mathfrak{p}}, {\bf mb}^+_{\mathfrak{q}}) 
+ a_{31}^{++}({\mathfrak{p}}, t_{\mathfrak{q}}) 
a_{12}(t_{\mathfrak{q}}, {\bf mb}^+_{\mathfrak{q}})$ \\ 
$- \mu^{-1} a_{31}^{++}({\mathfrak{p}}, t_{\mathfrak{q}}) 
a_{12}(t_{\mathfrak{q}}, {\bf tb}^-_{\mathfrak{q}}) 
- a_{31}^{++}({\mathfrak{p}}, t_{\mathfrak{q}}) 
a_{12}(t_{\mathfrak{q}}, {\bf tb}^+_{\mathfrak{q}})$ 
$+ \mu^{-1} a_{31}^{++}({\mathfrak{p}}, m^-_{\mathfrak{q}}) 
a_{11}(m^-_{\mathfrak{q}}, t_{\mathfrak{q}}) 
a_{12}(t_{\mathfrak{q}}, {\bf tb}^+_{\mathfrak{q}})$ \\ 
$+ \mu^{-1} a_{31}^{++}({\mathfrak{p}}, m^-_{\mathfrak{q}}) 
a_{12}(m^-_{\mathfrak{q}}, {\bf mb}^-_{\mathfrak{q}}) 
+ a_{31}^{++}({\mathfrak{p}}, m^+_{\mathfrak{q}}) 
a_{12}(m^+_{\mathfrak{q}}, {\bf mb}^+_{\mathfrak{q}})$ 
$- \mu^{-1} a_{31}^{++}({\mathfrak{p}}, t_{\mathfrak{q}}) 
a_{11}(t_{\mathfrak{q}}, m^+_{\mathfrak{q}}) 
a_{12}(m^+_{\mathfrak{q}}, {\bf mb}^+_{\mathfrak{q}})$ \\ 
$- a_{32}^{++}({\mathfrak{p}}, {\bf tm}_{\mathfrak{q}}) 
a_{11}(m^+_{\mathfrak{q}}, b^{++}_{\mathfrak{q}}) 
- a_{31}^{++}({\mathfrak{p}}, m^-_{\mathfrak{q}}) 
a_{21}({\bf tm}_{\mathfrak{q}}, b^{++}_{\mathfrak{q}})$ 
$+ \mu^{-1} a_{31}^{++}({\mathfrak{p}}, t_{\mathfrak{q}}) 
a_{12}(t_{\mathfrak{q}}, {\bf tm}_{\mathfrak{q}}) 
a_{11}(m^+_{\mathfrak{q}}, b^{++}_{\mathfrak{q}})$, \\ 
$\partial a_{3}^{++}({\mathfrak{p}})$ 
$= a_{2}({\bf tb}^+_{\mathfrak{p}})$ 
$- a_{2}({\bf tb}^-_{\mathfrak{p}})$  
$- a_{2}({\bf mb}^+_{\mathfrak{p}})$  
$+ \mu a_{22}({\bf tb}^-_{\mathfrak{p}}, {\bf mb}^+_{\mathfrak{p}})$ 
$- a_{21}({\bf tb}^-_{\mathfrak{p}}, m^+_{\mathfrak{p}}) 
a_{12}(m^+_{\mathfrak{p}}, {\bf mb}^+_{\mathfrak{p}})$ \\ 
$- \mu (a_{21}({\bf tb}^-_{\mathfrak{p}}, m^+_{\mathfrak{p}})$ 
$+ a_{12}(b^{--}_{\mathfrak{p}}, {\bf tm}_{\mathfrak{p}})$ 
$+ a_{21}({\bf mb}^+_{\mathfrak{p}}, m^+_{\mathfrak{p}})$ 
$- \mu^{-1} a_{11}(b^{--}_{\mathfrak{p}}, t_{\mathfrak{p}}) 
a_{12}(t_{\mathfrak{p}}, {\bf tm}_{\mathfrak{p}})) 
a_{21}({\bf tm}_{\mathfrak{p}}, b^{++}_{\mathfrak{p}})$ \\ 
$+ (a_{11}(b^{++}_{\mathfrak{p}}, m^+_{\mathfrak{p}}) 
a_{22}({\bf tm}_{\mathfrak{p}}, {\bf tm}_{\mathfrak{p}}) 
- a_{11}(b^{++}_{\mathfrak{p}}, m^+_{\mathfrak{p}}) 
a_{2}({\bf tm}_{\mathfrak{p}}) 
- \mu a_{22}({\bf tb}^-_{\mathfrak{p}}, {\bf tm}_{\mathfrak{p}}) 
+ a_{22}({\bf mb}^+_{\mathfrak{p}}, {\bf tm}_{\mathfrak{p}})$ \\ 
$- a_{21}({\bf mb}^+_{\mathfrak{p}}, m^+_{\mathfrak{p}}) 
a_{12}(m^+_{\mathfrak{p}}, {\bf tm}_{\mathfrak{p}}))$ 
$(a_{11}(m^-_{\mathfrak{p}}, b^{++}_{\mathfrak{p}}) 
- \mu^{-1} a_{11}(m^-_{\mathfrak{p}}, t_{\mathfrak{p}}) 
a_{11}(t_{\mathfrak{p}}, b^{++}_{\mathfrak{p}}))$ \\ 
$+ a_{31}^{++}({\mathfrak{p}}, t_{\mathfrak{p}}) 
a_{11}(t_{\mathfrak{p}}, b^{++}_{\mathfrak{p}})$ 
$+ a_{31}^{++}({\mathfrak{p}}, m^-_{\mathfrak{p}})$ 
$(a_{11}(m^-_{\mathfrak{p}}, b^{++}_{\mathfrak{p}}) 
- \mu^{-1} a_{11}(m^-_{\mathfrak{p}}, t_{\mathfrak{p}}) 
a_{11}(t_{\mathfrak{p}}, b^{++}_{\mathfrak{p}}))$ \\ 
$- \mu a_{31}^{++}({\mathfrak{p}}, b^{++}_{\mathfrak{p}})$ 
$+ \mu a_{13}^{++}(b^{--}_{\mathfrak{p}}, {\mathfrak{p}})$ 
$- \mu a_{22}({\bf mb}^-_{\mathfrak{p}}, {\bf tb}^+_{\mathfrak{p}})$ 
$+ a_{21}({\bf mb}^-_{\mathfrak{p}}, t_{\mathfrak{p}}) 
a_{12}(t_{\mathfrak{p}}, {\bf tb}^+_{\mathfrak{p}})$ 
$+ a_{2}({\bf mb}^-_{\mathfrak{p}})$, \\ 
$\partial a_{3b}^{++}({\mathfrak{p}}, {\textsl{k}})$ 
$= \mu a_{2b}({\bf tb}^-_{\mathfrak{p}}, {\textsl{k}}) 
+ a_{2b}({\bf tb}^+_{\mathfrak{p}}, {\textsl{k}}) 
- a_{21}({\bf tb}^+_{\mathfrak{p}}, m^-_{\mathfrak{p}}) 
a_{1b}(m^-_{\mathfrak{p}}, {\textsl{k}})$ \\ 
$- \mu a_{2b}({\bf mb}^-_{\mathfrak{p}}, {\textsl{k}}) 
- a_{2b}({\bf mb}^+_{\mathfrak{p}}, {\textsl{k}}) 
+ a_{21}({\bf mb}^+_{\mathfrak{p}}, t_{\mathfrak{p}}) 
a_{1b}(t_{\mathfrak{p}}, {\textsl{k}})$ \\ 
$- \mu a_{21}({\bf tb}^-_{\mathfrak{p}}, t_{\mathfrak{p}}) 
a_{1b}(t_{\mathfrak{p}}, {\textsl{k}}) 
- a_{21}({\bf tb}^+_{\mathfrak{p}}, t_{\mathfrak{p}}) 
a_{1b}(t_{\mathfrak{p}}, {\textsl{k}}) 
+ a_{21}({\bf tb}^+_{\mathfrak{p}}, t_{\mathfrak{p}}) 
a_{11}(t_{\mathfrak{p}}, m^-_{\mathfrak{p}}) 
a_{1b}(m^-_{\mathfrak{p}}, {\textsl{k}})$ \\ 
$+ \mu a_{21}({\bf mb}^-_{\mathfrak{p}}, m^-_{\mathfrak{p}}) 
a_{1b}(m^-_{\mathfrak{p}}, {\textsl{k}}) 
+ a_{21}({\bf mb}^+_{\mathfrak{p}}, m^+_{\mathfrak{p}}) 
a_{1b}(m^+_{\mathfrak{p}}, {\textsl{k}}) 
- a_{21}({\bf mb}^+_{\mathfrak{p}}, m^+_{\mathfrak{p}}) 
a_{11}(m^+_{\mathfrak{p}}, t_{\mathfrak{p}}) 
a_{1b}(t_{\mathfrak{p}}, {\textsl{k}})$ \\ 
$- a_{11}(b^{++}_{\mathfrak{p}}, m^+_{\mathfrak{p}}) 
a_{2b}({\bf tm}_{\mathfrak{p}}, {\textsl{k}}) 
- a_{12}(b^{++}_{\mathfrak{p}}, {\bf tm}_{\mathfrak{p}}) 
a_{1b}(m^-_{\mathfrak{p}}, {\textsl{k}}) 
+ a_{11}(b^{++}_{\mathfrak{p}}, m^+_{\mathfrak{p}}) 
a_{21}({\bf tm}_{\mathfrak{p}}, t_{\mathfrak{p}}) 
a_{1b}(t_{\mathfrak{p}}, {\textsl{k}})$ \\ 
$+ a_{32}^{++}({\mathfrak{p}}, {\bf dc}_{\textsl{k}})$, \\ 
$\partial a_{b3}^{++}({\textsl{k}}, {\mathfrak{p}})$ 
$= a_{23}^{++}({\bf dc}_{\textsl{k}}, {\mathfrak{p}})$ 
$+ a_{b2}({\textsl{k}}, {\bf tb}^-_{\mathfrak{p}}) 
+ \mu a_{b2}({\textsl{k}}, {\bf tb}^+_{\mathfrak{p}}) 
- a_{b1}({\textsl{k}}, m^-_{\mathfrak{p}}) 
a_{12}(m^-_{\mathfrak{p}}, {\bf tb}^+_{\mathfrak{p}})$ \\ 
$- a_{b2}({\textsl{k}}, {\bf mb}^-_{\mathfrak{p}}) 
- \mu a_{b2}({\textsl{k}}, {\bf mb}^+_{\mathfrak{p}}) 
+ a_{b1}({\textsl{k}}, t_{\mathfrak{p}}) 
a_{12}(t_{\mathfrak{p}}, {\bf mb}^+_{\mathfrak{p}})$ \\ 
$- \mu^{-1} a_{b1}({\textsl{k}}, t_{\mathfrak{p}}) 
a_{12}(t_{\mathfrak{p}}, {\bf tb}^-_{\mathfrak{p}}) 
- a_{b1}({\textsl{k}}, t_{\mathfrak{p}}) 
a_{12}(t_{\mathfrak{p}}, {\bf tb}^+_{\mathfrak{p}}) 
+ \mu^{-1} a_{b1}({\textsl{k}}, m^-_{\mathfrak{p}}) 
a_{11}(m^-_{\mathfrak{p}}, t_{\mathfrak{p}}) 
a_{12}(t_{\mathfrak{p}}, {\bf tb}^+_{\mathfrak{p}})$ \\ 
$+ \mu^{-1} a_{b1}({\textsl{k}}, m^-_{\mathfrak{p}}) 
a_{12}(m^-_{\mathfrak{p}}, {\bf mb}^-_{\mathfrak{p}}) 
+ a_{b1}({\textsl{k}}, m^+_{\mathfrak{p}}) 
a_{12}(m^+_{\mathfrak{p}}, {\bf mb}^+_{\mathfrak{p}}) 
- \mu^{-1} a_{b1}({\textsl{k}}, t_{\mathfrak{p}}) 
a_{11}(t_{\mathfrak{p}}, m^+_{\mathfrak{p}}) 
a_{12}(m^+_{\mathfrak{p}}, {\bf mb}^+_{\mathfrak{p}})$ \\ 
$- a_{b2}({\textsl{k}}, {\bf tm}_{\mathfrak{p}}) 
a_{11}(m^+_{\mathfrak{p}}, b^{++}_{\mathfrak{p}}) 
- a_{b1}({\textsl{k}}, m^-_{\mathfrak{p}}) 
a_{21}({\bf tm}_{\mathfrak{p}}, b^{++}_{\mathfrak{p}}) 
+ \mu^{-1} a_{b1}({\textsl{k}}, t_{\mathfrak{p}}) 
a_{12}(t_{\mathfrak{p}}, {\bf tm}_{\mathfrak{p}}) 
a_{11}(m^+_{\mathfrak{p}}, b^{++}_{\mathfrak{p}})$. \\ 
It is straightforward to see that the equation $\partial \circ \partial = 0$ 
holds on generators of $CR^{++}(D)$.

\section{Examples}

\begin{figure}
\begin{center}
\includegraphics{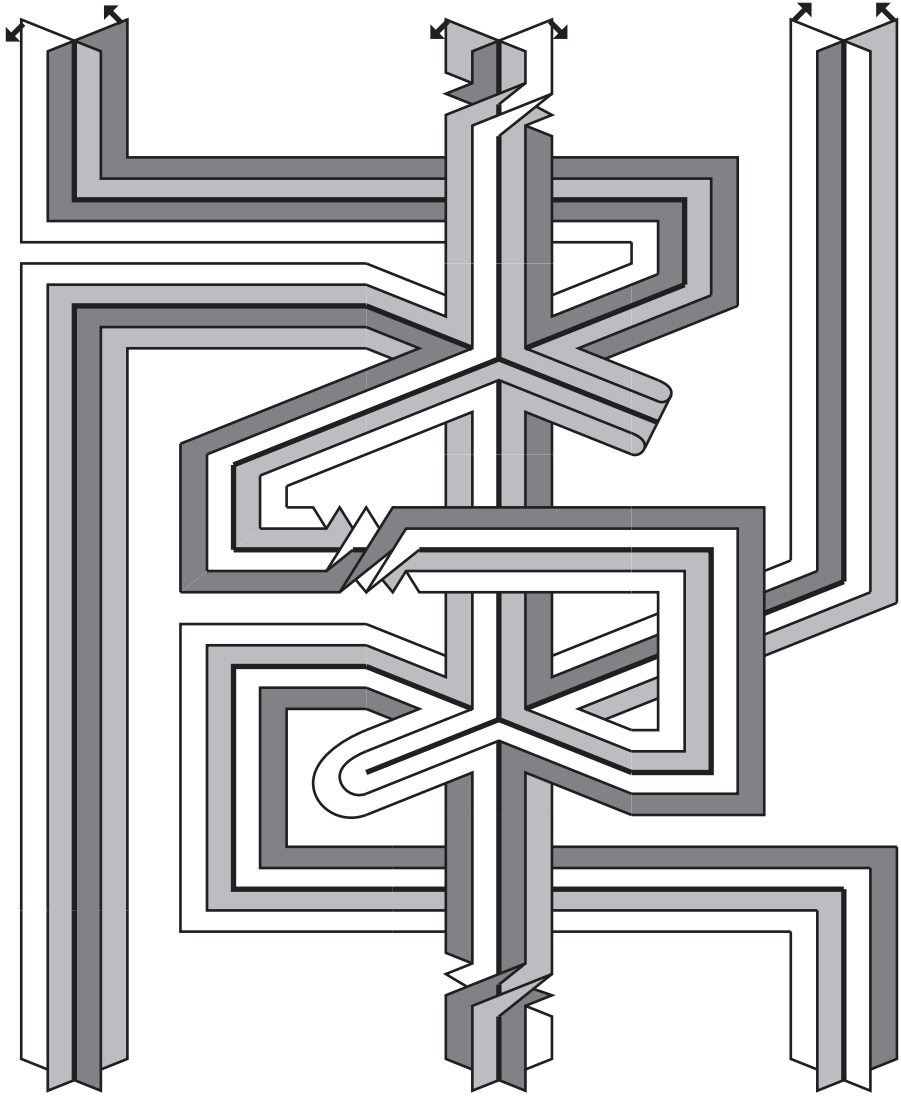}
\caption{one piece of $D^2(2, 3)$} 
\label{2twist0}
\end{center}
\end{figure}

Let $T(2, 3)$ denote a torus knot of type $(2, 3)$ in ${\mathbb{R}}^3$. 
Let $T^0(2, 3)$ denote a 2-sphere in ${\mathbb{R}}^4$ 
that is obtained from $T(2, 3)$ 
by the spinning construction, 
introduced by Artin \cite{artin}. 
Let $T^n(2, 3)$ denote 
a 2-sphere 
in ${\mathbb{R}}^4$ 
that is obtained from $T(2, 3)$ 
by the twist-spinning construction, introduced by Zeeman \cite{zeeman}, 
where $n \in {\mathbb{Z}}$. 
Satoh and Shima \cite{satoh-shima} 
exhibit an explicit construction of a diagram $D^n(2, 3)$ of $T^n(2, 3)$. 
In particular, 
we obtain a diagram $D^2(2, 3)$ by cyclically concatenating two copies 
of the diagram illustrated in Figure \ref{2twist0}. 
 
\vskip 6pt

\noindent 
{\it Proof of Theorem \ref{2-twist}}. 
The diagram $D^2(2, 3)$ 
lets us describe the differential of 
$(CR^{\varepsilon \delta}(D^2(2, 3)), \partial)$ 
on generators. 
We recall a characteristic algebra ${\mathcal{C}}^{\varepsilon \delta}(D)$ 
that is derived from a differential graded algebra 
$(CR^{\varepsilon \delta}(D), \partial)$, 
where $D$ denotes a diagram of a surface-knot in ${\mathbb{R}}^4$. 
A direct calculation shows that 
the numbers of algebra maps 
${\mathcal{C}}^{\varepsilon \delta}(D^2(2, 3)) 
\otimes {\mathbb{Z}}/3{\mathbb{Z}} 
\to {\mathbb{Z}}/3{\mathbb{Z}}$ 
are $1 + 3^9, 1 + 3^8, 1 + 3^8, 1 + 3^{10}$, 
up to powers of $3$, 
when $(\varepsilon, \delta) = (-, -), (-, +), (+, -), (+, +)$, respectively. 
This shows that 
${\mathcal{C}}^{--}(D^2(2, 3))$ is not equivalent to 
${\mathcal{C}}^{\varepsilon \delta}(D^2(2, 3))$ 
when $(\varepsilon, \delta) = (-, +), (+, -), (+, +)$. 
It follows that $(CR^{--}(D^2(2, 3)), \partial)$ is not 
stably tame isomorphic to $(CR^{\varepsilon \delta}(D^2(2, 3)), \partial)$ 
when $(\varepsilon, \delta) = (-, +), (+, -), (+, +)$. 
Similarly 
$(CR^{++} (D^2(2, 3)), \partial)$ is not stably tame isomorphic to 
$(CR^{-+} (D^2(2, 3)), \partial)$, $(CR^{+-} (D^2(2, 3)), \partial)$. 
\hfill $\qed$

\vskip 6pt

\noindent 
{\it Proof of Theorem \ref{plus-minus}}. 
We notice that 
$T^0(2, 3)$ has a diagram $D^0(2, 3)$ without triple points. 
Therefore 
we can construct $(CR^{\varepsilon \delta}(D^0(2, 3)), \partial)$ 
whose generators do not involve triple points. 
It follows that 
$(CR^{--}(D^0(2, 3)), \partial)$ is isomorphic to 
each of $(CR^{-+}(D^0(2, 3)), \partial)$, 
$(CR^{+-}(D^0(2, 3)), \partial)$ and 
$(CR^{++}(D^0(2, 3)), \partial)$. 
Theorems \ref{main-theorem} and \ref{2-twist} show that 
$T^2(2, 3)$ is not ambient isotopic to $T^0(2, 3)$ 
in ${\mathbb{R}}^4$. 
\hfill $\qed$

\end{document}